\documentclass[english,11pt, reqno, oneside]{amsart}
\usepackage{mathbbol}
\usepackage{amsmath}
\usepackage{amssymb}
\usepackage{amsthm}
\usepackage{amsfonts}
\usepackage{mathtools}
\usepackage{bbm,bm}
\usepackage{stmaryrd}
\usepackage[usenames,dvipsnames]{xcolor}
\usepackage[colorlinks=true,linkcolor=blue!95!black, citecolor = green!55!black,bookmarksdepth=3]{hyperref}

\usepackage{adjustbox}

\DeclareSymbolFontAlphabet{\mathbb}{AMSb}%

\usepackage[margin=1in]{geometry}

\usepackage{etoolbox}
\usepackage{enumitem}
\setlist{itemsep=4pt}
\usepackage{float}
\usepackage[mathscr]{euscript}
\usepackage{graphicx}
\usepackage{subcaption}
\usepackage{tikz,datatool}
\usepackage{xspace}

\DeclareSymbolFont{eulerletters}{U}{zeur}{m}{n}
\DeclareMathSymbol{\eulL}{\mathord}{eulerletters}{`L}

\usepackage{microtype}

\allowdisplaybreaks

\DeclareMathOperator*{\argmax}{argmax\,}

\newcommand{\hair}{\ifmmode\mskip1mu\else\kern0.08em\fi}
\renewcommand{\P}{\mathbb{P}}

\newcommand{\E}{\mathbb{E}}

\newcommand{\R}{\mathbb{R}}

\newcommand{\N}{\mathbb{N}}
\newcommand{\Z}{\mathbb{Z}}
\newcommand{\Q}{\mathbb{Q}}

\newcommand{\one}{\mathbbm{1}}
\newcommand{\intint}[1]{\llbracket #1 \rrbracket}

\newcommand{\mrm}{\mathrm}
\newcommand{\msf}{\mathsf}
\newcommand{\mc}{\mathcal}
\newcommand{\mf}{\mathfrak}

\newcommand{\cL}{\mathcal L}

\renewcommand{\S}{\mathcal S}

\newcommand{\h}{\mathfrak h}
\newcommand{\g}{\mathfrak g}
\newcommand{\f}{\mathfrak f}

\newcommand{\iid}{i.i.d.\ }

\newcommand{\midd}{\ \Big|\  }

\newcommand{\floor}[1]{\lfloor #1 \rfloor}

\newcommand{\hssv}{h^{\mrm{S6V}}}

\newcommand{\rrparen}{\Rparen}

\renewcommand{\th}{\textsuperscript{th}\xspace}

\colorlet{lightgray}{gray!50!white}

\DeclareMathAlphabet{\mathdutchcal}{U}{dutchcal}{m}{n}

\usetikzlibrary{calc}
\usetikzlibrary{math}
\usetikzlibrary{decorations.markings, decorations.pathreplacing, decorations.pathmorphing,shapes.misc}

\usetikzlibrary{fpu}
\makeatletter
\tikzset{use fpu reciprocal/.code={%
\def\pgfmathreciprocal@##1{%
    \begingroup
    \pgfkeys{/pgf/fpu=true,/pgf/fpu/output format=fixed}%
    \pgfmathparse{1/##1}%
    \pgfmath@smuggleone\pgfmathresult
    \endgroup
}}}%
\makeatother

\newcommand{\bb}[8]{
\xdef\y{#4}
\pgfmathsetseed{#7} 
\xdef\lst{}
\foreach \x [count=\n] in {1,...,#1}
{
    \pgfmathparse{\y + rand*#3} 
    \xdef\y{\pgfmathresult}
    \xdef\finaly{\pgfmathresult}
    \xdef\finaln{\n}
    \ifnum\n=1\relax
    \xdef\lst{{\x/\y}}
    \else
    \xdef\lst{\lst,{\x/\y}}
    \fi
}

\xdef\newlist{}
\xdef\trim{1}
\foreach \x/\y [count=\n] in \lst
{
\ifnum\n=\trim\relax
\xdef\newlist{{\x/\y}}
\fi
\ifnum\n>\trim\relax
\xdef\newlist{\newlist,{\x/\y}}
\fi
}

\draw[#6] (#8,#4)
\foreach \x/\y [count=\n] in \newlist
{
\pgfextra{\pgfmathsetmacro{\newy}{\y-((\n+\trim)/\finaln)*\finaly+(#5)*(\n+\trim)/\finaln}}
  -- (\x*#2+#8,\newy)
};
}

\newcommand{\bbpara}[8]{
\xdef\y{#4}
\pgfmathsetseed{#7} 
\xdef\lst{}
\foreach \x [count=\n] in {1,...,#1}
{
    \pgfmathparse{\y + rand*#3} 
    \xdef\y{\pgfmathresult}
    \xdef\finaly{\pgfmathresult}
    \xdef\finaln{\n}
    \ifnum\n=1\relax
    \xdef\lst{{\x/\y}}
    \else
    \xdef\lst{\lst,{\x/\y}}
    \fi
}

\xdef\newlist{}
\xdef\trim{1}
\foreach \x/\y [count=\n] in \lst
{
\ifnum\n=\trim\relax
\xdef\newlist{{\x/\y}}
\fi
\ifnum\n>\trim\relax
\xdef\newlist{\newlist,{\x/\y}}
\fi
}

\draw[#6] (#8,#4)
\foreach \x/\y [count=\n] in \newlist
{
\pgfextra{\pgfmathsetmacro{\newy}{\y-((\n+\trim)/\finaln)*\finaly+(#5)*(\n+\trim)/\finaln - ((\x-1)*#2)*((\x-#1)*#2)/(#1*#2)}}
  -- (\x*#2+#8,\newy)
};

}

\newcommand{\bbmelon}[9]{
\xdef\y{#4}
\pgfmathsetseed{#7} 
\xdef\lst{}
\foreach \x [count=\n] in {1,...,#1}
{
    \pgfmathparse{\y + rand*#3} 
    \xdef\y{\pgfmathresult}
    \xdef\finaly{\pgfmathresult}
    \xdef\finaln{\n}
    \ifnum\n=1\relax
    \xdef\lst{{\x/\y}}
    \else
    \xdef\lst{\lst,{\x/\y}}
    \fi
}

\xdef\newlist{}
\xdef\trim{1}
\foreach \x/\y [count=\n] in \lst
{
\ifnum\n=\trim\relax
\xdef\newlist{{\x/\y}}
\fi
\ifnum\n>\trim\relax
\xdef\newlist{\newlist,{\x/\y}}
\fi
}

\draw[#6] (#8,#4)
\foreach \x/\y [count=\n] in \newlist
{
\pgfextra{\pgfmathsetmacro{\sqrtx}{sqrt(\x)}}
\pgfextra{\pgfmathsetmacro{\newy}{\y-((\n+\trim)/\finaln)*\finaly+(#5)*(\n+\trim)/\finaln + #9*\sqrtx -0.007*\x}}
  -- (\x*#2+#8,\newy)
};
}

\newtheoremstyle{indented}{4pt}{4pt}{}{}{\bfseries}{.}{.5em}{}

\newtheorem{theorem}{Theorem}[section]
\newtheorem*{theorem*}{Theorem}
\newtheorem*{proposition*}{Proposition}
\newtheorem{proposition}[theorem]{Proposition}
\newtheorem*{corollary*}{Corollary}
\newtheorem{corollary}[theorem]{Corollary}
\newtheorem{lemma}[theorem]{Lemma}

\theoremstyle{definition}
\newtheorem{definition}[theorem]{Definition}

\newtheorem{remark}[theorem]{Remark}

\numberwithin{equation}{section}

\theoremstyle{indented}
\newtheorem{assumption}[theorem]{Assumption}

\title[KPZ fixed point convergence of the ASEP and stochastic six-vertex models]{KPZ fixed point convergence of the\\ ASEP and stochastic six-vertex models}

 \usepackage{amsaddr}
\author{Amol Aggarwal$^{1,2}$, Ivan Corwin$^1$, Milind Hegde$^1$}
\address{$^1$ Department of Mathematics, Columbia University, New York, NY USA. \newline $^2$ Clay Mathematics Institute, Denver, Colorado USA.}
\email{amolagga@gmail.com, ivan.corwin@gmail.com, mh4259@columbia.edu}

\newcommand{\intray}[1]{\llbracket #1,\infty\rrparen}

\begin{document}
\lineskiplimit=2pt

\begin{abstract}
We consider the stochastic six-vertex (S6V) model and asymmetric simple exclusion process (ASEP) under general initial conditions which are bounded below lines of arbitrary slope at $\pm\infty$. We show under Kardar-Parisi-Zhang (KPZ) scaling of time, space, and fluctuations that the height functions of these models converge to the KPZ fixed point.  Previously, our results were known in the case of ASEP (for a particular direction in the rarefaction fan) via a comparison approach \cite{quastel2022convergence}.
\end{abstract}

\maketitle

\setcounter{tocdepth}{1}
\tableofcontents

\section{Introduction}

Many models in the Kardar-Parisi-Zhang (KPZ) universality class can be thought of as dynamical systems driven by variants of space-time white noise. These include last passage percolation, polymer models, and the KPZ equation itself, as well as interacting particle systems and stochastic vertex models. The dynamical system perspective is intrinsically interesting but also produces several systems via projection procedures, such as ``colored'' models. A natural question in this area is to determine the scaling limit of the entire random dynamical system, i.e., of the joint scaling limit of the evolution across space, time, and initial conditions under the KPZ scaling exponents; more precisely, for a scaling parameter $\varepsilon>0$,  on scaling time as $\varepsilon^{-1}$, space as $\varepsilon^{-2/3}$, and the solutions' fluctuations as $\varepsilon^{-1/3}$.

For the models of last passage percolation, polymers, and the KPZ equation, the dynamical system is of a variational or Feynman-Kac form, in that it can be defined by maximizing or summing weights of paths through the space-time white noise environment. 
As a result, the evolution jointly at general times and under (multiple) general initial conditions is determined by the joint scaling limits of the solutions for delta initial condition (i.e., point to point passage times or partition functions), as the location of the delta varies. Since these solutions determine the general solutions (even in the prelimit), they can be regarded as fundamental solutions. Their putative universal scaling limit is the directed landscape $\cL:\{(x,s;y,t) : x,y,s,t\in\R, s<t\}\to\R$, where $\cL(x,s;y,t)$ is the scaling limit of the delta solution started at renormalized time $s$ from position $x$ and viewed at renormalized time $t$ at position $y$, and was identified and constructed for particular solvable models in \cite{dauvergne2018directed} and subsequent work such as \cite{dauvergne2021scaling,wu2023kpz}. Given the directed landscape, the scaling limit of the random dynamical system for such exactly solvable variational or Feynman-Kac type models  is
\begin{align}\label{e.scaling limit of dynamical system}
(\h_0, s, x,t)\mapsto \sup_{y\in\R}\Bigl(\h_0(y) + \cL(y,s;x,t)\Bigr),
\end{align}
where $\h_0 :\R\to\R\cup\{-\infty\}$ is a renormalized initial condition, $s\in\R$ represents the time from which the system evolves the initial condition, and $x\in\R$ and $t>s$ are the location and time at which the evolution is observed. 
The mapping that outputs the process $(x,t)\mapsto \sup_{y\in\R}(\h_0(y) + \cL(y,s;x,t))$ for a given initial state $\h_0:\R\to\R\cup\{-\infty\}$ is known as the \emph{KPZ fixed point}. It was first constructed in \cite{matetski2016kpz} by using exact determinantal formulas for the totally asymmetric simple exclusion process (TASEP), and it was later shown to be equivalent to the present definition in \cite{nica2020one}.

For models not of a variational or Feynman-Kac type, such as interacting particle systems and stochastic vertex models, the existence of a limiting fundamental solution in the sense that their limit exists and determines the limit of the dynamical systems as a whole as in \eqref{e.scaling limit of dynamical system} is not immediate. In fact, for initial conditions which are perturbations of a compact set of stationary initial conditions, a concise argument given here (Proposition~\ref{p.KPZ FP convergence via approximation bernoullis}) provides it (a similar argument was recently and independently given in \cite[Lemma 3.2]{dauvergne2024directed}). However, expanding to the full class of initial data (when the initial condition can grow linearly at $\pm\infty$) for which convergence has been shown in previous models such as TASEP is less clear. 

In our previous paper \cite{aggarwal2024scaling}, we determined the solutions which converge to the directed landscape, namely coupled step initial conditions, and showed that convergence for two models, the asymmetric simple exclusion process (ASEP) and the stochastic six-vertex model (S6V). That framework, which goes via structures known as colored Gibbsian line ensembles \cite{aggarwalborodin} and characterization results for the parabolic Airy line ensemble \cite{aggarwal2023strong,dimitrov2021characterization}, was developed for ASEP and S6V but should apply broadly to all integrable models satisfying the Yang-Baxter equation or their degenerations.

Here we present an approach for the same models of ASEP and S6V to show that the convergence of those special solutions (namely, under coupled step initial condition) implies convergence of the full dynamical solution (namely, under arbitrary coupled initial condition with linear growth at $\pm\infty$); in this way, these special solutions can be regarded as approximate fundamental solutions, recovering the full dynamical system in the scaling limit. 

Our arguments use the recently established convergence to the Airy sheet and directed landscape of the height functions of these models when started from coupled step initial conditions, along with monotonicity properties and previously known one-point convergence statements for stationary initial conditions of these models.  
As such, the proof gives a blueprint towards proving KPZ fixed point convergence for any model where such properties are known.

\subsection{Relation to recent work and main result} 

As mentioned, in models like ASEP and S6V which are not of a variational or Feynman-Kac type, the general initial condition cannot be expressed in an exact way in terms of the fundamental solution at the prelimiting level. In \cite[Appendix D]{aggarwal2024scaling}, we presented an argument showing that directed landscape convergence and distributional convergence to the KPZ fixed point from a general initial condition (i.e., for each single initial condition) together implies the scaling limit of the random dynamical system as a whole (i.e., to establish joint convergence of multiple general initial conditions and times); \cite{aggarwal2024scaling} gave such an argument for ASEP, combining the directed landscape convergence from the same paper with marginal KPZ fixed point convergence proved in~\cite{quastel2022convergence}. 

For S6V, while directed landscape convergence was proven in \cite{aggarwal2024scaling}, KPZ fixed point convergence was not known. A direct argument using directed landscape convergence, monotonicity properties, and known one-point convergence for stationary initial data establishes KPZ fixed point convergence for general initial conditions that lies below a line of fixed slope. Extending from compact perturbations to the most general class of initial conditions considered in the literature previously (namely those that grow at most linearly at $\pm\infty$) requires different ideas, which is the main purpose of our paper. An informal version of our main result is as follows.

\begin{theorem*}[Informal version of Theorem~\ref{t.main kpz fixed point convergence}]
Fix any parameters for ASEP or S6V corresponding to their rarefaction fans. Let $\smash{h_0^{(i)}}:\Z\to\Z$ be initial height functions for the models which, on KPZ rescaling, converge in the topology of local convergence in the space UC to upper semi-continuous functions growing at most linearly. Starting the $i$\th initial condition at rescaled time $s_i$ and evolving them jointly using the same randomness (i.e., the basic coupling), the rescaled height functions converge jointly to \eqref{e.scaling limit of dynamical system}, i.e, KPZ fixed points coupled via the directed landscape.
\end{theorem*}

Our main result shows in particular that, for any given initial condition with any non-trivial macroscopic density, the height function evolution is that of the KPZ fixed point.
For ASEP around density $\frac{1}{2}$, this is the earlier mentioned result of \cite{quastel2022convergence}, where it was proven by comparison to TASEP; the class of initial conditions we cover is the same as that originally proved for TASEP \cite{matetski2016kpz}. \cite{quastel2022convergence}'s approach applies to more general exclusion processes and presumably can be implemented for more general densities; but this comparison method is unclear for S6V, and hence, even our marginal KPZ fixed point convergence result is new in that context.

The recent work \cite{dauvergne2024directed} provides a framework to show directed landscape convergence assuming KPZ fixed point convergence from narrow wedge initial conditions. Their method  applies (in conjunction with \cite{quastel2022convergence}) to models such as non-nearest neighbor asymmetric exclusion processes, while our approach includes the stochastic six vertex model (for which KPZ fixed point convergence was not known, but directed landscape convergence is proven in \cite{aggarwal2024scaling}).

\subsection{Brief discussion of approach}\label{s.proof ideas}
Our approach starts with the observation that a large class of initial conditions can be approximated with positive probability by a choice of random initial condition that is more accessible to analysis; here, we consider stationary initial condition, which is the Bernoulli product measure on the particle configuration. 

First, we use a one-sided inequality in the prelimit (relating prelimiting versions of both sides of \eqref{e.scaling limit of dynamical system}) coming from height monotonicity (Lemmas~\ref{l.asep height monotonicity} and \ref{l.s6v approximate height monotonicity}) and one-point convergence of stationary initial data \cite{aggarwal2018current} to obtain KPZ fixed point convergence for stationary initial data (Proposition~\ref{p.bernoulli convergence}), as well as to upgrade marginal convergence to the KPZ fixed point of multiple initial conditions to their joint convergence (Lemma~\ref{l.joint kpz fixed point convergence}). Next, we consider initial conditions that converge (after rescaling) to a continuous function on a compact interval and equal a random walk with fixed drift outside that interval. By using that the stationary initial condition's evolution packages together the evolution from a large collection of deterministic initial conditions, we obtain the convergence of such an initial condition's evolution to the KPZ fixed point (Proposition~\ref{p.KPZ FP convergence via approximation bernoullis}) by looking at appropriate parts of the probability space where the stationary initial condition approximates the given one.

The main part of our argument is upgrading from such initial conditions to ones where the random walks on either side of the compact interval have different drifts ($+\lambda$ on the right and $-\lambda$ on the left), which will be needed to upper bound arbitrary initial conditions which are bounded above by $\lambda(1+|x|)$ for some $\lambda>0$; the statement is Proposition~\ref{p.KPZ FP convergence for V shape bernoulli}.
Essentially, we compare the evolution of such an initial condition to one which has a random walk with the same drift $\lambda$ on both sides (for which we have already shown the convergence) and show that their discrepancy disappears in the rescaling. The discrepancy can be encoded as lower color particles in colored versions of these models. We control the motion of such particles and show that, with high probability, they do not go far enough to appear in any interval which is seen in the convergence to the KPZ fixed point. 
A version of such a discrepancy bound was proved in \cite{quastel2022convergence} for ASEP as a direct consequence of a coupling provided in \cite[Lemma B.1]{quastel2022convergence}. However, a coupling of that strength does not hold for the S6V, so we work with a more general coupling instead \cite[Corollary 3.2]{drillick2024stochastic}  and carefully track how lower priority particles can overtake higher priority ones.  

  We next prove KPZ fixed point convergence for a continuous initial condition which is bounded above by $\lambda(1+|x|)$ for some $\lambda>0$ (Proposition~\ref{p.main result in continuous case}) using height monotonicity, the earlier mentioned lower bound by a prelimiting variational problem, and an upper bound from the previous paragraph. 
   A similar approximation argument then finally allows us to extend to upper semi-continuous functions satisfying the same bound $\lambda(1+|x|)$, which is the class of initial conditions for which convergence to the KPZ fixed point is known for TASEP \cite{matetski2016kpz}.

\subsection*{Notation}
For $a,b\in\Z$ with $a<b$, $\intint{a,b} := \{a, \ldots, b\}$. 
For random objects $X$ and $Y$ taking values in some measurable space, $X\stackrel{\smash{d}}{=} Y$ means that their distributions are the same; we will often omit explicitly specifying the $\sigma$-algebra on the target space, but will specify its topology (in which case we endow the space with the associated Borel $\sigma$-algebra). For random objects $X_n$ and $X$ taking values in some common topological space, $\smash{X_n\xrightarrow{d} X}$ means that $X_n$ converges weakly to $X$; the topological space will be specified or obvious in the context. The space of continuous functions from a topological space $\mc X$ to $\R$ will be denoted $\mc C(\mc X, \R)$. Events will be written in sans serif font, e.g., $\msf E$, and $\msf E^c$ will denote the complement of $\msf E$.
We will sometimes write $f(\bm\cdot)$ for a function of the variable $\bm\cdot$, and sometimes $x\mapsto f(x)$.

\subsection*{Organization of paper} 
We define our models and state the main result in Section~\ref{s.models and result}. In Section~\ref{s.preliminaries} we collect some useful properties and definitions for the models under study, as well as of the KPZ fixed point. In Section~\ref{s.convergence under bernoulli} we prove convergence to the KPZ fixed point under Bernoulli initial condition. In Section~\ref{s.general one-sided initial condition} we upgrade this to establish the same convergence under general continuous initial conditions which are Bernoulli outside of a compact set, and use the latter in Section~\ref{s.proofs of main theorems} to prove Theorem~\ref{t.main kpz fixed point convergence}. Proofs of miscellaneous statements not proved in the main text are provided in Appendix~\ref{app.other proofs}.

\subsection*{Acknowledgements}
Amol Aggarwal was partially supported by a Packard Fellowship for Science and Engineering, a Clay Research Fellowship, by the NSF through grant DMS-1926686, and by the IAS School of Mathematics.
Ivan Corwin was partially supported by the NSF through grants DMS-1937254, DMS-1811143, DMS-1664650, by the Simons Foundation through an Investigator Award and through the W.M.~Keck Foundation through a Science and Engineering Grant.
Milind Hegde was partially supported by the NSF through grants DMS-1937254 and DMS-2348156.

\section{Models and main result}\label{s.models and result}

In this section we define our models and state the precise form of our main result, Theorem~\ref{t.main kpz fixed point convergence}. We define ASEP in Section~\ref{s.asep definition} and S6V in Section~\ref{s.s6v definition}, and give the form of the rescaling of the height functions in Section~\ref{s.height function scaling}. In Section~\ref{s.landscape and fixed point} we define the directed landscape and KPZ fixed point and in Section~\ref{s.main result} we state our main result.

\subsection{The asymmetric simple exclusion process}\label{s.asep definition}

Fix a real number $q\in[0,1)$. Consider an initial particle configuration $\eta_0\in\{0,1\}^{\Z}$, where $\eta(x)=1$ indicates the presence of a particle at $x$ and $\eta(x) = 0$ indicates the absence of a particle at $x$ (equivalently, presence of a hole). The dynamics of ASEP are as follows.

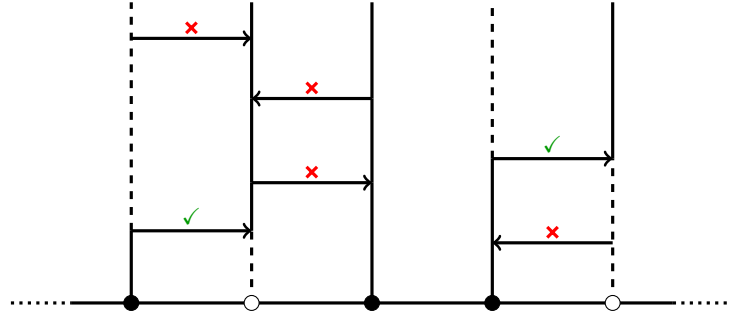
\begin{figure}[h]
\begin{tikzpicture}[scale=1.6]

\draw[line width=1.2pt] (-0.5,0) -- ++ (5,0);
\draw[line width=1.2pt, dotted] (-1,0) -- ++(0.5,0);
\draw[line width=1.2pt, dotted] (4.5,0) -- ++(0.5,0);

\draw[->, line width=1.2pt] (1,1) --node[above=2pt, cross out, inner sep=2pt, outer sep=0pt, draw=red, minimum size=4pt, line width=1.2pt]{} ++(1,0);

\draw[->, line width=1.2pt] (2,1.7) --node[midway, above=2pt, cross out, inner sep=2pt, outer sep=0pt, draw=red, minimum size=4pt, line width=1.2pt]{} ++(-1,0);
\draw[->, line width=1.2pt] (0,0.6) --node[midway, green!60!black, above=-1pt, scale=0.8]{$\bm{\checkmark}$} ++(1,0);
\draw[->, line width=1.2pt] (4,0.5) --node[midway, above=2pt, cross out, inner sep=2pt, outer sep=0pt, draw=red, minimum size=4pt, line width=1.2pt]{} ++(-1,0);
\draw[->, line width=1.2pt] (3,1.2) --node[midway, green!60!black, above=-1pt, scale=0.8]{$\bm{\checkmark}$} ++(1,0);
\draw[->, line width=1.2pt] (0,2.2) --node[midway, above=2pt, cross out, inner sep=2pt, outer sep=0pt, draw=red, minimum size=4pt, line width=1.2pt]{} ++(1,0);

\draw[line width=1.2pt, dashed] (1,0) -- (1,0.6);
\draw[line width=1.2pt, dashed] (0,0.6) -- (0,2.5);

\draw[line width=1.2pt, black] (0,0) -- (0,0.6);
\draw[line width=1.2pt, black] (1,0.6) -- (1,1);
\draw[line width=1.2pt, black] (2,1) -- (2,1.7);
\draw[line width=1.2pt, black] (1,1.7) -- (1,2.5);

\draw[line width=1.2pt, black] (2,0) -- (2,1);
\draw[line width=1.2pt, black] (1,1) -- (1,1.7);
\draw[line width=1.2pt, black] (2,1.7) -- (2,2.5);

\draw[line width=1.2pt, black] (3,0) -- (3,1.2);
\draw[line width=1.2pt, black] (4,1.2) -- (4,2.5);

\draw[line width=1.2pt, dashed] (4,0) -- (4,1.2);
\draw[line width=1.2pt, dashed] (3,1.2) -- (3,2.5);

\foreach \x/\thecolor in {0/black, 1/white, 2/black, 3/black, 4/white}
{
\node[circle, fill, inner sep = 2pt, \thecolor, draw=black] at (\x,0) {};
}
\end{tikzpicture}
\caption{A depiction of the graphical construction of ASEP. At the bottom is the initial particle configuration. The arrows indicate the times where a left clock $\xi^{\mrm{L}}_x$ or right clock $\xi^{\mrm{R}}_x$ rang. If an arrow goes from site $x$ to site $y$ at time $t$ and there is a particle at $x$ and no particle at site $y$ at time $t^-$, the particle at site $x$ moves to $y$ at time $t$ (indicated by a green check mark); if not, nothing happens (red cross).}\label{f.graphical construction}
\end{figure}

\subsubsection{Dynamics}\label{s.asep dynamics} We fix a family of independent Poisson clocks $\bm \xi = \{\xi_x^{\mrm{L}}, \xi_x^{\mrm{R}}\}_{x\in\Z}$, where $\xi_x^{\mrm{L}}$ has rate $q$ and $\xi_{x}^{\mrm{R}}$ has rate 1 for all $x\in\Z$. With the clocks fixed, the dynamics are deterministic (see Figure~\ref{f.graphical construction}): when the clock $\xi_x^{\mrm{L}}$ rings, if there is a particle at site $x$ at that time, it attempts a jump to the first site to the left, and when the clock $\xi_x^{\mrm{R}}$ rings, if there is a particle at site $x$ at that time, it attempts a jump by one site to the right. The attempt succeeds if there is no particle at the target site. The well-definedness of this prescription follows from \cite[Section 10]{harris1978additive}.

\subsubsection{Basic coupling for ASEP}

The above graphical construction of ASEP gives a natural way to couple together the evolution of ASEP from all choices of initial conditions simultaneously---this is known as the \emph{basic coupling}.

Consider countably many initial conditions $\eta_0^{(1)}, \eta_0^{(2)}, \ldots \in \{0,1\}^\Z$, where $\eta_0^{(j)}(i)$ represents the state at location $i$ in the $j$\th initial condition, with $0$ for a hole and $1$ for a particle. To define the basic coupling, as in Section~\ref{s.asep dynamics}, we fix a family of independent Poisson clocks $\bm \xi = \{\xi_x^{\mrm{L}}, \xi_x^{\mrm{R}}\}_{x\in\Z}$. Then we evolve $\smash{\eta_0^{(j)}}$ for each $j$ according to the dynamics described above, where we use the same family of clocks $\bm \xi$ for all $j\in\N$. For example, if $\xi_x^{\mrm{R}}$ rings at time $t$, then the particle at location $x$ in the configuration $\smash{\eta^{(j)}_{t^-}}$ (if there is one) attempts a jump to the right for each $j\in\N$.

We also note that we may consider the same evolution, using the fixed Poisson clocks, of initial conditions started at different times, i.e., for given times $s_i>0$ and initial configurations $\smash{\eta_0^{(i)}}$ for $i=1,2, \ldots, $ we set the configuration at time $s_i>0$ to be $\smash{\eta_0^{(i)}}$ and apply the same dynamics. In this way the basic coupling gives a coupling of initial conditions started at different times as well.

\subsubsection{Evolution of height function for ASEP}\label{s.asep height function}

\begin{definition}[Height function]\label{d.bernoulli function}
Let $\Lambda\subseteq \Z$ be a (possibly infinite) interval. We say $\gamma:\Lambda\to\Z$ is a \emph{height function} if $\gamma(x+1) - \gamma(x)\in\{0,-1\}$ for all $x$ with $x, x+1\in\Lambda$.
\end{definition}

Given a height function $h_0:\Z\to\Z$, the associated particle configuration $\eta_{h_0}\in \{0,1\}^{\Z}$ is defined, for all $x\in\Z$, by
\begin{align*}
\eta_{h_0}(x) = h_0(x-1) - h_0(x).
\end{align*}

\begin{figure}[t]
\begin{tikzpicture}[scale=0.85]
\draw[thick] (0,0) -- ++(2,0) -- ++(2,-2) -- ++(1,0) -- ++ (1,0) -- ++(1,-1) -- ++(2,0);

\draw[thick] (0,-4.5) -- ++(9.4,0);

\foreach \thecolor [count=\x] in {white, white, black, black, white, white, black, white, white}
\node[circle, fill=\thecolor, inner sep = 2pt, draw=black] at (\x,-4.5) {};

\draw[->, dashed, semithick] (7, -4.2) to[out=40, in=140] (8,-4.2);

\draw[->, dashed, semithick] (3, -4.2) to[out=140, in=40] (2,-4.2);

\draw[thick, dashed] (1,0) -- ++(1,-1) -- ++(1,0);
\draw[thick, dashed] (6,-2) -- ++(1,0) -- ++(1,-1);
\end{tikzpicture}
\caption{A depiction of a height function $h_0$ and the associated particle configuration $\eta_{h_0}$. The dashed portions of the path illustrate the definition of the new height function, under the corresponding displayed particle movements.}\label{f.general height function}
\end{figure}
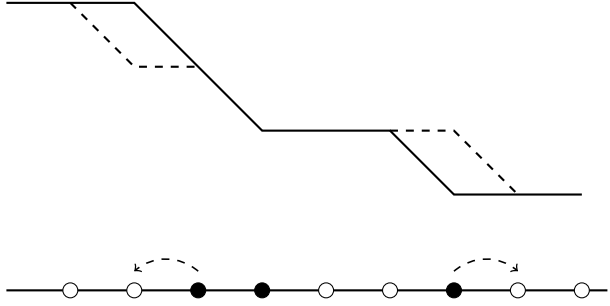

For a function $f$ with left limits, we use the notation $f(r^-) := \lim_{s\shortuparrow r} f(s)$; note that this does not require $f(r)$ to be defined.  For $h_0:\Z\to\Z$ a height function, define the height function $(y,t)\mapsto h^{\mrm{ASEP}}(h_0; y,t)$ for $t>0$ and $y\in\Z$ as follows (see also Figure~\ref{f.general height function}). 

Start ASEP under basic coupling with initial particle configuration $\eta_{h_0}$. Whenever a particle jumps from $y$ to $y+1$ at time $r>0$, define $h^{\mrm{ASEP}}(h_0; y,r) = h^{\mrm{ASEP}}(h_0; y,r^-) + 1$ (and $h^{\mrm{ASEP}}(h_0; x,r) = h^{\mrm{ASEP}}(h_0; x,r^-)$ for all $x\neq y$), and whenever a particle jumps from $y$ to $y-1$ at time $r$, define $h^{\mrm{ASEP}}(h_0; y-1, r) = h^{\mrm{ASEP}}(h_0; y-1,r^-)-1$ (and $h^{\mrm{ASEP}}(h_0; x,r) = h^{\mrm{ASEP}}(h_0; x,r^-)$ for all $x\neq y-1$).
The utility of the definition of $h^{\mrm{ASEP}}(h_0; y,t)$ is that, while the ASEP dynamics only sees the discrete derivative of $h_0$ (through $\eta_{h_0}$), here we keep track of the constant shift in $h_0$ as well.

For multiple initial conditions $(\smash{h_0^{(j)}})_{j=1}^\infty$, we couple the height functions $(y,t)\mapsto h^{\mrm{ASEP}}(\smash{h_0^{(j)}}; y,t)$ using the basic coupling for the underlying ASEP dynamics. If the initial condition $h_0$ starts at a non-zero time $s>0$, we denote the height functions at time $t>s$ by $h^{\mrm{ASEP}}(h_0,s; \bm\cdot, t)$ and its evolution is as in the previous paragraph. One can also start multiple initial conditions at different times, and couple the evolutions by the basic coupling.%

\subsection{The stochastic six-vertex model}\label{s.s6v definition}

The stochastic six-vertex model (S6V) we consider is defined on the $\Z_{\geq 1}\times \Z$. The model has two fixed parameters: an asymmetry parameter $q\in[0,1)$ and a direction parameter $b^{\shortrightarrow}\in(0,1)$.

\subsubsection{Configurations and weights} \label{s.s6v configuration}

A configuration of the model is described as follows. First, an \emph{arrow configuration} is a tuple $(a, i; b, j)$ such that $i, j, a, b \in \{0,1\}$ and $i+ a = j + b$; see the left panel of Figure~\ref{f.colored S6V}. The variables $i$, $a$, $j$, and $b$ indicate the presence or absence  of the arrows horizontally entering, vertically entering, horizontally exiting, and vertically exiting the vertex, respectively. The condition $i+a = j+b$ is called \emph{arrow conservation}, i.e., the number of incoming arrows equals the number of outgoing arrows.

A configuration of the model consists of an assignment of arrow configurations, one to each vertex $v$ of $\Z_{\geq 1}\times \Z$. We denote the arrow configuration at $v$ by $(a_v, i_v; b_v, j_v)$. We require that the arrow configurations are \emph{consistent}: $i_{(x,y)} = j_{(x-1,y)}$ and $a_{(x,y)} = b_{(x,y-1)}$, i.e., an arrow is horizontally entering at $(x,y)$ if and only if an arrow is horizontally exiting at $(x-1,y)$, and similarly an arrow is vertically entering at $(x,y)$ if and only if an arrow is vertically exiting at $(x,y-1)$. 
One consequence of this description is that the arrows form up-right paths; see the right panel of Figure~\ref{f.colored S6V}.

\begin{figure}[t]
\begin{tikzpicture}[scale=0.7]

\begin{scope}[shift={(-7.5,1.5)}, scale=0.9]

\draw[line width=1pt] (0,0) -- ++(3,0);
\draw[line width=1pt] (1.5,-1.5) -- ++(0,3);

\node[anchor = east] at (0,0) {$i_v$};
\node[anchor = west] at (3,0) {$j_v$};
\node[anchor = north] at (1.5,-1.5) {$a_v$};
\node[anchor = south] at (1.5,1.5) {$b_v$};
\end{scope}

\draw[gray, dashed] (0,-1) grid (4,4.4);
\draw[black, thick] (0,4.4) -- (0,-1) -- (4,-1);

\draw[very thick, ->] (0,0) -- ++(2,0) -- ++(0,1) -- ++(2,0);

\draw[very thick, ->] (0,1) -- ++(1,0) -- ++(0,1) -- ++(2,0) -- ++(0,1) -- ++(1,0);

\draw[very thick, ->] (0,3) -- ++(1,0) -- ++(0,1.4);

\draw[very thick, ->] (0,2) -- ++(1,0) -- ++(0,1) -- ++(1,0) -- ++(0,1) -- ++(2,0);

\draw[very thick, ->] (0,4) -- ++(2,0)-- ++(0,0.4);


\foreach \x in {-2,...,2}
{
  \node[anchor=east, scale=0.8] at (0,\x+2) {$\x$};
}

\node[circle, fill, inner sep=1pt] at (0,-1) {};
\node[anchor=east, scale=0.8] at (0,-1) {$(0,-3)$};


\end{tikzpicture}
\caption{Left: a depiction of the labels for the arrow configuration at a vertex $v$. Right: A sample configuration of the stochastic six-vertex model with the step initial condition, i.e., $\eta_0(k) = 1$ for all $k\geq k_0$ for some $k_0$ (in the depiction, $k_0=-2$) and $\eta_0(k) = 0$ otherwise. }\label{f.colored S6V}
\end{figure}
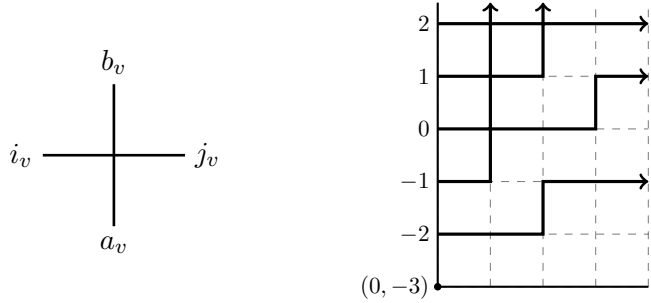

We turn now to initial conditions for the model, starting with the case that the initial condition is not bi-infinite. Fix $N\in\N$. The S6V model with \emph{initial condition} $\eta_0:\llbracket -N,\infty\rrparen\to\{0,1\}$ is specified by setting an arrow to enter horizontally at $(1, k)$ whenever $\eta_0(k) = 1$ (i.e., $i_{(1,k)} = \eta_0(k)$) for $k \in \llbracket -N,\infty\rrparen$, and no arrow  enters vertically at $(\ell, -N)$ (i.e., $a_{(\ell,-N)} = 0$) for $\ell \in \mathbb{N}$; see the right panel of Figure~\ref{f.colored S6V}. Next we explain how to sample a configuration with such a initial condition; this prescription can be extended to the case of initial conditions defined on $\Z$, which we will indicate shortly.

\begin{figure}[b]
\begin{tikzpicture}[scale=0.8]
\newcommand{\thedim}{4.2}
\begin{scope}[scale=0.65]

      \draw[step=\thedim] (0,0) grid (6*\thedim, \thedim);

      \draw[line width=1pt, dashed, gray] (0.5*\thedim, 0.25*\thedim) -- ++(0, 0.5*\thedim);
      \draw[line width=1pt, dashed, gray] (0.25*\thedim, 0.5*\thedim) -- ++(0.5*\thedim, 0);

      \begin{scope}[shift={(\thedim,0)}]
            \draw[->, line width=1.2pt, black] (0.5*\thedim, 0.25*\thedim) -- ++(0, 0.5*\thedim);
            \draw[->, line width=1.2pt, black] (0.25*\thedim, 0.5*\thedim) -- ++(0.5*\thedim, 0);

            \draw[line width=1pt, gray, dashed] (1.25*\thedim, 0.5*\thedim) -- ++(0.5*\thedim, 0);
            \draw[->, line width=1.2pt, black] (1.5*\thedim, 0.25*\thedim) -- ++(0, 0.5*\thedim);

            \draw[line width=1pt, gray, dashed] (2.5*\thedim, 0.25*\thedim) -- ++(0, 0.5*\thedim);
            \draw[->, line width=1.2pt, black] (2.25*\thedim, 0.5*\thedim) -- ++(0.5*\thedim, 0);

            \draw[line width=1pt, gray, dashed] (3.25*\thedim, 0.5*\thedim) -- ++(0.25*\thedim, 0) -- ++(0,0.25*\thedim);
            \draw[->, line width=1.2pt, black] (3.5*\thedim, 0.25*\thedim) -- ++(0, 0.25*\thedim) -- ++(0.25*\thedim, 0);

            \draw[line width=1pt, gray, dashed] (4.5*\thedim, 0.25*\thedim) -- ++(0, 0.25*\thedim) -- ++(0.25*\thedim, 0);
            \draw[->, line width=1.2pt, black] (4.25*\thedim, 0.5*\thedim) -- ++(0.25*\thedim, 0) -- ++(0,0.25*\thedim);
            
      \end{scope}

      \draw[xstep=\thedim, ystep=0.35*\thedim] (0,0) grid (6*\thedim, -0.35*\thedim);

      \newcommand{\they}{-0.175*\thedim}

      \node[scale=1] at (0.5*\thedim, \they) {$1$};

      \begin{scope}[shift={(\thedim,0)}]

      \node[scale=1] at (0.5*\thedim, \they) {$1$};
      \node[scale=1] at (1.5*\thedim, \they) {$qb^{\shortrightarrow}$};
      \node[scale=1] at (2.5*\thedim, \they) {$b^{\shortrightarrow}$};
      \node[scale=1] at (3.5*\thedim, \they) {$1-qb^{\shortrightarrow}$};
      \node[scale=1] at (4.5*\thedim, \they) {$1-b^{\shortrightarrow}$};
      \end{scope}

     \end{scope}
\end{tikzpicture}
\caption{The vertex weights for the S6V model.}
\label{f.R weights}
\end{figure}

\subsubsection{Dynamics}\label{s.s6v dynamics} The probability measure on configurations is defined by the following Markovian sampling procedure. Start at the vertex $(1,-N)$. Here there is a horizontally incoming arrow determined by the initial condition $\eta_0$ and no vertically entering arrow. The outgoing arrow configuration is determined by choosing each of the two possible configurations (whether the horizontally incoming arrow exits vertically or horizontally) with probability equal to the vertex weight read from Figure~\ref{f.R weights}; this uses that the vertex weights are \emph{stochastic}, i.e., non-negative and for a given incoming arrow configuration, the sum  of the possible (i.e., satisfying arrow conservation) outgoing arrow configurations' weights is 1. We then proceed iteratively in a Markovian fashion: if the outgoing arrow configuration at all vertices in the domain $\mc D_n=\{(x,y)\in\Z_{\geq 1}\times\llbracket-N,\infty\rrparen: x+y < n\}$ has been determined, the configuration for vertices on the line $x+y=n$ is determined by picking the outgoing arrow configuration at each vertex independently with probability equal to that vertex's weight; note that the incoming arrows configuration at each such vertex is already determined by the requirement of consistency of the overall configuration, and that the order in which the outgoing arrow configurations are sampled does not matter. The particle configuration at a later time $t\in\N$ is denoted $\eta_t:\llbracket -N,\infty\rrparen\to\{0,1\}$ and defined by
\begin{align*}
\eta_t(x) = j_{(t,x)}.
\end{align*}

As mentioned, the above sampling is for an initial condition which is finite on one side, i.e., $\eta_0$ is defined on $\llbracket -N, \infty\rrparen$ for some $N\in\N$ rather than $\Z$. The well-definedness of the model for initial conditions defined on $\Z$ follows by a method similar to that of Harris \cite{harris1978additive} used for the well-definedness of ASEP; we refer the reader to \cite[Section 2.1]{aggarwal2020limit} for more details.

\subsubsection{The basic coupling for S6V}\label{s.s6v basic coupling}

We next give an analog of ASEP's basic coupling for S6V, which again corresponds to using the same randomness for the evolution of different initial conditions.
Consider finitely many initial conditions $\eta_0^{(1)}, \ldots, \eta_0^{(K)}\in \{0,1\}^{\Z}$ where $\eta_0^{(j)}(x)$ represents the presence or absence of a horizontal arrow entering at location $(1,x)$, i.e., $\smash{i_{(1,x)} = \eta_0^{(j)}(x)}$. 

 We start with a collection of independent $0$-$1$ Bernoulli random variables $\{\smash{X^{\shortuparrow}_{v}, X^{\shortrightarrow}_{v}\}_{v\in\Z_{\geq 1}\times\Z}}$, where $\P(\smash{X^{\shortuparrow}_{v}} = 1) = qb^{\shortrightarrow}$ and $\P(\smash{X^{\shortrightarrow}_{v}} = 1) = b^{\shortrightarrow}$ for each $\smash{v\in\Z_{\geq 1}\times\Z}$.
Given this randomness, the configuration is specified deterministically from the initial condition as follows. For any vertex $(x,y)\in\Z_{\geq 1}\times\Z$ such that the incoming arrow configuration has been assigned to $(x,y)$ (by the assignment of outgoing arrow configurations to $(x-1, y)$ and $(x,y-1)$), let $I_{\shortrightarrow}, I_{\shortuparrow}\in\{0,1\}$ respectively be $i_{(x,y)}$ and $a_{(x,y)}$, the indicator function of the horizontally and vertically incoming arrows at $(x,y)$ respectively. Then the horizontally and vertically outgoing arrows $O_{\shortrightarrow}, O_{\shortuparrow}\in\{0,1\}$ from $(x,y)$ (i.e., $j_{(x,y)}$ and $b_{(x,y)}$)  are defined by:
\begin{enumerate}
  \item If $I_{\shortrightarrow} = I_{\shortuparrow}$, then $O_{\shortrightarrow} = O_{\shortuparrow} = I_{\shortrightarrow}= I_{\shortuparrow}$.

  \item If $(I_{\shortrightarrow}, I_{\shortuparrow})= (0,1)$, then $(O_{\shortrightarrow}, O_{\shortuparrow}) = (0,1)$ if $X^{\shortuparrow}_{(x,y)} = 1$ and $(O_{\shortrightarrow}, O_{\shortuparrow}) = (1,0)$ if $X^{\shortuparrow}_{(x,y)} = 0$.

  \item If $(I_{\shortrightarrow}, I_{\shortuparrow})= (1,0)$, then $(O_{\shortrightarrow}, O_{\shortuparrow}) = (1,0)$ if $X^{\shortrightarrow}_{(x,y)} = 1$ and $(O_{\shortrightarrow}, O_{\shortuparrow}) = (0,1)$ if $X^{\shortrightarrow}_{(x,y)} = 0$.
\end{enumerate}

\subsubsection{Evolution of height function for S6V}\label{s.s6v height function}
Next we define the S6V height function when started from a general initial height function.  Let  $h_0:\Z\to \Z$ be a height function. As in ASEP, we define an associated arrow configuration $\eta_{h_0}\in\{0,1\}^{\Z}$ by
\begin{align}\label{e.s6v particle from h}
\eta_{h_0}(x) := h_0(x-1) - h_0(x).
\end{align}

Now, for such a height function $h_0$, consider the S6V model with initial condition $\eta_{h_0}$, and let $j_{(x,y)}$ be the arrow exiting horizontally from $(x,y)$. Define the height function $(y,t)\mapsto \hssv(h_0; y,t)$ for $t\in\N$ and $y\in \Z$ iteratively as follows. First set $\hssv(h_0; y,0) = h_0(y)$ for all $y\in\Z$. Informally speaking, for $t\in\N$, $\hssv(h_0; y,t)$ is $\hssv(h_0; y,t-1)$ plus the number of arrows that exited horizontally at a location at or below height $y$ at time $t-1$ and above $y$ at time $t$. Since at most one arrow can occupy an edge, the latter number is either 0 or 1, and is 1 if and only if an arrow exits vertically from $(t,y)$. More formally, define, for $t\in\N$ and $y\in\Z$,
\begin{equation}\label{e.s6v height function gen initial condition}
\hssv(h_0; y,t) := \hssv(h_0; y,t-1) + b_{(t,y)}.
\end{equation}

For multiple initial conditions $(h_0^{(k)})_{k=1}^K$, the height functions $(y,t)\mapsto\smash{\hssv(h_0^{(k)}; y,t)}$ are coupled together via the basic coupling for S6V. As in ASEP, we can associate each initial condition with a starting time $s_k\in\N$ by using the evolution as described in Section~\ref{s.s6v basic coupling} from $s_i$ onward (in particular, coupled via the basic coupling) and define the height functions (now denoted $\smash{\hssv(h_0^{(k)}, s_k; y,t)}$) by \eqref{e.s6v height function gen initial condition} for $t>s_k$.

\begin{remark}
If the particle configuration associated to $h_0$ has only finitely many particles, the definition \eqref{e.s6v height function gen initial condition} of $h^{\mrm{S6V}}(h_0; y,t)$ can be interpreted as the number of arrows exiting horizontally from the infinite ray $\{(t,z): z>y\}$, up to a constant global shift (equaling $\lim_{y\to\infty} h_0(y)$). A constant global shift of $k$ can in turn be interpreted as the system having $k$ arrows at height ``$\infty$'' at all times. These observations provide a link between the definition of the S6V height function given here and that given in \cite[eq. (2.22)]{aggarwal2024scaling}, where all initial configurations have finitely many arrows and the S6V height function is precisely the number of arrows exiting horizontally above a point.
\end{remark}

\subsection{Height function scaling}\label{s.height function scaling}

\subsubsection{ASEP} Fix $q\in[0,1)$ and a velocity $\alpha\in(-1,1)$ in the rarefaction fan. Let $\mu^{\mrm{ASEP}}(\alpha)$ and $\sigma^{\mrm{ASEP}}(\alpha)$ be defined by
\begin{align}\label{e.mu sigma ASEP}
\mu(\alpha) = \mu^{\mrm{ASEP}}(\alpha) := \tfrac{1}{4}(1-\alpha)^2 \qquad\text{and}\qquad \sigma(\alpha) = \sigma^{\mrm{ASEP}}(\alpha) :=\tfrac{1}{2}(1-\alpha^2)^{2/3},
\end{align}
and let the spatial scaling factor $\beta^{\mrm{ASEP}}(\alpha)$ be given by
\begin{align}\label{e.nu ASEP}
\beta(\alpha)=\beta^{\mrm{ASEP}}(\alpha) := \frac{2\sigma(\alpha)^2}{|\mu'(\alpha)|(1-|\mu'(\alpha)|)} = 2(1-\alpha^2)^{1/3}.
\end{align}

\begin{definition}[Rescaled height functions for ASEP]\label{d.rescaled ASEP height function}
 Given a sequence of initial height functions $h_0^\varepsilon:\Z\to\Z$ indexed by $\varepsilon>0$, we define $\h_0^\varepsilon:\R\to\R$ by setting, for all $x$ such that the argument of $h_0^\varepsilon$ is an integer,
\begin{align}\label{e.rescaled h for ASEP}
\h_0^\varepsilon(x) = \sigma(\alpha)^{-1}\varepsilon^{1/3}\left(\mu'(\alpha)\beta(\alpha)x\varepsilon^{-2/3} - h_0^\varepsilon(\beta(\alpha)x\varepsilon^{-2/3})\right),
\end{align}
and the values of $\h_0^\varepsilon$ at all other $x$ determined by linear interpolation. Note that the dependence of $\h_0^\varepsilon$ on $\alpha$ is not explicit in the notation.

Let $\gamma = 1-q$. We define the rescaled ASEP height function by
\begin{align*}
\MoveEqLeft[16]
\h^{\mrm{ASEP},\varepsilon}(\h_0^\varepsilon; x,t)
= \sigma(\alpha)^{-1}\varepsilon^{1/3}\Bigl(2\mu(\alpha)t\varepsilon^{-1} + \mu'(\alpha)\beta(\alpha)x\varepsilon^{-2/3}\\
&- h^{\mrm{ASEP}}(h_0^\varepsilon; 2\alpha t\varepsilon^{-1} + \beta(\alpha)x\varepsilon^{-2/3}, 2\gamma^{-1}\varepsilon^{-1}t)\Bigr)
\end{align*}
for all $x$ such that the argument of $h^{\mrm{ASEP}}(h_0;\bm\cdot, 2\gamma^{-1}\varepsilon^{-1}t)$ is an integer, and by linear interpolation elsewhere. Here $\h_0^\varepsilon$ and $h_0^\varepsilon$ are related by \eqref{e.rescaled h for ASEP}.
\end{definition}

The source of these scaling coefficients is explained in \cite[Section 2.2.4]{aggarwal2024scaling} and the interested reader is referred there for more details.

\subsubsection{S6V} Fix $q\in[0,1)$, $b^{\shortrightarrow}\in (0,1)$. Let $z= \frac{1-b^{\shortrightarrow}}{1-qb^{\shortrightarrow}}\in(0,1)$ and fix a velocity $\alpha \in (z,z^{-1})$ (which corresponds to the full rarefaction fan). Define $\mu(\alpha)$ and $\sigma(\alpha)$ by
\begin{equation}\label{e.mu and sigma}
\begin{split}
\mu(\alpha) &= \mu^{\mrm{S6V}}(\alpha) := - \frac{(\sqrt{\alpha} - \sqrt{z})^2}{1-z} \quad\text{and}\\
\sigma(\alpha) &= \sigma^{\mrm{S6V}}(\alpha) :=\frac{\alpha^{-1/6}z^{1/6}(1-\sqrt{z\alpha})^{2/3}(\sqrt{\alpha}-\sqrt{z})^{2/3}}{1-z};
\end{split}
\end{equation}
Next define the spatial scaling factor $\beta(\alpha) = \beta^{\mrm{S6V}}(\alpha)$ by
\begin{equation}\label{e.S6V spatial scaling}
\beta(\alpha) = \beta^{\mrm{S6V}}(\alpha) := \frac{2\sigma(\alpha)^2}{|\mu'(\alpha)|(1-|\mu'(\alpha)|)}.
\end{equation}

\begin{definition}[Rescaled height functions for S6V]\label{d.rescaled S6V height function}
Given a sequence of initial height functions $h_0^\varepsilon: \Z\to\Z$ indexed by $\varepsilon>0$, we define $\h_0^\varepsilon:\R\to\R$ by setting, for all $x$ such that the argument of $h_0^\varepsilon$ is an integer,
\begin{equation}\label{e.rescaled h for S6V}
\begin{split}
\MoveEqLeft[13]
\h_0^\varepsilon(x) = \sigma(\alpha)^{-1}\varepsilon^{1/3}\Bigl(h_0^\varepsilon(\beta(\alpha)x\varepsilon^{-2/3})  - \mu'(\alpha)\beta(\alpha)x\varepsilon^{-2/3} \Bigr),
\end{split}
\end{equation}
and the values of $\h_0^\varepsilon$ at all other $x$ determined by linear interpolation. Note that the dependence of $\h_0^\varepsilon$ on $\alpha$, $q$, and $z$ is not explicit in the notation.

We also define the rescaled S6V height function, for $t \in (0,\infty)$, by 
\begin{align*}
\MoveEqLeft[16]
\h^{\mrm{S6V},\varepsilon}(\h_0^\varepsilon; x,t)
= \sigma(\alpha)^{-1}\varepsilon^{1/3}\Bigl(h^{\mrm{S6V}}(h_0^\varepsilon; \alpha t\varepsilon^{-1} + \beta(\alpha)x\varepsilon^{-2/3}, \floor{\varepsilon^{-1}t})  - \mu(\alpha)t\varepsilon^{-1} - \mu'(\alpha)\beta(\alpha)x\varepsilon^{-2/3}\Bigr)
\end{align*}
for all $x$ such that the argument of $h^{\mrm{S6V}}(h_0^\varepsilon;\bm\cdot, \floor{\varepsilon^{-1}t})$ is an integer, and by linear interpolation elsewhere. Again $\h_0^\varepsilon$ and $h_0^\varepsilon$ are related by \eqref{e.rescaled h for S6V}.
\end{definition}

\subsection{The directed landscape and KPZ fixed point}\label{s.landscape and fixed point}

\subsubsection{The directed landscape}

We give the definition of the directed landscape in terms of the Airy sheet $\S:\R^2\to\R$; the reader is referred to \cite[Definition 8.1]{dauvergne2018directed} for the definition of $\S$. 

\begin{definition}[Directed landscape, {\cite[Definition 10.1]{dauvergne2018directed}}]\label{d.directed landscape}
The \emph{directed landscape} is the uni\-que (in law) random continuous function $\mc L : \{(x,s;y,t) \in \R^4: s<t\} \to \R$ that satisfies the following properties.
\begin{enumerate}
  \item[(i)] \emph{Airy sheet marginals}: For any $s \in \R$ and $t > 0$ the increment over time interval $[s, s + t)$ is a rescaled Airy sheet:
 \begin{align*}
  (x,y)\mapsto \mc L(x, s; y, s + t) \stackrel{d}{=} (x,y)\mapsto t^{1/3}\S(xt^{-2/3}; yt^{-2/3}).
 \end{align*}

\item[(ii)] \emph{Independent increments}: For any $k\in\N$ and disjoint time intervals $\{(s_i, t_i) : i \in \intint{1, k}\}$, the random functions
$$(x,y)\mapsto\mc L(x, s_i; y, t_i)$$
are independent as $i \in \intint{1,k}$ varies.

\item[(iii)] \emph{Metric composition law}: Almost surely, for any $r < s < t$ and $x, y \in \R$,
$$\mc L(x, r; y, t) = \max_{z\in\R}\Bigl(\mc L(x, r; z, s) + \mc L(z, s; y, t)\Bigr).$$
\end{enumerate}
The existence and uniqueness of the directed landscape is proved in \cite[Theorem 10.9]{dauvergne2018directed}.
\end{definition}

\subsubsection{The KPZ fixed point}

\begin{definition}[Space of upper semi-continuous functions]\label{d.uc metric}
For each $\lambda>0$, let $\mrm{UC}_\lambda$ be the space of upper semi-continuous functions
$f : \R \to \R \cup \{-\infty\}$ satisfying $f(x) \leq \lambda(1+|x|)$ for all $x\in\R$. For an interval $I\subseteq \R$ and $f:I\to\R$, let $\mrm{hypo}(f) = \{(x,y)\in I\times\R: y\leq f(x)\}$ be the hypograph of $f$. Consider the metric $d_{[-\infty, \infty)}(x, y) = |e^x - e^y |$ on $\R\cup\{-\infty\}$, and endow $\mrm{UC}_\lambda$ with the metric $d_{\mrm{UC},\lambda}$ defined by
\begin{align*}
d_{\mrm{UC},\lambda}(f,g) = \sum_{\ell=1}^{\infty}2^{-\ell} \min\left(1, d_{\mrm{Haus}, [-\infty,\infty), [-\ell,\ell]}\left(\mrm{hypo}(f|_{[-\ell,\ell]}), \mrm{hypo}(g|_{[-\ell,\ell]})\right)\right),
\end{align*}
where $d_{\mrm{Haus}, [-\infty,\infty), [-\ell,\ell]}$ is the Hausdorff metric on subsets of $[-\ell, \ell]\times\R$ induced by the metric $d_{[-\infty, \infty)}$.
\end{definition}

\begin{definition}[Convergence in UC]
We define $\mrm{UC} = \cup_{\lambda=1}^\infty \mrm{UC}_\lambda$. We say $\{f^{\varepsilon}\}_{\varepsilon>0}\subseteq \mrm{UC}$ converges locally in UC to $f\in\mrm{UC}$ if there is some $\lambda>0$ such that $f^\varepsilon\in \mrm{UC}_{\lambda}$ for all $\varepsilon>0$ and $f^\varepsilon\to f$ in $\mrm{UC}_\lambda$.
\end{definition}

\begin{definition}
For a function $\h_0 \in\mrm{UC}$, and real numbers $s,t>0$ with $s>t$, define the \emph{KPZ fixed point} $(y,t)\mapsto\h(\h_0, s; y,t): \R\times(0,\infty)\to \R$ started from $\h_0$ at time $s$ by
\begin{align}\label{e.fixed point via landscape}
\h(\h_0, s; y, t) := \sup_{x\in\R} \Bigl(\h_0(x) + \cL(x,s;y,t)\Bigr).
\end{align}
\end{definition}

The KPZ fixed point was first defined in \cite{matetski2016kpz} as a Markov process on UC by specifying its transition kernel through explicit formulas. This was later shown to be equivalent to \eqref{e.fixed point via landscape}, which arises as a scaling limit of solvable last passage percolation models, in \cite[Corollary~4.2]{nica2020one}. Observe that a natural coupling of $\smash{\mf h(\mf h^{(j)}_0; y, t)}$ for different initial conditions $\smash{\mf h^{(j)}_0}$ is obtained by using the same directed landscape $\mc L$ in the variational formula.

\subsection{Main result: Coupled KPZ fixed point convergence}\label{s.main result}

The following is the assumption on the sequence of initial height functions for ASEP or S6V under which we prove convergence to the KPZ fixed point.

\begin{assumption}\label{as.master assumption}
Let $h_0^\varepsilon:\Z\to\Z$ be a sequence of height functions indexed by $\varepsilon>0$ for S6V or ASEP. For ASEP, fix $q\in[0,1)$ and $\alpha\in(-1,1)$, and for S6V, fix $q\in[0,1)$, $b^{\shortrightarrow}\in(0,1)$, and $\alpha\in(z,z^{-1})$, with $z=\frac{1-b^{\shortrightarrow}}{1-qb^{\shortrightarrow}}$. Define $\h_0^\varepsilon$ by Definition~\ref{d.rescaled ASEP height function} or \ref{d.rescaled S6V height function} for ASEP and S6V, respectively, with the corresponding choice of parameters.
Assume that $\h^\varepsilon_0\to\h_0$ locally in UC for a function $\h_0:\R\to\R\cup\{-\infty\}$ in UC and that $\h_0^\varepsilon$ satisfies, for some fixed $\lambda>0$ for all $\varepsilon>0$,
\begin{equation}\label{e.initial condition assumption}
\begin{gathered}
\h_0^\varepsilon(x) \leq \lambda(1+|x|) \quad\text{for all } x\in\R\  \text{ and }
\sup_{x\in[-\lambda,\lambda]} \h_0^\varepsilon(x) \geq -\lambda.
\end{gathered}
\end{equation}
\end{assumption}

Here we explain the meaning of Assumption~\ref{as.master assumption}. First, some growth condition of the type of \eqref{e.initial condition assumption} is necessary for the convergence to hold, as without any growth condition the KPZ fixed point would not be well-defined. While the weakest condition to ensure the existence of the KPZ fixed point would allow the initial condition to grow almost as $x^2$, condition \eqref{e.initial condition assumption} is the best known even in the case of the convergence of TASEP to the KPZ fixed point \cite{matetski2016kpz}.

Second, the condition that $\alpha\in(-1,1)$ or $\alpha\in(z,z^{-1})$ in the case of ASEP and S6V, respectively, corresponds to the full \emph{rarefaction fans} of the two models. If an initial particle configuration has macroscopic density $|\mu'(\alpha)|\in(0,1)$ around the origin for some $\alpha$ in the rarefaction fan, then we expect particles to move at macroscopic speed $\alpha$; in other words, $\alpha$ is the slope of the characteristic emanating from 0 for that initial condition. In particular, the full rarefaction fan corresponds exactly to the full range of possible macroscopic densities for the initial configuration.

We also note from \eqref{e.rescaled h for ASEP} and \eqref{e.rescaled h for S6V} that the assumption that $\h_0^\varepsilon\to\h_0$ for an $\h_0$ lying in UC implicitly assumes that the initial particle configuration is a perturbation, on the $\varepsilon^{1/3}$ scale, of a macroscopic density in the hydrodynamic limit corresponding to a speed in the rarefaction fan. The restriction to a speed in the rarefaction fan is necessary to ensure that the quantities appearing in \eqref{e.rescaled h for ASEP} and \eqref{e.rescaled h for S6V} are well-defined (but not sufficient to ensure the convergence in UC). Also, contributions from parts of the system away from the origin will be of lower order and not survive under the rescaling in Definitions~\ref{d.rescaled ASEP height function} or \ref{d.rescaled S6V height function}.

We can now give the precise statement of our main result for ASEP and S6V. Its proof is given in Section~\ref{s.proofs of main theorems}.

\begin{theorem}\label{t.main kpz fixed point convergence}
Fix parameters for ASEP or S6V as in Assumption~\ref{as.master assumption} and let $k\in\N$. Suppose that, for each $i\in\intint{1,k}$, $s_i>0$ and $\smash{h_0^{(i),\varepsilon}}:\Z\to \Z$ is a sequence (indexed by $\varepsilon>0$) of initial height functions for ASEP or S6V satisfying Assumption~\ref{as.master assumption}, and let $\smash{\h_0^{(i)}}$ be the limit of $\smash{\h_0^{(i),\varepsilon}}$ for each $i\in\intint{1,k}$ as $\varepsilon\to 0$. Let $\mc T\subseteq (0,\infty)$ be a countable set. Then, under the basic coupling of the evolutions, for $* = $ ASEP or S6V,
$$\h^{*,\varepsilon}(\h_0^{(i),\varepsilon}, s_i;\bm\cdot, t) \stackrel{d}{\to} \h(\h_0^{(i)}, s_i;\bm\cdot, t) \text{ as } \varepsilon\to 0 \text{ jointly over } t\in\mc T\cap \{s>s_i\} \text{ and } i\in\intint{1,k}$$
under the topology of uniform convergence on compact sets, where the righthand sides are coupled via using the same directed landscape $\cL$ in \eqref{e.fixed point via landscape}.
\end{theorem}

We note that we do not establish temporal tightness of the process in Theorem~\ref{t.main kpz fixed point convergence}, and this was also not done in \cite{aggarwal2024scaling}. It would be interesting to obtain such tightness and upgrade the above convergence to also hold as a continuous function in time.

As a special case, Theorem~\ref{t.main kpz fixed point convergence} also yields convergence for general initial conditions for the colored ASEP and S6V (see, e.g., \cite[Section 2]{aggarwal2024scaling}) by recalling that the colored coupling is the basic coupling in the case that the initial particle configurations are ordered.

Lemma~\ref{l.joint kpz fixed point convergence} ahead explains that to obtain the joint convergence in Theorem~\ref{t.main kpz fixed point convergence}, it is sufficient to prove the marginal convergence for a single initial condition and fixed time $t$. Thus, the proof of Theorem~\ref{t.main kpz fixed point convergence} is reduced to its $k=1$ case, and the bulk of the paper will be devoted to establishing this.

\section{Model preliminaries and properties of the KPZ fixed point}\label{s.preliminaries}

In the first two subsections of this section we collect some useful facts about the two models, ASEP and S6V, including the convergence to the directed landscape and monotonicity properties for the height function under the basic coupling. ASEP is addressed in Section~\ref{s.asep preliminaries} and S6V in Section~\ref{s.s6v preliminaries}. In Section~\ref{s.colored models}, we introduce the colored versions of the models and state a corollary of our main result for them. In Section~\ref{s.fixed point properties} we record an estimate on the location of the maximizer in the definition \eqref{e.fixed point via landscape} of the KPZ fixed point and some consequences.

\subsection{Preliminary facts about ASEP}\label{s.asep preliminaries}

To define the ASEP landscape, which is the prelimiting object that will converge to the directed landscape, we need to introduce a two-parameter height function. For $y\in\Z$, consider the step-initial conditions $h^{\mrm{step}}_{0,y}$ defined by $h^{\mrm{step}}_{0,y}(w) = (y-w)\one_{w\leq y}$ and couple the evolutions of all of them under the ASEP dynamics via the basic coupling. Then for any $0<s<t$, define $h^{\mrm{ASEP}}(\bm\cdot, s;\bm\cdot, t):\Z^2\to\Z$ by 
\begin{equation}\label{e.asep unscaled sheet}
h^{\mrm{ASEP}}(y,s;x,t) = h^{\mrm{ASEP}}(h^{\mrm{step}}_y, s; x,t).
\end{equation}
In words, $h^{\mrm{ASEP}}(y,s;x,t)$ is the number of particles to the right of $x$ at time $t$ when started at time $s$ from the initial condition consisting of a particle placed at every site at or to the left of $y$ and none to the right of it.

\begin{definition}[ASEP landscape and sheet]\label{d.asep sheet}
Fix $q\in[0,1)$ and a velocity $\alpha\in(-1,1)$ in the rarefaction fan. Let $\gamma=1-q$ and recall the scaling coefficients $\mu^{\mrm{ASEP}}$ and $\sigma^{\mrm{ASEP}}$ from \eqref{e.mu sigma ASEP}. For $\varepsilon>0$, we define the \emph{ASEP landscape} $\cL^{\mrm{ASEP}, \varepsilon}: \{(y,s;x,t) \in \R^4 : s<t\}\to\R$, for $x$ and $y$ such that the arguments of $h^{\mrm{ASEP}}$ below are integers, by
\begin{equation}\label{e.rescaled asep definition}
\begin{split}
\MoveEqLeft[8]
\cL^{\mrm{ASEP}, \varepsilon}(y, s; x, t) := \sigma(\alpha)^{-1}\varepsilon^{1/3}\Bigl(\mu(\alpha)2(t-s)\varepsilon^{-1} + \mu'(\alpha)\beta(\alpha)(x-y)\varepsilon^{-2/3}\\
&- h^{\mrm{ASEP}}\bigl(\beta(\alpha)y\varepsilon^{-2/3}, 2\gamma^{-1}\varepsilon^{-1}s; 2\alpha \varepsilon^{-1}(t-s) + \beta(\alpha)x\varepsilon^{-2/3}, 2\gamma^{-1} \varepsilon^{-1}t\bigr)\Bigr);
\end{split}
\end{equation}
the values at all other $x$, $y$ are determined by linear interpolation. The \emph{ASEP sheet} $\S^{\mrm{ASEP},\varepsilon}:\R^2\to\R$ is defined by $\S^{\mrm{ASEP},\varepsilon}(y;x) = \cL^{\mrm{ASEP},\varepsilon}(y,0;x,1)$ for all $x,y\in\R$.
\end{definition}

Next we state the convergence of $\cL^{\mrm{ASEP},\varepsilon}$ to the directed landscape $\cL$ as $\varepsilon\to 0$ for any $q\in[0,1)$ and velocity $\alpha$ in the rarefaction fan. For a set $\mc T$, we define $\mc T^2_{<} := \{(s,t)\in\mc T^2: s<t\}$.

\begin{theorem}[Directed landscape convergence for ASEP, {\cite[Corollary 2.12]{aggarwal2024scaling}}]\label{t.asep airy sheet}
Fix any asymmetry $q \in [0,1)$ and any velocity $\alpha\in(-1,1)$ in the rarefaction fan. Let $\mc T\subseteq [0,\infty)$ be a countable set. Then, as $\varepsilon\to 0$, $\smash{\cL^{\mrm{ASEP},\varepsilon}(\bm\cdot, s; \bm\cdot, t) \stackrel{d}{\to} \cL(\bm\cdot, s;\bm\cdot, t)}$ weakly in $\mc C(\R^2,\R)$ with the topology of uniform convergence on compact sets, jointly over $(s,t)\in\mc T^2_<$.
\end{theorem}

The following records the well-known height monotonicity property of ASEP under the basic coupling. Its proof follows immediately from the definition of the basic coupling.

\begin{lemma}[Height monotonicity of ASEP]\label{l.asep height monotonicity}
Let $H\in\Z$ be an integer and $h_0, h_0':\Z\to\Z$ be height functions such that $h_0(x) +H \geq h_0'(x)$ for all $x\in\Z$. Then, under the basic coupling, almost surely, $h^{\mrm{ASEP}}(h_0, 0; x,t) +H \geq h^{\mrm{ASEP}}(h_0', 0; x,t)$ for all $x\in\Z$ and $t>0$.
\end{lemma}

A useful consequence of Lemma~\ref{l.asep height monotonicity} is that a prelimiting version of \eqref{e.fixed point via landscape} holds with an inequality, which we record next. Its proof is given as a part of the proof of \cite[Corollary~2.12 (3)]{aggarwal2024scaling} in that paper, but we give a self-contained version in Appendix~\ref{app.other proofs} for the benefit of the reader.

\begin{lemma}\label{l.prelimiting variation inequality asep}
Fix $q\in[0,1)$ and $\alpha\in(-1,1)$. Let $h_0^\varepsilon:\Z\to\Z$ be a sequence of height functions for ASEP indexed by $\varepsilon>0$ and let $\h_0^\varepsilon$ be defined as in Definition~\ref{d.rescaled ASEP height function}. For every $\varepsilon>0$, it holds almost surely under the basic coupling that, for all $x\in\R$ and $t>0$,
\begin{align*}
\h^{\mrm{ASEP},\varepsilon}(\h_0^\varepsilon; x,t) \geq \sup_{y\in\R}\left(\h_0^\varepsilon(y) + \cL^{\mrm{ASEP},\varepsilon}(y,0;x,t)\right).
\end{align*}
\end{lemma}

\subsection{Preliminary facts about S6V}\label{s.s6v preliminaries}

In the remainder of this paper we will often parametrize the S6V model by the parameters $q\in[0,1)$ and $z\in(0,1)$ rather than $q$ and $b^{\shortrightarrow}\in(0,1)$. The relation between the two parametrizations is given by $z=(1-b^{\shortrightarrow})/(1-qb^{\shortrightarrow})$ and $b^{\shortrightarrow} = (1-z)/(1-qz)$.

Next we introduce the S6V landscape and state its convergence to the directed landscape as proved in \cite{aggarwal2024scaling}.
Let $y\in\Z$ and consider the step-initial conditions $h^{\mrm{step}}_{0,y}$ defined by $h^{\mrm{step}}_{0,y}(w) = (y-w)\one_{w\geq y}$. We couple the evolutions of the S6V model associated to $h^{\mrm{step}}_{0,y}$ as $y$ varies over $\Z$ by the basic coupling, and define, for $x\in\Z$ and $s,t\in\N$ with $t>s$,
\begin{align}\label{e.s6v colored height function}
h^{\mrm{S6V}}(y,s; x, t) := h^{\mrm{S6V}}(h_{0,y},s; x,t).
\end{align}

\begin{definition}[S6V landscape and sheet]\label{d.s6v sheet}
Fix $q\in[0,1)$, $z\in(0,1)$, and a velocity $\alpha\in(z,z^{-1})$ in the rarefaction fan. Recall the scaling coefficients $\mu^{\mrm{S6V}}$ and $\sigma^{\mrm{S6V}}$ from \eqref{e.mu and sigma}. For $\varepsilon>0$, we define the \emph{S6\kern-0.04em V landscape} $\cL^{\mrm{S6V}, \varepsilon} : \{(y,s;x,t) \in \R^4 : s<t\} \to \R$, in the case that the arguments of $\hssv$ below are integers, by
\begin{equation}\label{e.rescaled s6v definition}
\begin{split}
\MoveEqLeft[18]
\cL^{\mrm{S6V},\varepsilon}(y, s; x, t) := \sigma(\alpha)^{-1}\varepsilon^{1/3}\Bigl(\hssv\bigl(\beta(\alpha) y \varepsilon^{-2/3}, \floor{\varepsilon^{-1}s}; \alpha \varepsilon^{-1}(t-s)+\beta(\alpha) x\varepsilon^{-2/3}, \floor{\varepsilon^{-1}t}\bigr)\\
&   -\mu(\alpha)\varepsilon^{-1}(t-s) - \mu'(\alpha)\beta(\alpha)(x-y)\varepsilon^{-2/3}\Bigr),
\end{split}
\end{equation}
and when the arguments of $\hssv$ are not integers we define $\cL^{\mrm{S6V},\varepsilon}(x;y)$ by linear interpolation. The \emph{S6\kern-0.04em V sheet} $\S^{\mrm{S6V}, \varepsilon} : \R^2 \to \R$ is defined by $\S^{\mrm{S6V}, \varepsilon}(y;x) = \cL^{\mrm{S6V}, \varepsilon}(y,0;x,1)$ for all $x,y\in\R$.
\end{definition}

\noindent Recall that for a set $\mc T$, we define $\mc T^2_{<} = \{(s,t)\in\mc T^2: s<t\}$.

\begin{theorem}[Directed landscape convergence for S6V, {\cite[Corollary 2.22]{aggarwal2024scaling}}]\label{t.s6v airy sheet}
Fix any asymmetry $q \in [0,1)$, spectral parameter $z\in (0,1)$, and velocity $\alpha\in(z,z^{-1})$ in the rarefaction fan. Let $\mc T\subseteq [0,\infty)$ be a countable set. Then, as $\varepsilon\to 0$, $\smash{\cL^{\mrm{S6V},\varepsilon}(\bm\cdot, s; \bm\cdot, t) \stackrel{d}{\to} \cL(\bm\cdot, s;\bm\cdot, t)}$ weakly in $\mc C(\R^2,\R)$ with the topology of uniform convergence on compact sets, jointly over $(s,t)\in\mc T^2_<$.

\end{theorem}

Theorem~\ref{t.s6v airy sheet} actually differs slightly from \cite[Corollary 2.22]{aggarwal2024scaling} in that the definitions \eqref{e.s6v colored height function} and \eqref{e.rescaled s6v definition} differ slightly from their counterparts in \cite{aggarwal2024scaling}, equations (2.17) and (2.20) there, respectively. We give the simple proof of Theorem~\ref{t.s6v airy sheet} given \cite[Corollary 2.22]{aggarwal2024scaling} in Appendix~\ref{app.other proofs}.

Unlike ASEP, height monotonicity does not hold in an exact form for S6V. Instead, an approximate form holds, which we record next. It is a straightforward consequence of a similar approximate height monotonicity for systems with finitely many arrows proved as \cite[Proposition D.3]{aggarwal2024scaling}  along with a finite speed of discrepancy estimate, and so we defer its proof to Appendix~\ref{app.other proofs}.

\begin{lemma}[Approximate height monotonicity of S6V]\label{l.s6v approximate height monotonicity}
Fix $q\in[0,1)$, $b^{\shortrightarrow}\in(0,1)$, and $N\in\N$. Let $H\in\Z$ and let $\smash{h_0^{(1)}, h_0^{(2)}}:\Z\to\Z$ be two height functions such that $\smash{h_0^{(1)}(x)} +H \geq \smash{h_0^{(2)}(x)}$ for all $x\in\intint{-2N,2N}$. There exist absolute constants $C, c>0$ such that, for $t\in\intint{1,\frac{1}{2}(1-b^{\shortrightarrow})N}$ and under the basic coupling, the following holds. For any $(\log N)^2\leq M \leq N$, with probability at least $1-C\exp(-cM)$, $\hssv(h_0^{(1)}; x,t) +H \geq \hssv(h_0^{(2)}; x,t) - M$ for all $x\in\intint{-N,N}$.
\end{lemma}

Next we record the S6V analog of Lemma~\ref{l.prelimiting variation inequality asep}. Here, because height monotonicity holds only in an approximate sense and on a high probability event, the same is true of the variational inequality. Its proof is also deferred to Appendix~\ref{app.other proofs}.

\begin{lemma}\label{l.prelimiting variation inequality s6v}
Fix $q\in[0,1)$, $z\in(0,1)$, $\alpha\in(z,z^{-1})$, and $t>0$. Let $h_0^\varepsilon:\Z\to\Z$ be a sequence of height functions for S6V indexed by $\varepsilon>0$ and let $\h_0^\varepsilon$ be defined as in Definition~\ref{d.rescaled S6V height function}. There exist positive constants $C,c$ such that, for every $\varepsilon>0$, with probability at least $1-C\exp(-c\varepsilon^{-1/6})$ under the basic coupling, for all $x\in [-\varepsilon^{-1}, \varepsilon^{-1}]$,
\begin{align*}
\h^{\mrm{S6V},\varepsilon}(\h_0^\varepsilon; x,t) \geq \sup_{y\in [-\varepsilon^{-1}, \varepsilon^{-1}]}\left(\h_0^\varepsilon(y) + \cL^{\mrm{S6V},\varepsilon}(y,0;x,t)\right) - \sigma(\alpha)^{-1}\varepsilon^{1/6}.
\end{align*}
\end{lemma}

\noindent The above estimate has not been optimized, i.e., quantities  like $\varepsilon^{-1}$ and $\varepsilon^{1/6}$ have no significance.

The following statement upgrades marginal distributional convergence of a given initial condition at a fixed time  to joint distributional convergence across multiple initial conditions and multiple times, thus reducing Theorem~\ref{t.main kpz fixed point convergence} to showing convergence for a single initial condition and time. As its proof is essentially the same as that of \cite[Corollary 2.12 (3)]{aggarwal2024scaling}, we defer it to Appendix~\ref{app.other proofs}. 

\begin{lemma}[Joint KPZ fixed point convergence from marginal for ASEP \& S6V]\label{l.joint kpz fixed point convergence}
Fix parameters for ASEP or S6V as in Assumption~\ref{as.master assumption}. Suppose for any sequence of initial height functions $h_0^{\varepsilon}:Z\to\Z$ such that $\h_0^{\varepsilon}\to\h_0$ locally in UC for a UC function $\h_0$, it holds that, for $*=$ ASEP or S6V,
$\h^{*,\varepsilon}(\h_0^{\varepsilon}; \bm\cdot, 1) \stackrel{d}{\to} \h(\h_0; \bm\cdot, 1).$

Then the following holds. Let $k\in\N$ and fix a countable set $\mc T\subseteq (0,\infty)$, and suppose $s_i>0$ and $(h_0^{(i),\varepsilon})_{i=1}^{k}$ are initial height functions for ASEP or S6V such that, for each $i\in\intint{1,k}$ and as $\varepsilon\to 0$,
$\h_0^{(i),\varepsilon} \to\h_0^{(i)}$
locally in UC for some $\h^{(i)}_0\in\mrm{UC}$.
 Then, if the evolutions are coupled across $i\in\intint{1,k}$ via the basic coupling and $\h(\h_0^{(i)}, s_i; \bm\cdot, t)$ are coupled across $i\in\intint{1,k}$ via the same directed landscape in \eqref{e.fixed point via landscape}, for $*=$ ASEP or S6V,
 \begin{align*}
  \h^{*,\varepsilon}(\h_0^{\varepsilon,(i)}, s_i; \bm\cdot ,t) \stackrel{d}{\to} \h(\h_0^{(i)}, s_i; \bm\cdot ,t)
  \end{align*}
   holds jointly across $i\in\intint{1,k}$ and $t\in\mc T$ in the topology of uniform convergence on compact sets.
\end{lemma}

\subsection{Colored models}\label{s.colored models}

In our analysis we will need to work with colored version of ASEP and S6V, and we introduce them now. Briefly, in the colored models, each particle/arrow in the system has an associated ``color'', which is an integer and can be thought of as its priority. The evolution of the system is as described in Sections~\ref{s.asep definition} and \ref{s.s6v definition}, except that a particle or arrow of a given color treats those of lower color as being not present.

\begin{remark}\label{r.colroed to basic coupling}
 Both colored ASEP and colored S6V corresponds to the basic coupling of multiple initial particle configurations which are ordered for their uncolored counterparts, i.e., $\smash{\eta_0^{(1)}, \ldots, \eta_0^{(k)}}$ such that $\eta_0^{(1)}(x)\leq \ldots \leq \eta_0^{(k)}(x)$ for all $x\in\Z$; in this case the color of a particle is the number of initial conditions it is present in.
 \end{remark} 

\subsubsection{Colored ASEP} Like its uncolored counterpart, colored ASEP has a single parameter $q\in[0,1)$. Its dynamics are as follows. A particle configuration is given by a function $\eta_0:\Z\to\Z\cup\{-\infty\}$, where $\eta_0(x)$ denotes the color of the particle at $x$, with $-\infty$ denoting no particle. As in uncolored ASEP, each site has associated with it two independent Poisson clocks, $\xi_x^{\mrm{L}}$ and $\xi_x^{\mrm{R}}$, or rate $q$ and $1$, respectively. At a time when $\xi_x^{\mrm{L}}$ rings, if a particle is at $x$, it attempts to jump one site to the left, and if at time that $\xi_x^{\mrm{R}}$ rings a particle is at $x$, it attempts to jump one site to the right. The attempt succeeds if the particle at the target site is of lower color than at $x$, in which case the two particles swap positions, and fails otherwise, in which case nothing happens.

\begin{definition}[Colored height function]
A \emph{colored height function} is a function $h_0:\Z^2\to\Z$ such that $h_0(x,y) - h_0(x,y+1) \in \{0,1\}$ and $h_0(x,y) - h_0(x+1,y) \in \{0,1\}$ for all $x,y\in\Z$.
\end{definition}

Given a colored height function $h_0:\Z^2\to\Z$, the associated initial particle configuration $\eta_0:\Z\to\Z\cup\{-\infty\}$ is given for all $y\in\Z$ by $\eta_{0}(y) = \min\{x: h_0(x,y) - h_0(x+1,y) = 1\}$, with $\min \emptyset = -\infty$. Informally, $h_0(x,y)$ is, up to centering (which may be infinite, in which case this description is only heuristic), the number of particles of color at least $x$ which are strictly to the right of $y$ at time $0$. 

We denote the colored height function at a time $t>0$ by $(x,y)\mapsto h^{\mrm{ASEP}}(h_0; x,y,t)$ and define it now. The evolution of the height function under the colored ASEP dynamics is similar to that described for uncolored ASEP in Section~\ref{s.asep height function}. If at time $t$ a particle of color $x$ jumps from site $y$ to $y+1$, then $h^{\mrm{ASEP}}(h_0; w,y,t) = h^{\mrm{ASEP}}(h_0; w,y,t)+1$ for all $w\leq x$, and $h^{\mrm{ASEP}}(h_0; w,z,t) = h^{\mrm{ASEP}}(h_0; w,z,t)$ for all $(w,z)$ such that $w>x$ or $z\neq y$. If at time $t$ a particle of color $x$ jumps from site $y$ to $y-1$, then $h^{\mrm{ASEP}}(h_0; w,y-1,t) = h^{\mrm{ASEP}}(h_0; w,y-1,t)-1$ for all $w\leq x$, and $h^{\mrm{ASEP}}(h_0; w,z,t) = h^{\mrm{ASEP}}(h_0; w,z,t)$ for all $(w,z)$ such that $w>x$ or $z\neq y-1$.

\subsubsection{Colored S6V}
As in uncolored S6V as described in Section~\ref{s.s6v definition}, colored S6V is determined by parameters $q\in[0,1)$ and $b^{\shortrightarrow}\in(0,1)$ and consists of arrow configurations $(a_v,i_v; b_v,j_v)$ satisfying arrow conservation across vertices $v\in\Z_{\geq 1}\times\Z$. Here, $a_v,i_v,b_v,j_v\in\Z\cup\{-\infty\}$ represent the colors of the vertically incoming, horizontally incoming, vertically outgoing, and horizontally outgoing arrows, respectively, at the vertex $v$ (with $-\infty$ representing an absence of an arrow), and arrow conservation means $\{a_v,i_v\} = \{b_v,j_v\}$ as multisets. The evolution of the system is as described in Section~\ref{s.s6v dynamics} with weights as in Figure~\ref{f.colored R weights}.

\begin{figure}[h]
\begin{tikzpicture}[scale=0.8]
\newcommand{\thedim}{4.2}
\begin{scope}[scale=0.65]

      \draw[step=\thedim] (0,0) grid (5*\thedim, \thedim);

      \draw[->, line width=1.2pt, red] (0.5*\thedim, 0.25*\thedim) -- ++(0, 0.5*\thedim);
      \draw[->, line width=1.2pt, red] (0.25*\thedim, 0.5*\thedim) -- ++(0.5*\thedim, 0);
      \node[scale=0.8, anchor=west] at (0.75*\thedim+0.1, 0.5*\thedim) {$i$};
      \node[scale=0.8, anchor=south] at (0.5*\thedim, 0.75*\thedim+0.1) {$i$};
      \node[scale=0.8, anchor=east] at (0.25*\thedim-0.1, 0.5*\thedim) {$i$};
      \node[scale=0.8, anchor=north] at (0.5*\thedim, 0.25*\thedim-0.1) {$i$};

      \draw[->, line width=1.2pt, blue] (1.5*\thedim, 0.25*\thedim) -- ++(0, 0.5*\thedim);
      \draw[->, line width=1.2pt, red] (1.25*\thedim, 0.5*\thedim) -- ++(0.5*\thedim, 0);
      \node[scale=0.8, anchor=west] at (1.75*\thedim+0.1, 0.5*\thedim) {$i$};
      \node[scale=0.8, anchor=south] at (1.5*\thedim, 0.75*\thedim+0.1) {$j$};
      \node[scale=0.8, anchor=east] at (1.25*\thedim-0.1, 0.5*\thedim) {$i$};
      \node[scale=0.8, anchor=north] at (1.5*\thedim, 0.25*\thedim-0.1) {$j$};

      \draw[->, line width=1.2pt, red] (2.5*\thedim, 0.25*\thedim) -- ++(0, 0.5*\thedim);
      \draw[->, line width=1.2pt, blue] (2.25*\thedim, 0.5*\thedim) -- ++(0.5*\thedim, 0);
      \node[scale=0.8, anchor=west] at (2.75*\thedim+0.1, 0.5*\thedim) {$j$};
      \node[scale=0.8, anchor=south] at (2.5*\thedim, 0.75*\thedim+0.1) {$i$};
      \node[scale=0.8, anchor=east] at (2.25*\thedim-0.1, 0.5*\thedim) {$j$};
      \node[scale=0.8, anchor=north] at (2.5*\thedim, 0.25*\thedim-0.1) {$i$};

      \draw[->, line width=1.2pt, red] (3.25*\thedim, 0.5*\thedim) -- ++(0.25*\thedim, 0) -- ++(0,0.25*\thedim);
      \draw[->, line width=1.2pt, blue] (3.5*\thedim, 0.25*\thedim) -- ++(0, 0.25*\thedim) -- ++(0.25*\thedim, 0);
      \node[scale=0.8, anchor=west] at (3.75*\thedim+0.1, 0.5*\thedim) {$j$};
      \node[scale=0.8, anchor=south] at (3.5*\thedim, 0.75*\thedim+0.1) {$i$};
      \node[scale=0.8, anchor=east] at (3.25*\thedim-0.1, 0.5*\thedim) {$i$};
      \node[scale=0.8, anchor=north] at (3.5*\thedim, 0.25*\thedim-0.1) {$j$};

      \draw[->, line width=1.2pt, blue] (4.25*\thedim, 0.5*\thedim) -- ++(0.25*\thedim, 0) -- ++(0,0.25*\thedim);
      \draw[->, line width=1.2pt, red] (4.5*\thedim, 0.25*\thedim) -- ++(0, 0.25*\thedim) -- ++(0.25*\thedim, 0);
      \node[scale=0.8, anchor=west] at (4.75*\thedim+0.1, 0.5*\thedim) {$i$};
      \node[scale=0.8, anchor=south] at (4.5*\thedim, 0.75*\thedim+0.1) {$j$};
      \node[scale=0.8, anchor=east] at (4.25*\thedim-0.1, 0.5*\thedim) {$j$};
      \node[scale=0.8, anchor=north] at (4.5*\thedim, 0.25*\thedim-0.1) {$i$};

      \draw[xstep=\thedim, ystep=0.35*\thedim] (0,0) grid (5*\thedim, -0.35*\thedim);

      \newcommand{\they}{-0.175*\thedim}

      \node[scale=1] at (0.5*\thedim, \they) {$1$};
      \node[scale=1] at (1.5*\thedim, \they) {$qb^{\shortrightarrow}$};
      \node[scale=1] at (2.5*\thedim, \they) {$b^{\shortrightarrow}$};
      \node[scale=1] at (3.5*\thedim, \they) {$1-qb^{\shortrightarrow}$};
      \node[scale=1] at (4.5*\thedim, \they) {$1-b^{\shortrightarrow}$};

     \end{scope}
\end{tikzpicture}
\caption{The weights for colored S6V written for colors $i$ and $j$ satisfying $i < j$.}
\label{f.colored R weights}
\end{figure}

\begin{remark}
 We note for future use that one can also ``merge'' intervals of colors in a colored system of both ASEP or S6V to obtain a colored system with fewer colors; this projection is preserved by the dynamics, in that evolving the original system for a given amount of time and merging the colors is the same as evolving the merged system directly.
\end{remark}

Next consider a colored height function $h_0:\Z^2\to\Z$. This is a valid initial condition for colored S6V, with the associated initial particle configuration $\eta_0:\Z\to\Z\cup\{-\infty\}$ given for all $y\in\Z$ by $\eta_{0}(y) = \min\{x: h_0(x,y) - h_0(x+1,y) = 1\}$, with $\min \emptyset = -\infty$. As in colored ASEP, informally, $h_0(x,y)$ is, up to centering (which may be infinite, in which case this description is only heuristic), the number of particles of color at least $x$ which are horizontally entering the infinite ray formed by vertices $(1,z)$ over $z>y$. 

We denote the colored height function at a time $t\in\N$ by $(x,y)\mapsto h^{\mrm{S6V}}(h_0; x,y,t)$ and define it iteratively by $h^{\mrm{S6V}}(h_0; x,y,0) = h_0(x,y)$ for all $x,y\in\Z$, and, for $t\in\N$, by
\begin{align*}
h^{\mrm{S6V}}(h_0; x,y,t) = h^{\mrm{S6V}}(h_0; x,y,t-1) + \one_{b_{(t,y)}\geq x}.
\end{align*}

\subsubsection{Corollary of main result for colored models}

For a colored height function, we define the set of colors in $h_0$ by $\{x: y\mapsto h_0(x,y) \text{ is not constant}\}$, and the number of colors in $h_0$ by the cardinality of this set. A finite color height function is a colored height function $h_0:\Z^2\to\Z$ that contains only finitely many colors. In the below, given a sequence of colored height function $h_0^\varepsilon:\Z^2\to\Z$ indexed by $\varepsilon$ and fixing parameters for ASEP or S6V as in Assumption~\ref{as.master assumption}, we define $\h_0^\varepsilon(x,\bm\cdot)$ for each $x\in\Z$ by applying the transformation \eqref{e.rescaled h for ASEP} or \eqref{e.rescaled h for S6V} (for ASEP or S6V, respectively) to the function $y\mapsto h_0^\varepsilon(x,y)$, and analogously for $\h^{\mrm{S6V},\varepsilon}(\h_0^\varepsilon; x,y,t)$ and its ASEP counterpart; in other words, no rescaling is done on the $x$ argument.

By the observation in Remark~\ref{r.colroed to basic coupling} about the relation of colored models and evolution under the basic coupling for multiple initial conditions for the uncolored model, we obtain the following result on the scaling limit of colored ASEP and S6V with general initial condition as an immediate corollary of Theorem~\ref{t.main kpz fixed point convergence}.

\begin{corollary}
Fix parameters for colored ASEP or colored S6V as in Assumption~\ref{as.master assumption} and $k\in\N$. Suppose that $\smash{h_0^{\varepsilon}}:\Z^2\to \Z$ is a sequence (indexed by $\varepsilon>0$) of initial colored height functions for colored ASEP or S6V satisfying Assumption~\ref{as.master assumption}, and suppose each $h_0^\varepsilon$ contains $k$ colors, denoted $x_1, \ldots, x_k$. Let $\smash{\h_0}(x_i, \bm\cdot)$ be the limit of $\smash{\h_0^{\varepsilon}}(x_i, \bm\cdot)$ for each $i\in\intint{1,k}$ as $\varepsilon\to 0$. Let $\mc T\subseteq (0,\infty)$ be a countable set. Then, under the colored evolutions, for $* = $ ASEP or S6V,
$$\h^{*,\varepsilon}(\h_0^{\varepsilon}; x_i, \bm\cdot, t) \stackrel{d}{\to} \h(\h_0; x_i, \bm\cdot, t) \text{ as } \varepsilon\to 0 \text{ jointly over } t\in\mc T \text{ and } i\in\intint{1,k}$$
under the topology of uniform convergence on compact sets, where the righthand sides are coupled via using the same directed landscape $\cL$ in \eqref{e.fixed point via landscape}.
\end{corollary}

\subsection{KPZ fixed point properties}\label{s.fixed point properties}

The following is a statement about the localization of the argmax in the definition \eqref{e.fixed point via landscape} of the KPZ fixed point. Its proof is fairly standard using well-known properties of $\cL$ and so is deferred to Appendix~\ref{app.other proofs}.

\begin{lemma}[Maximizer control in KPZ fixed point definition]\label{l.argmax control}
Suppose $\h_0:\R\to\R\cup{-\infty}$ satisfies \eqref{e.initial condition assumption} for some $\lambda>0$. Then for any $J>0$ and $t>0$ there exist positive constants $C$ and $c$ depending on $\lambda, J$, and $t$ such that, for all $M>0$,
\begin{align*}
\P\left(\argmax_{z\in\R}\bigl(\h_0(z) + \cL(z,0;y,t)\bigr) \in [-M,M]  \ \ \forall y\in[-J,J]\right) \geq 1- C\exp(-cM^{3/2}).
\end{align*}
\end{lemma}

Next we give two corollaries of Lemma~\ref{l.argmax control}. The first asserts that if two initial conditions for the KPZ fixed point agree on a large interval, then the fixed points (with the coupling via the directed landscape as in \eqref{e.fixed point via landscape}) agree also at a fixed time on a smaller interval with high probability. It is in fact an immediate consequence of the variational formula \eqref{e.fixed point via landscape} for the fixed point and Lemma~\ref{l.argmax control}, so its proof is omitted.

\begin{corollary}[Agreement of KPZ fixed points]\label{c.agreement of fixed points}
Let $\h_0:\R\to\R\cup\{-\infty\}$ and satisfy \eqref{e.initial condition assumption} for some fixed $\lambda>0$ and let $J>0$. Then there exist $M_0 = M_0(\lambda,J)>0$ and $c = c(\lambda,J)>0$ such that the following holds. Suppose $\tilde \h_0:\R\to\R\cup\{-\infty\}$ is such that, for some $M>M_0$, $\tilde \h_0(y) = \h_0(y)$ for all $y\in[-M,M]$ and
\begin{align*}
 \tilde\h_0(y) \leq \lambda|y| + M \text{ for all }|y|>M.
 \end{align*}
Then with probability at least $1-\exp(-cM^{3/2})$, for all $|x|\leq J$,
\begin{align*}
\sup_{y\in \R}\Bigl(\h_0(y) + \cL(y,0;x,1)\Bigr) = \sup_{y\in\R}\left(\tilde\h_0(y) + \cL(y,0;x,1)\right).
\end{align*}
\end{corollary}

To state the second corollary we introduce some notation. For any upper semi-continuous $f:\R\to\R\cup\{-\infty\}$, $t>0$, $I>0$, and $y\in\R$, let 
\begin{align}\label{e.restricted supremum notation}
Q_t^{\mrm{S6V},\varepsilon}(f, I, y) = \sup_{z\in[-I,I]}\left(f(z) + \cL^{\mrm{S6V}, \varepsilon}(z,0; y,t)\right),
\end{align}
and the analogous definition for $Q_t^{\mrm{ASEP},\varepsilon}$. Similarly, we define $Q_t(\h_0, I, y) := \sup_{x\in[-I,I]}(\h_0(z) + \cL(z,0;y,t))$. Observe that all these quantities are non-decreasing in $I$.

\begin{corollary}\label{c.maximizer location claim}
Let $\rho, \delta, J,\lambda$, $t$, and $R$ be positive. Suppose $\{\h_0^\varepsilon\}_{\varepsilon>0}$ is a sequence of continuous functions satisfying \eqref{e.initial condition assumption} for all $\varepsilon>0$ with the fixed value of $\lambda$ and which converges to a continuous function $\h_0:\R\to\R$. Then there exist $\varepsilon_0 = \varepsilon_0(\rho, \delta, J,\lambda, t, R, \h_0, \{\h_0^\varepsilon\}_{\varepsilon>0})>0$ and  $M_0=M_0(\rho,J,\lambda, t)$ such that, for all $0<\varepsilon<\varepsilon_0$ and $M>M_0$,
\begin{align*}
\P\Bigl(Q_t^{\mrm{S6V},\varepsilon}(\h_0^\varepsilon, M, y) \geq Q_t^{\mrm{S6V},\varepsilon}(\h_0^\varepsilon, R, y) - \delta \ \ \forall y\in[-J,J]\Bigr) > 1-\rho.
\end{align*}
The same also holds for $Q_t^{\mrm{ASEP},\varepsilon}$.
\end{corollary}

\begin{proof}
We will use $Q^\varepsilon_t$ to denote $Q^{\mrm{S6V},\varepsilon}_t$ and $Q^{\mrm{ASEP},\varepsilon}_t$ so as to give a simple proof with unified notation. By the monotonicity noted after \eqref{e.restricted supremum notation}, it is sufficient to prove that, for some $M_0 = M_0(\rho, J, \lambda, t)$,
\begin{align*}
\lim_{\varepsilon\to 0} \P\Bigl(Q_t^\varepsilon(\h_0^\varepsilon, M_0, y) \geq Q_t^\varepsilon(\h_0^\varepsilon, R, y) - \delta \ \ \forall y\in[-J,J]\Bigr) > 1 - \rho.
\end{align*}
Observe that by the weak convergence of $\cL^{\mrm{S6V},\varepsilon}$ and $\cL^{\mrm{ASEP},\varepsilon}$ to $\cL$ (Theorems~\ref{t.asep airy sheet} and \ref{t.s6v airy sheet}) and of $\h_0^\varepsilon$ to $\h_0$, both on compact sets, the Portmanteau theorem yields that the lefthand side equals
\begin{align*}
\P\Bigl(Q_t(\h_0, M_0, y) \geq Q_t(\h_0, R, y) - \delta \ \ \forall y\in[-J,J]\Bigr) \geq \P\Bigl(Q_t(\h_0, M_0, y) \geq Q_t(\h_0, R, y) \ \ \forall y\in[-J,J]\Bigr),
\end{align*}
where the inequality is simply due to the inclusion of the event on the righthand side  in the event on the lefthand side. The latter probability equals
\begin{align*}
\P\left(\argmax_{z\in\R}\bigl(\h_0(z) + \cL(z,0;y,t)\bigr) \in [-M_0,M_0]  \ \ \forall y\in[-J,J]\right).
\end{align*}
Thus the proof is complete by Lemma~\ref{l.argmax control}.
\end{proof}

\section{KPZ fixed point convergence under stationary initial condition}\label{s.convergence under bernoulli}

In this section we prove the convergence of ASEP and S6V started from stationary, i.e., Bernoulli, initial condition to the KPZ fixed point. The precise statement is given as Proposition~\ref{p.bernoulli convergence}, and its proof appears in Section~\ref{s.process convergence for bernoulli proof}. Here and below, for $\rho\in(0,1)$, Bernoulli($\rho$) refers to the distribution on $\{0,1\}$ which assigns probability $\rho$ to $1$ and $1-\rho$ to $0$ and Bernoulli($\rho$) initial condition refers to the initial condition for S6V or ASEP in which the initial particle/arrow configuration $\eta_0 = \{\eta_0(x)\}_{x\in\Z}$ is defined as an \iid collection of Bernoulli($\rho$) random variables for all $x\in\Z$.

\begin{proposition}[Convergence to KPZ fixed point under Bernoulli]\label{p.bernoulli convergence}
Fix $t>0$, $q\in[0,1)$, $z\in(0,1)$ and $\alpha\in(z,z^{-1})$. Let $\lambda\in\R$ and define $\rho_{\varepsilon,\lambda} := |(\mu^{\mrm{S6V}})'(\alpha)| + \lambda \sigma^{\mrm{S6V}}(\alpha)\beta^{\mrm{S6V}}(\alpha)^{-1}\varepsilon^{1/3}$. 
 Let $h_0^\varepsilon$ be the height function associated to the Bernoulli($\rho_{\varepsilon,\lambda}$) initial condition with $h^\varepsilon_0(0) = 0$. Let $B_\lambda:\R\to\R$ equal a rate 2 two-sided Brownian motion with drift $\lambda$ started at $(0,0)$. Then $(\h_0^\varepsilon, \h^{\mrm{S6V},\varepsilon}(\h_0^\varepsilon;\bm\cdot, t))$ converges weakly to $(B_\lambda,\h(B_\lambda; \bm\cdot, t))$ as $\varepsilon\to0$ in the topology of uniform convergence on compact sets.

 The same holds for ASEP, for fixed $t>0$, $q\in[0,1)$, $\alpha\in (-1,1)$, $\lambda\in\R$, and $\rho_{\varepsilon,\lambda} := |(\mu^{\mrm{ASEP}})'(\alpha)| + \lambda\sigma^{\mrm{ASEP}}(\alpha)\beta^{\mrm{ASEP}}(\alpha)^{-1}\varepsilon^{1/3}$.
\end{proposition}

The tightness implicit in Proposition~\ref{p.bernoulli convergence} for the height function at rescaled time $t$ follows from the stationarity of the product measure on the particle configuration. Thus the main task is to prove that all subsequential limits are given by the KPZ fixed point started from a drifted Brownian initial condition. For this, given the prelimiting one-sided inequality from Lemmas~\ref{l.prelimiting variation inequality asep} and \ref{l.prelimiting variation inequality s6v}, it suffices to establish that the limiting one-point distributions coincide, by an argument closely following that of \cite[Corollary 2.12]{aggarwal2024scaling}. The one-point convergence is available in the literature and we collect them in the next subsection.

\subsection{One-point convergence under Bernoulli initial condition}\label{s.one-point convergence bernoulli}

The one-point distribution of $\h(\h_0; \bm\cdot,1)$ with $\h_0:\R\to\R$ given by a rate 2 two-sided Brownian motion of drift $\lambda$ started at $(0,0)$ is (up to a shift) the Baik-Rains distribution of parameter $\frac{\lambda}{2}+x$, denoted $\mrm{BR}_{\frac{\lambda}{2}+x}$. As the details of the definition of the Baik-Rains distribution do not play a role in our argument, we do not give it here, but it can be found in \cite[eq. (1.20)]{ferrari2006scaling} (where $F_w$ is the distribution function of $\mrm{BR}_w$); the original definition in a different form is from \cite[Definition 3]{baik2000limiting}.

\begin{lemma}[One-point distribution of KPZ fixed point under stationary]\label{l.bbp for fixed point}
Let $\lambda\in\R$ and let $\h_0:\R\to\R$ be given by a rate 2 two-sided Brownian motion of drift $\lambda$ started at $(0,0)$. Then for any $x\in\R$, $\h(\h_0; x,1)$ has the law of $X_{x + \frac{1}{2}\lambda} + \lambda x + \frac{1}{4}\lambda^2$, where $X_{x + \frac{1}{2}\lambda}$ is a random variable distributed as $\mrm{BR}_{x + \frac{1}{2}\lambda}$.
\end{lemma}

\begin{proof}
By the convergence of the rescaled TASEP (i.e., ASEP with $q=0$) height function under general initial condition to the KPZ fixed point \cite[Theorem 3.13]{matetski2016kpz}, it is enough to identify the one-point distributional limit of the rescaled TASEP height function under Bernoulli initial condition with $\alpha=0$. We start with the case where the drift $\lambda=0$, so the particle configuration is given by \iid Bernoulli($\frac{1}{2}$). By \cite[Theorem~1.6]{aggarwal2018current} (for a statement easily translatable to our situation; the TASEP case of the result that we are using was originally proven in the paper of Ferrari-Spohn, see \cite[Theorem 1.2]{ferrari2006scaling}), for every $x\in\R$,
\begin{align*}
\h^{\mrm{TASEP}, \varepsilon}(\h_0^\varepsilon; x, 1) \stackrel{d}{\to} \mrm{BR}_x,
\end{align*}
where $\h^{\mrm{TASEP},\varepsilon}$ is simply $\h^{\mrm{ASEP},\varepsilon}$ with $q=0$ and $h_0^\varepsilon(0) = 0$ is set.
It is easy to check that $\h_0^{\varepsilon}$ converges weakly to rate 2 two-sided Brownian motion started at $(0,0)$. This establishes the result in the case of $\lambda=0$, and the general case follows from \cite[Theorem 4.5 (vi)]{matetski2016kpz}, which gives a symmetry of the KPZ fixed point when an affine function is added to the initial condition.
\end{proof}

With the above lemma, the one-point convergence of the ASEP and S6V height functions under Bernoulli initial conditions is a consequence of results from \cite{aggarwal2018current}.

\begin{lemma}[One-point convergence under Bernoulli]\label{l.one-point convergence for bernoulli}
Fix $t>0$, $q\in[0,1)$, $z\in(0,1)$ and $\alpha\in(z,z^{-1})$. Let $\lambda\in\R$ and define $\rho_{\varepsilon,\lambda} := |(\mu^{\mrm{S6V}})'(\alpha)| + \lambda \sigma^{\mrm{S6V}}(\alpha)\beta^{\mrm{S6V}}(\alpha)^{-1}\varepsilon^{1/3}$. 
 Let $h_0^\varepsilon$ be the height function associated to the Bernoulli($\rho_{\varepsilon,\lambda}$) initial condition with $h^\varepsilon_0(0) = 0$. Let $\h_0:\R\to\R$ equal a rate 2 two-sided Brownian motion with drift $\lambda$ started at $(0,0)$. Then for every $x\in\R$, $\h^{\mrm{S6V},\varepsilon}(\h_0^\varepsilon;x, t)$ converges weakly to $\h(\h_0; x, t)$ as $\varepsilon\to0$.

  The same holds for ASEP, for fixed $t>0$, $q\in[0,1)$, $\alpha\in (-1,1)$, $\lambda\in\R$, and $\rho_{\varepsilon,\lambda} := |(\mu^{\mrm{ASEP}})'(\alpha)| + \lambda\sigma^{\mrm{ASEP}}(\alpha)\beta^{\mrm{ASEP}}(\alpha)^{-1}\varepsilon^{1/3}$.
\end{lemma}

\begin{proof}
The distribution of $\h(\h_0; x, 1)$ is identified in Lemma~\ref{l.bbp for fixed point}, and the convergence to the same of the height function at a fixed point in the cases of ASEP and S6V are respectively provided by \cite[Theorems~1.4 and 1.6]{aggarwal2018current}. The case of general $t>0$ follows by scaling properties of the KPZ fixed point \cite[Theorem 4.5 (i)]{matetski2016kpz}.
\end{proof}

\subsection{Process convergence under Bernoulli initial condition}\label{s.process convergence for bernoulli proof}

Here we give the proof of Proposition~\ref{p.bernoulli convergence}. It will need a simple lemma whose proof we omit.

\begin{lemma}\label{l.X=Y}
Let $X,Y$ be random variables taking values in $\R\cup\{\pm\infty\}$. Suppose $X\geq Y$ almost surely and $\smash{X \stackrel{d}{=} Y}$. Then $X=Y$ almost surely.
\end{lemma}

In the below, for $\rho\in(0,1)$, a Bernoulli($\rho$) random walk is a random walk whose increment distribution assigns probability $\rho$ to $-1$ and $1-\rho$ to $0$, i.e., the increment is the negative of a Bernoulli($\rho$) random variable.

\begin{proof}[Proof of Proposition~\ref{p.bernoulli convergence}]
We give the proof for S6V and outline the changes required for ASEP. By scaling properties of the KPZ fixed point \cite[Theorem 4.5]{matetski2016kpz}, it suffices to assume $t=1$. Let $[-M,M]$ be the interval on which we will show uniform convergence. 

 We recall from Lemma~\ref{l.prelimiting variation inequality s6v} that, under the basic coupling, there exists some deterministic sequence $\{o_\varepsilon(1)\}_{\varepsilon>0}$ such that $\lim_{\varepsilon\to0} o_\varepsilon(1) = 0$ and the following holds. For any $M>0$, for all $\varepsilon>0$ small enough (depending on $M$), with probability at least  $1-o_\varepsilon(1)$, for all $x\in [-\varepsilon^{-1}, \varepsilon^{-1}]$ and $M>0$,
\begin{align}\label{e.prelimiting inequality}
\h^{\mrm{S6V},\varepsilon}(\h_0^\varepsilon; x, 1) \geq \sup_{y\in [-M,M]} \left( \h_0^{\varepsilon}(y) + \S^{\mrm{S6V},\varepsilon}(y;x)\right) - o_\varepsilon(1).
\end{align}

Now, $\eta_0^\varepsilon$ is stationary for the S6V dynamics for all $\varepsilon>0$ (\cite[Lemma A.2]{aggarwal2020limit}), i.e., the S6V height function at any time $s>0$ (and prior to rescaling) is a Bernoulli$(\rho_{\varepsilon})$ random walk up to a height shift. Thus, by Donsker's theorem, $\h^{\mrm{S6V},\varepsilon}(\h_0^\varepsilon; \bm\cdot, 1) - \h^{\mrm{S6V},\varepsilon}(\h_0^\varepsilon; 0, 1)$ converges weakly in the topology of uniform convergence on compact sets to $B_\lambda:\R\to\R$, a rate 2 two-sided Brownian motion with drift $\lambda$. 

In particular, combining with the one-point convergence from Lemma~\ref{l.one-point convergence for bernoulli}, $\h^{\mrm{S6V},\varepsilon}(\h_0^\varepsilon; \bm\cdot, 1)$ is a tight sequence in the topology of uniform convergence on compact sets. 
  By Theorem~\ref{t.asep airy sheet}, it also holds that $\S^{\mrm{S6V},\varepsilon}$ converges weakly in the same topology to the Airy sheet $\S$, and, again by Donsker's theorem, that $\h_0^\varepsilon$ converges weakly to $B_\lambda$. So $(\h_0^\varepsilon, \h^{\mrm{S6V},\varepsilon}(\h_0^\varepsilon; \bm\cdot, 1), \S^{\mrm{S6V},\varepsilon})$ is a tight sequence, and, by moving to a subsequence of $\varepsilon>0$ if necessary, we may assume that $(\h_0^\varepsilon, \h^{\mrm{S6V},\varepsilon}(\h_0^\varepsilon; \bm\cdot, 1), \S^{\mrm{S6V},\varepsilon})$ converges weakly. So by the Skorohod representation theorem, we may assume that we are working on a probability space and have a continuous process $\g:\R\to\R$ such that, almost surely, as $\varepsilon\to 0$ and uniformly on compact sets,
$$\left(\h_0^{\varepsilon}, \h^{\mrm{S6V},\varepsilon}(\h_0^\varepsilon; \bm\cdot, 1), \S^{\mrm{S6V},\varepsilon}\right) \to (B_\lambda, \g, \S);$$
here the components on the lefthand side are coupled via the basic coupling for each $\varepsilon>0$, and, on the righthand side, $B_\lambda$ and $\S$ are independent while $(B_\lambda,\S)$ and $\g$ have an unspecified coupling. In particular, \eqref{e.prelimiting inequality} holds with probability at least $1-o_\varepsilon(1)$ for each $\varepsilon>0$. By taking $\varepsilon\to0$ and then a supremum over $M$, the latter along with Lemma~\ref{l.argmax control} (to ensure the convergence of \eqref{e.prelimiting inequality} is uniform over $x$ in any given compact set) implies that, for all $x\in\R$,
\begin{align}\label{e.limiting lower bound}
\g(x) \geq \sup_{y\in\R}\left(B_\lambda(y) + \S(y;x)\right).
\end{align}

Now, we know by the S6V case of Lemma~\ref{l.one-point convergence for bernoulli} that $\g(x) \stackrel{d}{=} \sup_{y\in\R} (B_\lambda(y) + \S(y;x))$ for every fixed $x\in\Q$. So by applying Lemma~\ref{l.X=Y} and taking a countable intersection we obtain that in fact \eqref{e.limiting lower bound} holds with equality almost surely for all $x\in\Q$ simultaneously. Thus we have obtained that $\g=\h(B_\lambda; x,1)$ for all $x\in\Q$ simultaneously on a probability $1$ event. Since both are continuous functions, they are equal for all $x$. Thus all weak subsequential limits of $(\h_0^\varepsilon,\h^{\mrm{S6V},\varepsilon}(\h_0^\varepsilon; \bm\cdot, 1))$ are $(B_\lambda,\h(B_\lambda; \bm\cdot, 1))$, so we obtain that $(\h_0^\varepsilon, \h^{\mrm{S6V},\varepsilon}(\h_0^\varepsilon; \bm\cdot, 1)) \to (B_\lambda,\h(B_\lambda; \bm\cdot, 1))$ weakly uniformly on compact sets.

The case of ASEP is the same (in fact simpler due to its satisfying exact height monotonicity), modulo a few small changes which we list now. First, Theorem~\ref{t.s6v airy sheet} is replaced by Theorem~\ref{t.asep airy sheet}. Second, the stationarity of Bernoulli($\rho$) for ASEP is a well-known fact. Finally, \eqref{e.prelimiting inequality} holds without the $-o_\varepsilon(1)$ term and with probability 1 for all $z\in\R$ simultaneously. This completes the proof.
\end{proof}

\section{KPZ fixed point convergence with Bernoullis on either side}\label{s.general one-sided initial condition}

In this section we prove convergence to the KPZ fixed point for ASEP and S6V height functions when started from an initial condition which is given by a random walk on either side if one goes out far enough. It will be used to give the proof of Theorem~\ref{t.main kpz fixed point convergence} in Section~\ref{s.proofs of main theorems}. The precise statement is below, and its proof will be given just after its statement.

Fix parameters for ASEP or S6V as in Assumption~\ref{as.master assumption}. Fix $\lambda>0$ and let $\rho_{\varepsilon,\lambda} := |\mu'(\alpha)| + 2\sigma(\alpha)\beta(\alpha)^{-1}\lambda\varepsilon^{1/3}$. Here and in the remainder of the paper, for a real number $M$, define 
\begin{equation}\label{e.Msc}
M^\mrm{sc} = M\beta(\alpha)^{-1},
\end{equation}
where ``sc'' stands for ``scaled''; thus, by Definitions~\ref{d.rescaled ASEP height function} or \ref{d.rescaled S6V height function}, if a discrete height function $h_0^\varepsilon$ is defined on $\intint{-M\floor{\varepsilon^{-2/3}}, M\floor{\varepsilon^{-2/3}}}$, then $\h_0^\varepsilon$ will  be defined on $[-M^\mrm{sc},M^\mrm{sc}]$.

 Let $M>0$ and let $h_0^\varepsilon:\Z\to\Z$ be a sequence of height functions such that the associated particle configuration $\eta_0^\varepsilon:\Z\to\{0,1\}$ is such that $\eta_0(x)\sim \mrm{Bern}(\rho_{\varepsilon,\lambda})$ \iid for all $|x|> M\floor{\varepsilon^{-2/3}}$. Within $\intint{-M\floor{\varepsilon^{-2/3}},M\floor{\varepsilon^{-2/3}}}$ we allow the configuration to be arbitrary so long as $\h_0^\varepsilon$ converges uniformly on $[-M^\mrm{sc},M^\mrm{sc}]$ to a continuous function. Recall the definition of $\h_0^\varepsilon$ as the rescaled version of $h_0^\varepsilon$ from Definitions~\ref{d.rescaled ASEP height function} or \ref{d.rescaled S6V height function} for ASEP and S6V, respectively.

\begin{proposition}[KPZ fixed point convergence with Bernoullis pasted on sides]\label{p.KPZ FP convergence via approximation bernoullis}
Let $t>0$, $q\in[0,1)$, $z\in(0,1)$, and $\alpha\in(z,z^{-1})$.
Fix $\lambda>0$ and $M>0$. Let $h_0^\varepsilon$ be as above; in particular,  $\h_0^\varepsilon$ converges weakly to a continuous function $\h_0: \R\to\R$, in the topology of uniform convergence on compact sets, where $\h_0$ is deterministic on $[-M^\mrm{sc},M^\mrm{sc}]$.
Then $\h^{\mrm{S6V},\varepsilon}(h^\varepsilon_0; \bm\cdot, t) \stackrel{\smash{d}}{\to} \h(\h_0; \bm\cdot, t)$ in the topology of uniform convergence on compact sets. 
The same holds for the ASEP height function for fixed $t>0$, $q\in[0,1)$, and $\alpha\in(-1,1)$.
\end{proposition}

As mentioned in Section~\ref{s.proof ideas}, the idea of the proof of Proposition~\ref{p.KPZ FP convergence via approximation bernoullis} is to approximate the initial condition inside $\intint{-M\floor{\varepsilon^{-2/3}}, M\floor{\varepsilon^{-2/3}}}$ by the stationary configuration i.e., the Bernoulli product measure on the particle configuration, and then make use of Proposition~\ref{p.bernoulli convergence}. This approximation comes down to approximating the limiting height function by a Brownian motion on a compact interval, which is why we assume that $\h_0$ is continuous.

\begin{proof}[Proof of Proposition~\ref{p.KPZ FP convergence via approximation bernoullis}]
We give the proof for S6V as it is analogous for ASEP. Fix any $J>0$ and let $G: \mc C([-J,J],\R)\to\R$ be any bounded uniformly continuous function. It suffices to prove that, as $\varepsilon\to0$,
\begin{align}\label{e.to establish for stationary + approximation}
\E\Bigl[G(\h^{\mrm{S6V},\varepsilon}(\h_0^\varepsilon;\bm\cdot, 1))\Bigr] \to \E\Bigl[G(\h(\h_0;\bm\cdot, 1))\Bigr].
\end{align}

Let $\xi_0^\varepsilon$ be the initial particle configuration given by \iid Bern($\rho_{\varepsilon}$) coupled with $\eta_0^\varepsilon$ such that $\xi_0^\varepsilon(x) = \eta_0^\varepsilon(x)$ for all $|x|>M\floor{\varepsilon^{-2/3}}$. Let $g_0^\varepsilon$ be the height function satisfying $g_0^\varepsilon(0) = h_0^\varepsilon(0)$ and associated to $\xi_0^\varepsilon$ and $\g_0^\varepsilon$ the rescaled version of $g_0^\varepsilon$ from Definition~\ref{d.rescaled S6V height function}. Then, it follows from Proposition~\ref{p.bernoulli convergence}, the assumptions on $\h_0^\varepsilon$, and the coupling of $\h_0^\varepsilon$ and $\g_0^\varepsilon$ that 
\begin{equation}\label{e.joint convergence with h_0}
(\h_0^\varepsilon, \g_0^\varepsilon, \h^{\mrm{S6V},\varepsilon}(\g_0^\varepsilon;\bm\cdot,1))\stackrel{\smash{d}}{\to} (\h_0, B_\lambda, \h(B_\lambda;\bm\cdot, 1))
\end{equation}
in the topology of uniform convergence on compact sets, where $B_\lambda$ is Brownian motion with drift $\lambda$, and $\h_0$ and $B_\lambda$ differ by $\h_0(M^\mrm{sc}) - B_\lambda(M^\mrm{sc})$ on all of $[M^\mrm{sc},\infty)$ and by $\h_0(-M^\mrm{sc}) - B_\lambda(-M^\mrm{sc})$ on all of $(-\infty, -M^\mrm{sc}]$.

Note that $\h_0$ on $[-M^\mrm{sc},M^\mrm{sc}]$ is deterministic. For $\delta>0$, consider the function $F_\delta: \mc C([-M^\mrm{sc},M^\mrm{sc}],\R)\to\R$ given by $f\mapsto \one_{\sup_{x\in[-M^\mrm{sc},M^\mrm{sc}]}|f(x)-\h_0(x)| \leq \delta}$ and observe that the set of discontinuities of $F_\delta$ has probability 0 under the law of $B_\lambda$.
Since  $F_\delta$ is almost surely continuous with respect to the law of $B_\lambda$, it follows that $(f,\tilde f)\mapsto F_\delta(f)G(\tilde f)$ is almost surely continuous with respect to the law of $(B_\lambda, \h(B_\lambda;\bm\cdot, 1))$.  So, the continuous mapping theorem \cite[Theorem 5.27]{Kallenberg} (along with the fact that $(f,\tilde f)\mapsto F_\delta(f)G(\tilde f)$ is bounded) yields that, for all $\delta>0$, as $\varepsilon\to 0$,
\begin{align*}
\E\left[\one_{\sup_{[-M^\mrm{sc},M^\mrm{sc}]}|\g_0^\varepsilon - \h_0|\leq \delta}G(\h^{\mrm{S6V},\varepsilon}(\g_0^\varepsilon;\bm\cdot, 1))\right] \to \E\left[\one_{\sup_{[-M^\mrm{sc},M^\mrm{sc}]}|B_\lambda - \h_0|\leq \delta}G(\h(B_\lambda;\bm\cdot, 1))\right].
\end{align*}

Now, by the Portmanteau theorem and \eqref{e.joint convergence with h_0}, it holds that $\P(\sup_{[-M^\mrm{sc},M^\mrm{sc}]}|\g_0^\varepsilon - \h_0|\leq \delta) \to \P(\sup_{[-M^\mrm{sc},M^\mrm{sc}]}|B_\lambda - \h_0|\leq \delta)$ as $\varepsilon\to0$ for all $\delta>0$, and the latter probability is positive (as the law of Brownian motion assigns positive probability to all open sets). So the previous display may be rewritten, for any $\delta>0$, as
\begin{equation}\label{e.G limit}
\begin{split}
\MoveEqLeft[10]
\lim_{\varepsilon\to 0}\biggl|\E\Bigl[G(\h^{\mrm{S6V},\varepsilon}(\g_0^\varepsilon;\bm\cdot, 1))\midd \sup_{[-M^\mrm{sc},M^\mrm{sc}]}|\g_0^\varepsilon - \h_0|\leq \delta\Bigr]\\
& - \E\Bigl[G(\h(B_\lambda;\bm\cdot, 1))\midd \sup_{[-M^\mrm{sc},M^\mrm{sc}]}|B_\lambda - \h_0|\leq \delta\Bigr]\biggr| = 0,
\end{split}
\end{equation}
where $\E[X\mid \msf A] := \E[X\one_{\msf A}]/\P(\msf A)$ for a random variable $X$ and positive probability event $\msf A$.

To obtain \eqref{e.to establish for stationary + approximation}, we will first argue that, on the events conditioned on in the expectations in \eqref{e.G limit}, we can replace $\g_0^\varepsilon$ by $\h_0^\varepsilon$ and $B_\lambda$ by $\h_0$ in the arguments of $\h^{\mrm{S6V},\varepsilon}$ and $\h$ on the lefthand and righthand sides, respectively, up to a small error. 

Recall from \eqref{e.joint convergence with h_0} that $\h_0$ and $B_\lambda$ are coupled to differ by $\h_0(M^\mrm{sc}) - B_\lambda(M^\mrm{sc})$ on $[M^\mrm{sc},\infty)$ and $\h_0(-M^\mrm{sc}) - B_\lambda(-M^\mrm{sc})$ on $[-M^\mrm{sc},\infty)$. So when $\sup_{[-M^\mrm{sc},M^\mrm{sc}]}|B_\lambda-\h_0|\leq \delta$, it follows that $\sup_{\R}|B_\lambda-\h_0|\leq \delta$. Then monotonicity of $\f\mapsto \h(\f;\bm\cdot, 1)$ when coupled via the landscape (as in \eqref{e.fixed point via landscape}, from which the monotonicity is immediate) yields that, when $\sup_{[-M,M]}|B_\lambda-\h_0|\leq \delta$,
\begin{align}\label{e.B comaprison}
\sup_{\R}\bigl|\h(B_\lambda;\bm\cdot, 1) - \h(\h_0;\bm\cdot, 1)\bigr|\leq \delta.
\end{align}

We also observe that, since $\h_0^\varepsilon\to\h_0$ uniformly on $[-M^\mrm{sc},M^\mrm{sc}]$, for all small enough $\varepsilon$ it holds that $\sup_{[-M^\mrm{sc},M^\mrm{sc}]}|\h_0-\h_0^\varepsilon|\leq \delta$. Thus when $\sup_{[-M,M]}|\g_0^\varepsilon - \h_0|\leq \delta$, for small enough $\varepsilon$ it holds that $\sup_{[-M^\mrm{sc},M^\mrm{sc}]}|\g_0^\varepsilon - \h_0^\varepsilon|\leq 2\delta$; since the increments of $\g_0^\varepsilon$ and $\h_0^\varepsilon$ are the same outside $[-M^\mrm{sc},M^\mrm{sc}]$, we obtain similarly to the case with $B_\lambda$ that
\begin{align*}
 \sup_{\R}\bigl|\g_0^\varepsilon - \h_0^\varepsilon\bigr| \leq 2\delta.
 \end{align*}
So, by approximate height monotonicity (Lemma~\ref{l.s6v approximate height monotonicity}), when $\sup_{[-M^\mrm{sc},M^\mrm{sc}]}|\g_0^\varepsilon-\h_0|\leq \delta$ and the two evolutions are coupled via the basic coupling, for small enough $\varepsilon$,
\begin{align}\label{e.height function comparison}
\sup_{x\in [-J,J]}\Bigl|\h^{\mrm{S6V},\varepsilon}(\g_0^\varepsilon;\bm\cdot, 1) - \h^{\mrm{S6V},\varepsilon}(\h_0^\varepsilon;\bm\cdot, 1)\Bigr| \leq 2\delta + o_\varepsilon(1)
\end{align}
with probability at least $1-o_\varepsilon(1)$. 

Let $\rho>0$. By the uniform continuity of $G$, there exists $\delta = \delta(\rho)>0$ such that, if $\sup_{[-J,J]}|f-\tilde f|\leq 3\delta$, then $|G(f)-G(\tilde f)|\leq \rho$. Returning to \eqref{e.G limit}  and using \eqref{e.B comaprison} and \eqref{e.height function comparison} with this value of $\delta$ yields that
\begin{align*}
\MoveEqLeft[10]
\limsup_{\varepsilon\to 0}\biggl|\E\Bigl[G(\h^{\mrm{S6V},\varepsilon}(\h_0^\varepsilon;\bm\cdot, 1))\midd \sup_{[-M,M]}|\g_0^\varepsilon - \h_0|\leq \delta\Bigr]\\
& - \E\Bigl[G(\h(\h_0;\bm\cdot, 1))\midd \sup_{[-M^\mrm{sc},M^\mrm{sc}]}|B_\lambda - \h_0|\leq \delta\Bigr]\biggr| \leq 2\rho.
\end{align*}
Next we observe that $\g_0^\varepsilon$ and $\h^{\mrm{S6V},\varepsilon}(\h_0^\varepsilon;\bm\cdot, 1)$ are independent, and $B_\lambda$ and $\h(\h_0;\bm\cdot, 1)$ are independent. So the previous display becomes
\begin{align*}
\MoveEqLeft[10]
\limsup_{\varepsilon\to 0}\biggl|\E\Bigl[G(\h^{\mrm{S6V},\varepsilon}(\h_0^\varepsilon;\bm\cdot, 1))\Bigr] - \E\Bigl[G(\h(\h_0;\bm\cdot, 1))\Bigr]\biggr| \leq 2\rho.
\end{align*}
Since $\rho>0$ was arbitrary, we obtain \eqref{e.to establish for stationary + approximation} and complete the proof.
\end{proof}

\section{KPZ fixed point convergence bounded by arbitrary drift} \label{s.proofs of main theorems}

In this section, we will give the proof of Theorem~\ref{t.main kpz fixed point convergence}, but we will first prove an intermediate statement. Recall that Proposition~\ref{p.KPZ FP convergence via approximation bernoullis} yields convergence to the KPZ fixed point for initial conditions whose particle configurations are \iid Bernoulli $\rho_{\varepsilon,\lambda} = |\mu'(\alpha)| + \lambda\sigma(\alpha)\beta(\alpha)^{-1} \varepsilon^{1/3}$ far enough to the left and right. The intermediate statement, Proposition~\ref{p.KPZ FP convergence for V shape bernoulli}, yields convergence to the KPZ fixed point when the initial condition's particle configuration is $\rho_{\varepsilon,-\lambda}$ far enough to the right and $\rho_{\varepsilon,\lambda}$ far enough to the left. This will be used to upper bound the height function from general initial condition satisfying the growth condition in Assumption~\ref{as.master assumption} so as to prove Theorem~\ref{t.main kpz fixed point convergence} in Section~\ref{s.general proof}. 

In the next subsection we give some notation and state and prove Proposition~\ref{p.KPZ FP convergence for V shape bernoulli}.

\subsection{KPZ fixed point convergence with V-drifted Bernoullis pasted}
Fix parameters for S6V or ASEP as in Assumption~\ref{as.master assumption}; in particular, $\alpha$ is fixed and lies in the rarefaction fan for the selected parameters. For $\varepsilon>0$ and $\lambda\in\R$, let $\rho_{\varepsilon,\lambda} = |\mu'(\alpha)| + \lambda\sigma(\alpha)\beta(\alpha)^{-1}\varepsilon^{1/3}$ be as above. Let $\xi^\varepsilon_{+}, \xi^\varepsilon_{-}:\Z\to\{0,1\}$ be initial configurations defined by $\xi^\varepsilon_{\pm}(x) \sim \mrm{Bern}(\rho_{\varepsilon, \pm \lambda})$ independently for each $x\in\Z$.  Next we describe the initial conditions for ASEP and S6V; they are slightly different due to a discrepancy in sign in the definitions \eqref{e.rescaled h for ASEP} and \eqref{e.rescaled h for S6V} of the rescaled height functions, and we start with S6V.

For S6V, let $h_0^\varepsilon:\Z\to\Z$ be an initial height function and $\eta_0^\varepsilon:\Z\to\{0,1\}$ its associated particle configuration. For $M>0$, define the modified particle configuration $\eta_0^{M,\varepsilon}:\Z\to\{0,1\}$ by

\begin{align}\label{e.dominating particle configuration for V}
\eta^{M,\varepsilon}_0(x) = \begin{cases}
\eta_0^\varepsilon(x) & \text{if } |x| \leq M\floor{\varepsilon^{-2/3}},\\
\xi^\varepsilon_{-}(x) & \text{if } x > M\floor{\varepsilon^{-2/3}},\text{ and} \\
\xi^\varepsilon_{+}(x) & \text{if } x < -M\floor{\varepsilon^{-2/3}},
\end{cases}
\end{align}
 and let $h_0^{M,\varepsilon}$ be the associated height function satisfying $h_0^{M,\varepsilon}(0) = h_0^\varepsilon(0)$. 
Suppose $\h_0^\varepsilon$ converges uniformly on compact sets to a continuous function $\h_0:\R\to\R$. Then we observe from Definition~\ref{d.rescaled S6V height function} (and recalling the definition of $M^\mrm{sc}$ from \eqref{e.Msc}) that $\h_0^{M,\varepsilon}$ converges in distribution, in the topology of uniform convergence on compact sets, to $\h_0^{M} : \R\to\R$ given by (see left panel of Figure~\ref{f.dominating height functions})
\begin{align}\label{e.rescaled h_0^M definition}
\h_0^{M}(x) := \begin{cases}
\h_0(x)  & \text{if } |x|\leq M^\mrm{sc},\\
B(x-M^\mrm{sc}) + \lambda (x-M^\mrm{sc}) + \h_0(M^\mrm{sc}) & \text{if } x>M^\mrm{sc}, \text{ and}\\
B(-x+M^\mrm{sc}) - \lambda (x+M^\mrm{sc}) + \h_0(-M^\mrm{sc}) & \text{if } x< -M^\mrm{sc},
\end{cases}
\end{align}
where $B:\R\to\R$ is a rate 2 two-sided Brownian motion.
 Note that we omit $\lambda$ from the notation $h_0^{M,\varepsilon}$, $\h_0^{M,\varepsilon}$, and $\h_0^M$, though it is present in their definitions.

 For ASEP, we define $\eta_0^{M,\varepsilon}:\Z\to\{0,1\}$ by
 \begin{align}\label{e.dominating particle configuration for V ASEP}
 \eta^{M,\varepsilon}_0(x) = \begin{cases}
\eta_0^\varepsilon(x) & \text{if } |x| \leq M\floor{\varepsilon^{-2/3}},\\
\xi^\varepsilon_{+}(x) & \text{if } x > M\floor{\varepsilon^{-2/3}},\text{ and} \\
\xi^\varepsilon_{-}(x) & \text{if } x < -M\floor{\varepsilon^{-2/3}},
\end{cases}
 \end{align}
 i.e., we switch the sides on which $\xi^\varepsilon_{+}$ and $\xi^\varepsilon_-$ appear, and let $h_0^{M,\varepsilon}$ be the associated height function with $h_0^{M,\varepsilon}(0) = h_0^\varepsilon(0)$. Note that due to the extra negative sign in the definition \eqref{e.rescaled h for ASEP} of $\h_0^{M,\varepsilon}$ for ASEP in comparison to that of \eqref{e.rescaled h for S6V}, it still holds that $\h_0^{M,\varepsilon}$ converges in distribution to $\h_0^M$ as in \eqref{e.rescaled h_0^M definition} as $\varepsilon\to 0$ in the topology of uniform convergence on compact sets.

\begin{proposition}\label{p.KPZ FP convergence for V shape bernoulli}
Let $t>0$, $q\in[0,1)$, $z\in(0,1)$, and $\alpha\in(z,z^{-1})$.
Fix $\lambda>0$. Assume $h_0^\varepsilon:\Z\to\Z$ is a sequence of height functions such that $\h_0^\varepsilon$ converges uniformly on compact sets to a continuous function $\h_0: \R\to\R$. Let $\smash{h_0^{M,\varepsilon}}$ be as above.
%
Then there exists a sequence $M_\varepsilon\to\infty$ such that $\h^{\mrm{S6V},\varepsilon}(\h^{M_\varepsilon,\varepsilon}_0; \bm\cdot, t) \stackrel{\smash{d}}{\to} \h(\h_0; \bm\cdot, t)$  in the topology of uniform convergence on compact sets as $\varepsilon\to0$. 
The same holds for the ASEP height function for fixed $t>0$, $q\in[0,1)$, and $\alpha\in(-1,1)$.
\end{proposition}

We will prove Proposition~\ref{p.KPZ FP convergence for V shape bernoulli} by comparison to the case where Bernoullis of the \emph{same} parameter are pasted on both sides as in Proposition~\ref{p.KPZ FP convergence via approximation bernoullis}. The joint coupling of such an initial condition with one as in Proposition~\ref{p.KPZ FP convergence for V shape bernoulli} can be described in terms of colored particles in a colored system (as described in Section~\ref{s.colored models} and Remark~\ref{r.colroed to basic coupling}), and for the comparison  we will need to have control over the movement of particles of certain colors in a colored system. 

 The following result, Lemma~\ref{l.overtaking}, from \cite{drillick2024stochastic} will suffice for the purpose of controlling the movement of particles of lower color (much weaker versions would also suffice). In Lemma~\ref{l.overtaking} and the proof of Proposition~\ref{p.KPZ FP convergence for V shape bernoulli}, we will refer to particles to the left or right of given locations; in both the cases of ASEP and S6V, left and right of a location $y$ refers to particles in a particle configuration at sites $x$ which are smaller or larger than $y$, respectively. In particular, it does \emph{not} refer to the relative locations of arrow trajectories in the S6V system on $\Z_{\geq 1}\times \Z$ as in Figure~\ref{f.colored S6V}.

\begin{lemma}[{\cite[Corollary 3.2]{drillick2024stochastic}}]\label{l.overtaking}
Let $(\eta_t)_t$ be a colored stochastic six-vertex process or a colored ASEP process with the following initial conditions:
\begin{enumerate}
  \item There are some (finitely or infinitely many) particles of color 3 or greater.
  \item There are finitely many color 2 particles.
  \item There is a single color 1 particle, to the left of all color 2 particles.
\end{enumerate}
Let $L_t$ be the number of color 2 particles to the left of the color 1 particle at time $t$. Then conditioned on the paths of the color 3 particles, and the combined paths of the color 1 and 2 particles (i.e., without distinguishing the color), for any $t \geq 0$, $L_t$ is stochastically dominated by a random variable $X\sim\mrm{Geo}(q)$, i.e., $X$ satisfies $\P(X\geq k) = q^{k}$. 
\end{lemma}

\cite[Corollary 3.2]{drillick2024stochastic} is stated only for the stochastic six-vertex model, while we also include ASEP in the above statement. However, as noted after \cite[Theorem 1.6]{drillick2024stochastic}, since the bound has no dependencies on the scale of the system, taking the limit of S6V to ASEP as in \cite{aggarwal2017convergence} yields the same statement for ASEP.

We will also need a finite speed of propagation statement, saying that in time $\varepsilon^{-1}$, a particle will, with very high probability, not move by more than $\varepsilon^{-2}$. As its proof is standard, we defer it to Appendix~\ref{app.other proofs}.

\begin{lemma}\label{l.finite speed of propagation}
Let $(\eta_t)_t$ be a colored S6V process or a colored ASEP process. Let $i$ be a color such that there is a rightmost particle of color $i$. There exists an absolute constant $c>0$ such that, with probability at least $1-\exp(-c\varepsilon^{-1})$, for any $\varepsilon>0$, the rightmost particle of color $i$ travels by at most $\varepsilon^{-2}$ in time $\floor{\varepsilon^{-1}}$.
\end{lemma}

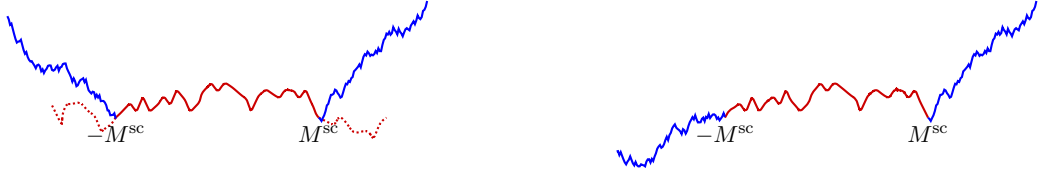
\begin{figure}
\begin{tikzpicture}[scale=0.9,default/.style={very thick,decorate, decoration={random steps,segment length=3pt,amplitude=1pt}}, use fpu reciprocal]

\draw[red!80!black, default, thick, rounded corners=1pt] (0,0) -- ++(0.2,0.2) -- ++(0.1,-0.1) -- ++(0.1, 0.2) -- ++(0.1, -0.2) -- ++(0.2,0.2) -- ++ (0.1, -0.1) -- ++(0.1,0.2) -- ++(0.2, -0.3) -- ++(0.3, 0.4) -- ++(0.1,-0.1) -- ++(0.1,0.1) -- ++(0.3, -0.2) -- ++(0.1, -0.2) -- ++(0.1, 0.2) -- ++(0.1,0.1) -- ++(0.2, -0.1) -- ++(0.1, 0.1) -- ++(0.1,0) -- ++(0.1,-0.1) -- ++(0.1,0.1) -- ++(0.1, -0.2) -- ++(0.1,-0.2);

\draw[red!80!black, default, thick, rounded corners=1.25pt, densely dotted] (0,0) -- ++(-0.2,-0.2) -- ++(-0.1, 0.2) -- ++(-0.2, 0.2) -- ++(-0.2,0) -- ++(-0.1,-0.3) -- ++(-0.1, 0.2) -- ++(-0.1, 0);

\draw[red!80!black, default, thick, rounded corners=1.25pt, densely dotted] (3,0) -- ++(0.2,-0.1) -- ++(0.1, 0.1) -- ++(0.2, -0.2) -- ++(0.2,0) -- ++(0.1,-0.1) -- ++(0.1, 0.2) -- ++(0.1, 0.1);

\bb{80}{0.02}{0.07}{0}{1.7}{thick, blue}{46}{3} 
\bb{80}{0.02}{0.07}{1.5}{0}{thick, blue}{50}{-1.6} 

\node[anchor=north, scale=0.8] at (0,0) {$-M^\mrm{sc}$};
\node[anchor=north, scale=0.8] at (3,0) {$M^\mrm{sc}$};

\begin{scope}[shift={(9,0)}]
\draw[red!80!black, default, thick, rounded corners=1.25pt] (0,0) -- ++(0.2,0.2) -- ++(0.1,-0.1) -- ++(0.1, 0.2) -- ++(0.1, -0.2) -- ++(0.2,0.2) -- ++ (0.1, -0.1) -- ++(0.1,0.2) -- ++(0.2, -0.3) -- ++(0.3, 0.4) -- ++(0.1,-0.1) -- ++(0.1,0.1) -- ++(0.3, -0.2) -- ++(0.1, -0.2) -- ++(0.1, 0.2) -- ++(0.1,0.1) -- ++(0.2, -0.1) -- ++(0.1, 0.1) -- ++(0.1,0) -- ++(0.1,-0.1) -- ++(0.1,0.1) -- ++(0.1, -0.2) -- ++(0.1,-0.2);


\bb{80}{0.02}{0.07}{0}{1.7}{thick, blue}{46}{3} 
\bb{80}{0.02}{0.07}{-0.5}{0}{thick, blue}{50}{-1.6} 

\node[anchor=north, scale=0.8] at (0,0) {$-M^\mrm{sc}$};
\node[anchor=north, scale=0.8] at (3,0) {$M^\mrm{sc}$};
\end{scope}
\end{tikzpicture}
\caption{A depiction of the height functions $\h_0^{M,\varepsilon}$ (left) and $\tilde \h_0^{M,\varepsilon}$ (right). The red portion in the middle of both is $\h_0^\varepsilon$ when restricted to an interval $[-M,M]$ around 0 (and the dotted portions in the left figure are $\h_0$ outside $[-M,M]$).}\label{f.dominating height functions}
\end{figure}

\begin{proof}[Proof of Proposition~\ref{p.KPZ FP convergence for V shape bernoulli}]
We give the proof in the case of S6V and explain the minor change for ASEP at the end. 
We wish to show that there exists a sequence $M_\varepsilon\to\infty$ such that $\h^{\mrm{S6V},\varepsilon}(\h_0^{M_\varepsilon,\varepsilon}; \bm\cdot, 1)$ converges weakly, as $\varepsilon\to 0$, to $\h(\h_0;\bm\cdot, 1)$ uniformly on compact sets. We fix an arbitrary $K>0$ and consider the compact set $[-K^\mrm{sc},K^\mrm{sc}]$ ($K^{\mrm{sc}}$ as defined in \eqref{e.Msc}), and we will produce such a sequence $M_\varepsilon$ that does not depend on $K$. This suffices since weak convergence uniformly on every fixed compact set $K$ is equivalent to weak convergence in the topology of uniform convergence on compact sets.

A brief outline of the proof is as follows. We will produce an auxiliary particle configuration $\tilde\eta_0^{M,\varepsilon}$ whose associated height function satisfies the assumptions of Proposition~\ref{p.KPZ FP convergence via approximation bernoullis}, thus converges to the KPZ fixed point from $\h_0$ as $\varepsilon\to0$ and $M\to\infty$. We will couple this particle configuration to $\eta_0^{M,\varepsilon}$ such that $\eta_0^{M,\varepsilon}\geq \tilde\eta_0^{M,\varepsilon}$. To show the convergence to $\h(\h_0;\bm\cdot, t)$ for $\h^{\mrm{S6V},\varepsilon}(\h_0^{M_\varepsilon,\varepsilon}; \bm\cdot, 1)$, it suffices to show that the discrepancy (which are color 1 particles) between the two particle configurations $\eta_0^{M,\varepsilon}$ and $\tilde\eta_0^{M,\varepsilon}$ goes to zero as $\varepsilon\to0$ and $M\to\infty$. The bulk of the proof will be devoted to establishing this.

Next we introduce the auxiliary initial configuration $\tilde\eta_0^{M, \varepsilon}:\Z\to\{0,1\}$ which agrees with $\eta_0^{M,\varepsilon}$ to the right of $-M\floor{\varepsilon^{-2/3}}$ and is \iid Bern($\rho_{\varepsilon,-\lambda})$ to the left of $-M\floor{\varepsilon^{-2/3}}$; in particular, it will be \iid Bern($\rho_{\varepsilon,-\lambda})$ outside of $\intint{-M\floor{\varepsilon^{-2/3}}, M\floor{\varepsilon^{-2/3}}}$ (as in the assumptions of Proposition~\ref{p.KPZ FP convergence via approximation bernoullis}). More precisely, $\tilde\eta_0^{M, \varepsilon}$ is defined by
\begin{align}\label{e.gen init condition tilde eta}
\tilde\eta^{M, \varepsilon}_0(x)= \begin{cases}
\eta^{M,\varepsilon}_0(x) & \text{if } x \geq -M\floor{\varepsilon^{-2/3}} \text{ and} \\
\xi^\varepsilon_{-}(x) & \text{if } x < -M\floor{\varepsilon^{-2/3}},
\end{cases}
\end{align}
where $\xi^\varepsilon_{-}(x)\sim\mrm{Bern}(\rho_{\varepsilon,-\lambda})$ independently across $x\in\Z$. Let $\tilde h^{M,\varepsilon}_0$ be the associated height function with $\tilde h^{M,\varepsilon}_0(0) = h^{M,\varepsilon}_0(0)$ and let $\tilde \h_0^M:\R\to\R$ be given by (see right panel of Figure~\ref{f.dominating height functions})
\begin{align*}
 \tilde \h_0^M(x) = \begin{cases}
 \h^M_0(x) & \text{if } x>-M^\mrm{sc} \text{ and} \\
 B(x+M^\mrm{sc}) +\lambda (x+M^\mrm{sc}) + \h^M_0(-M^\mrm{sc}) & \text{if } x<-M^\mrm{sc},
 \end{cases}
 \end{align*}
 where $B:(-\infty,0]\to\R$ is a reversed rate 2 Brownian motion. By our observation that $\tilde \eta_0^{M,\varepsilon}$ is \iid Bern($\rho_{\varepsilon,-\lambda}$) outside $\intint{-M\floor{\varepsilon^{-2/3}}, M\floor{\varepsilon^{-2/3}}}$, Proposition~\ref{p.KPZ FP convergence via approximation bernoullis} and Corollary~\ref{c.agreement of fixed points} yield that 
\begin{align}\label{e.gen init condition first FP convergence}
\lim_{M\to\infty} \lim_{\varepsilon\to0} \h^{\mrm{S6V},\varepsilon}(\tilde \h^{M,\varepsilon}_0; \bm\cdot, 1) \stackrel{d}{=} \h(\h_0; \bm\cdot, 1),
\end{align}
where the convergence is in distribution in the topology of uniform convergence on compact sets. Thus to establish the same convergence for $\h^{\mrm{S6V},\varepsilon}(\h_0^{M,\varepsilon}; \bm\cdot, 1)$, it suffices to show that
\begin{align*}
\lim_{M\to\infty} \lim_{\varepsilon\to 0}\sup_{[-K^\mrm{sc},K^\mrm{sc}]}\left|\h^{\mrm{S6V},\varepsilon}(\h_0^{M,\varepsilon}; \bm\cdot, 1) -  \h^{\mrm{S6V},\varepsilon}(\tilde \h^{M,\varepsilon}_0; \bm\cdot, 1)\right| = 0,
\end{align*}
where the limit is in probability.

To this end, we first couple $\xi^\varepsilon_{-}$ and $\eta_0^{M,\varepsilon}$ such that $\xi^\varepsilon_{-}(x) \leq \eta_0^{M,\varepsilon}(x)$ for all $x < -M\floor{\varepsilon^{-2/3}}$, by writing $\xi^\varepsilon_{-}$ as a thinning of $\eta_0^{M,\varepsilon}$, i.e., each particle in $\eta_0^{M,\varepsilon}$ is independently removed with probability $2\lambda\sigma(\alpha)\beta(\alpha)^{-1}\varepsilon^{1/3}/\rho_{\varepsilon,\lambda}$ to yield $\xi^\varepsilon_{-}$. Doing so, the particle configurations $\tilde\eta_0^{M,\varepsilon}$ and $\eta_0^{M,\varepsilon}$ are also ordered, i.e., 
\begin{equation}\label{e.eta ordering for V argument}
\tilde\eta_0^{M,\varepsilon}(x) \leq \eta_0^{M,\varepsilon}(x)\ \ \text{for all } x\in\Z;
\end{equation}
see Figure~\ref{f.particle configurations}.

We now couple the evolutions from both these initial configurations by the basic coupling. Since the initial configurations are ordered, this can equivalently be written in terms of the evolution of a colored system $\tilde\eta^{\mrm{col},M,\varepsilon}_0:\Z\to\{-\infty,1,2\}$ defined (see Figure~\ref{f.particle configurations}) by 
\begin{align}
\tilde\eta^{\mrm{col},M,\varepsilon}_0(x) &= \begin{cases}
2 & \text{if } \tilde\eta_0^{M, \varepsilon}(x) = 1,\\
1 & \text{if } \tilde\eta_0^{M, \varepsilon}(x) = 0 \text{ and } \eta_0^{M,\varepsilon}(x) = 1, \text{ and} \\
-\infty & \text{if } \eta_0^{M,\varepsilon}(x) = 0
\end{cases}\nonumber\\[6pt]
&= \begin{cases}
2 & \text{if } \eta_0^{M,\varepsilon}(x) = 1 \text{ and } x \geq -M\floor{\varepsilon^{-2/3}};\\
2 & \text{if } \xi^\varepsilon_{-}(x) = 1 \text{ and } x < -M\floor{\varepsilon^{-2/3}}; \\
1 & \text{if } \eta_0^{M,\varepsilon}(x) = 1, \xi^\varepsilon_{-}(x) = 0, \text{ and } x < -M\floor{\varepsilon^{-2/3}}; \text{ and} \\
-\infty & \text{otherwise}.
\end{cases}\label{e.tilde eta_0 col}
\end{align}
Observe that the evolution of the color 2 particles alone  is that of $\tilde\eta^{M,\varepsilon}_0$ (whose height function we know converges to the KPZ fixed point as in \eqref{e.gen init condition first FP convergence}), while that of the color 1 and color 2 particles together if the labels are ignored is that of $\eta^{M,\varepsilon}_0$. Therefore, to show the convergence on compact sets of the height function associated to $\eta_0^{M,\varepsilon}$ to the KPZ fixed point, it suffices to show that the number of color 1 particles in $\tilde\eta^{\mrm{col},M,\varepsilon}_t$ at $t=\varepsilon^{-1}$ that are to the right of $-K\floor{\varepsilon^{-2/3}}$ is $o_\varepsilon(\varepsilon^{-1/3})$ with probability $1-o_\varepsilon(1)$ as $\varepsilon\to 0$. We will do this by showing that (i) the color 1 particles could not overtake many color 2 particles, and (ii) that not too many color 2 particles can be to the right of $-K\floor{\varepsilon^{-2/3}}$ at time $t=\floor{\varepsilon^{-1}}$. 

\begin{figure}
\begin{adjustbox}{scale=0.95}
\begin{tikzpicture}
\foreach \y in {-2, -1, 0, 1}
{
  \draw (-2.25,\y) -- (7.75,\y);
  \draw[dashed] (7.75,\y) -- ++(0.5,0);
  \draw[dashed] (-2.25,\y) -- ++(-0.5,0);

  \draw[dotted, semithick] (0.275, \y+ 0.25) -- ++(0,-0.5);
  \draw[dotted, semithick] (1.925, \y+ 0.25) -- ++(0,-0.5);
  \draw[dotted, semithick] (5.225, \y+ 0.25) -- ++(0,-0.5);
}

\foreach \x/\thecolor in {-4/black, -3/white, -2/black, -1/black, 0/black, 1/black, 2/white, 3/black, 4/black, 5/white, 6/black, 7/black, 8/black, 9/white, 10/black, 11/black, 12/white, 13/black, 14/black}
    \node[circle, fill=\thecolor, inner sep = 1.75pt, draw=black] at (0.55*\x, 1) {}; 

\foreach \x/\thecolor in {-4/white, -3/white, -2/black, -1/white, 0/white, 1/black, 2/white, 3/black, 4/black, 5/white, 6/black, 7/black, 8/black, 9/white, 10/black, 11/black, 12/white, 13/black, 14/black}
    \node[circle, fill=\thecolor, inner sep = 1.75pt, draw=black] at (0.55*\x, 0) {}; 

\foreach \x/\thecolor in {-4/red, -3/white, -2/black, -1/red, 0/red, 1/black, 2/white, 3/black, 4/black, 5/white, 6/black, 7/black, 8/black, 9/white, 10/black, 11/black, 12/white, 13/black, 14/black}
    \node[circle, fill=\thecolor, inner sep = 1.75pt, draw=black] at (0.55*\x, -1) {}; 

\foreach \x/\thecolor/\thetext in {-4/green!60!black/-1, -3/white/ , -2/cyan/3, -1/yellow/0, 0/red/1, 1/cyan/3, 2/white/ , 3/black/2, 4/cyan/3, 5/white/ , 6/cyan/3, 7/cyan/3, 8/cyan/3, 9/white/ , 10/cyan/3, 11/cyan/3, 12/white/ , 13/cyan/3, 14/cyan/3}
{
    \node[circle, fill=\thecolor, inner sep = 1.75pt, draw=black] at (0.55*\x, -2) {}; 

    \node[anchor=south, scale=0.75] at (0.55*\x, -1.9) {$\thetext$}; 
    
}

\node[anchor=south west, scale=0.9] at (7.75,1) {$\eta_0^{M,\varepsilon}$};
\node[anchor=south west, scale=0.9] at (7.75,0) {$\tilde \eta_0^{M,\varepsilon}$};
\node[anchor=south west, scale=0.9] at (7.75,-1) {$\tilde\eta_0^{\mrm{col}, M,\varepsilon}$};
\node[anchor=south west, scale=0.9] at (7.75,-2) {$\eta_0^{\mrm{col}, M,\varepsilon}$};

\node[scale=0.7, anchor=north] at (0.275, -2.25) {$-M\floor{\varepsilon^{-2/3}}$};
\node[scale=0.7, anchor=north] at (1.925, -2.25) {$-\tfrac{1}{2}M\floor{\varepsilon^{-2/3}}$};
\node[scale=0.7, anchor=north] at (5.225, -2.25) {$M\floor{\varepsilon^{-2/3}}$};

\end{tikzpicture}
\end{adjustbox}
\caption{The particle configurations $\eta_0^{M,\varepsilon}$, $\tilde\eta_0^{M,\varepsilon}$, $\tilde\eta_0^{\mrm{col}, M, \varepsilon}$, and $\eta_0^{\mrm{col}, M, \varepsilon}$. In the top two, i.e., $\eta_0^{M,\varepsilon}$ and $\tilde\eta_0^{M,\varepsilon}$, $\tilde\eta^{\mrm{col}, M, \varepsilon}$, black indicates a particle and white indicates a whole. For the last two, i.e., $\tilde\eta_0^{\mrm{col}, M, \varepsilon}$, and $\eta_0^{\mrm{col}, M, \varepsilon}$,  white indicates color $-\infty$ (hole), green color $-1$, yellow color 0, red color 1, black color 2, and cyan color 3. Note that in $\eta_0^{\mrm{col}, M,\varepsilon}$ all color 2 particles are in $\intint{-\frac{1}{2}M\floor{\varepsilon^{-2/3}}, -M\floor{\varepsilon^{-2/3}}}$ and there is at most one particle of each color that is 1 or lower.}\label{f.particle configurations}
\end{figure}
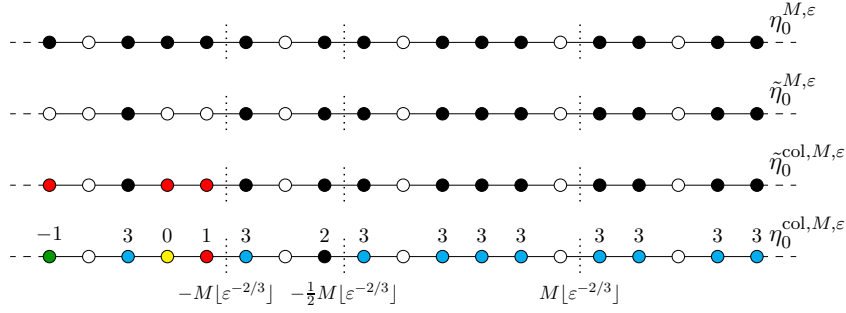

To establish this, we refine the colors in $\tilde\eta^{\mrm{col},M,\varepsilon}_0$ to single out one color 1 particle at a time. To this end, let the positions of the color 1 particles in $\tilde\eta^{\mrm{col},M,\varepsilon}_0$ be labeled $x_1 > x_2 >  \ldots $, let $\zeta^\varepsilon:\Z\to\{0,1\}$ be such that $\zeta^\varepsilon(x)$ is distributed across $x\in\Z$ as \iid $\mrm{Bern}(1-\varepsilon^{1/3})$ (independent of everything else), and define $\eta^{\mrm{col},M,\varepsilon}_0:\Z\to\Z\cup\{-\infty\}$(see Figure~\ref{f.particle configurations}) by
\begin{align}\label{e.eta_0 col}
\eta^{\mrm{col},M,\varepsilon}_0(x)
&= \begin{cases}
3 & \text{if } x > -\tfrac{1}{2}M\floor{\varepsilon^{-2/3}} \text{ or } x < -M\floor{\varepsilon^{-2/3}}, \text{ and }  \tilde\eta_0^{M, \varepsilon}(x) = 1; \\
2 + \zeta^\varepsilon(x) & \text{if } \eta_0^{M,\varepsilon}(x) = 1 \text{ and } x \in \intint{-M\floor{\varepsilon^{-2/3}}, -\tfrac{1}{2}M\floor{\varepsilon^{-2/3}}}; \\
2-k & \text{if } x=x_k; \text{ and} \\
-\infty & \text{otherwise}.
\end{cases}
\end{align}
In words, an $\varepsilon^{1/3}$ proportion of the particles in $\tilde\eta^{\mrm{col},M,\varepsilon}_0$ of color 2 which are in $\intint{-M\floor{\varepsilon^{-2/3}}, -\frac{1}{2}M\floor{\varepsilon^{-2/3}}}$ remain of color 2; the remaining particles of color 2 in $\tilde\eta^{\mrm{col},M,\varepsilon}_0$ in $\intint{-M\floor{\varepsilon^{-2/3}}, -\frac{1}{2}M\floor{\varepsilon^{-2/3}}}$ and all those outside $\intint{-M\floor{\varepsilon^{-2/3}}, -\frac{1}{2}M\floor{\varepsilon^{-2/3}}}$ are made to be color 3; and the particles of color 1 in $\tilde\eta^{\mrm{col},M,\varepsilon}_0$ have their colors decreased sequentially from right to left. Observe that by merging appropriate intervals of particle colors in $\eta^{\mrm{col},M,\varepsilon}_0$ one obtains the colored particle configuration in $\tilde\eta^{\mrm{col},M,\varepsilon}_0$; thus one can evolve $\eta^{\mrm{col},M,\varepsilon}_0$ and do the merging to obtain the evolution of $\tilde\eta^{\mrm{col},M,\varepsilon}_0$. In particular, the evolution of color 2 and 3 particles combined in $\eta_0^{\mrm{col},M,\varepsilon}$, on ignoring the labels, is the same as that of the color 2 particles in $\tilde\eta_0^{\mrm{col},M,\varepsilon}$, i.e, the (uncolored) particles in $\tilde\eta_0^{M,\varepsilon}$ (recall from \eqref{e.tilde eta_0 col}).

\medskip

\noindent \textbf{Control on number of color 2 particles to right of $\bm{-K\floor{\varepsilon^{-2/3}}}$.}
Now we show (ii) above, i.e., that not too many color 2 particles can be to the right of $-K\floor{\varepsilon^{-2/3}}$ at time $t=\floor{\varepsilon^{-1}}$.  Let $\tilde h_{0,3}^{M,\varepsilon}$ be the height function associated to particle configuration $x\mapsto\one_{\eta_0^{\mrm{col},M,\varepsilon}(x) = 3}$ with $\smash{\tilde h^{M,\varepsilon}_{0,3}}(0) = h^{\varepsilon}_0(0)$. 

Observe from \eqref{e.dominating particle configuration for V}, \eqref{e.gen init condition tilde eta}, and \eqref{e.eta_0 col} that, outside of $\intint{-M\floor{\varepsilon^{-2/3}}, M\floor{\varepsilon^{-2/3}}}$, this particle configuration is \iid $\mrm{Bern}(\rho_{\varepsilon, -\lambda})$, and its rescaled height function converges uniformly on compact sets to a continuous function (since this is true for $h_0^\varepsilon$). In particular, the assumptions of Proposition~\ref{p.KPZ FP convergence via approximation bernoullis} are satisfied for $\tilde h_{0,3}^{M,\varepsilon}$. Thus by Proposition~\ref{p.KPZ FP convergence via approximation bernoullis} and Corollary~\ref{c.agreement of fixed points},
\begin{align}\label{e.gen init condition second FP convergence}
 \lim_{M\to\infty} \lim_{\varepsilon\to0} \h^{\mrm{S6V},\varepsilon}(\tilde \h^{M,\varepsilon}_{0,3}; \bm\cdot, 1) = \h(\h_0; \bm\cdot, 1),
 \end{align}
 where the convergence is again in distribution in the topology of uniform convergence on compact sets.
Now, for any $t\geq 0$ and $M>0$,
\begin{align}
\MoveEqLeft[6]
\#\left\{x\geq -K\floor{\varepsilon^{-2/3}}: \eta_t^{\mrm{col},M,\varepsilon}(x) = 2\right\}\nonumber\\
&= \#\left\{x\geq -K\floor{\varepsilon^{-2/3}} : \eta_t^{\mrm{col},M,\varepsilon}(x) \geq 2\right\} - \#\left\{x\geq -K\floor{\varepsilon^{-2/3}} : \eta_t^{\mrm{col},M,\varepsilon}(x) = 3\right\}\nonumber\\
&= h^{\mrm{S6V},\varepsilon}\left(\tilde h_0^{M,\varepsilon}; -K\floor{\varepsilon^{-2/3}}, t\right) - h^{\mrm{S6V},\varepsilon}\left(\tilde h_{0,3}^{M,\varepsilon}; -K\floor{\varepsilon^{-2/3}}, t\right);\label{e.flux difference}
\end{align}
this uses that $\tilde h_0^{M,\varepsilon}(-K\floor{\varepsilon^{-2/3}}) = \tilde h^{M,\varepsilon}_{0,3}(-K\floor{\varepsilon^{-2/3}})$ by assumption and that
$$h^{\mrm{S6V},\varepsilon}(\tilde h_0^{M,\varepsilon}; -K\floor{\varepsilon^{-2/3}}, t) - \tilde h_0^{M,\varepsilon}(-K\floor{\varepsilon^{-2/3}})$$
is the flux of particles of color at least 2 in $\eta_0^{\mrm{col},M,\varepsilon}$ across $-K\floor{\varepsilon^{-2/3}}$ ($+1$ for every particle crossing from left to right, and $-1$ for crossing from right to left) by time $\floor{\varepsilon^{-1}}$, and analogously for $\tilde h_{0,3}^{M,\varepsilon}$ with particles of color at least $3$.

From \eqref{e.gen init condition first FP convergence} and \eqref{e.gen init condition second FP convergence}, it follows that there exists a sequence $M_\varepsilon\to\infty$ as $\varepsilon\to 0$ such that both those convergences hold as $\varepsilon\to 0$ if $M$ is replaced by $M_\varepsilon$. So \eqref{e.gen init condition first FP convergence} and \eqref{e.gen init condition second FP convergence} imply that, taking $M=M_\varepsilon$ and $t=\floor{\varepsilon^{-1}}$ in \eqref{e.flux difference}, the rescaled versions of the two functions in \eqref{e.flux difference} converge to the same limit as $\varepsilon\to 0$; in particular, we obtain that
\begin{align*}
\lim_{\varepsilon\to0}\varepsilon^{1/3}\cdot\#\left\{x\geq -K\floor{\varepsilon^{-2/3}}: \eta_{\floor{\varepsilon^{-1}}}^{\mrm{col},M_\varepsilon,\varepsilon}(x) = 2\right\} = 0,
\end{align*}
where the convergence is in distribution, and, since the limit is a constant, also in probability. In particular, for any $\delta>0$, we obtain
\begin{align}\label{e.count prob}
\lim_{\varepsilon\to 0}\P\left((\msf{Count}^{(2), \varepsilon}_{\delta})^c\right) = 0,
\end{align}
where
\begin{align*}
\msf{Count}^{(2), \varepsilon}_{\delta} := \left\{\#\left\{x\geq -K\floor{\varepsilon^{-2/3}}: \eta_{\floor{\varepsilon^{-1}}}^{\mrm{col},M_\varepsilon,\varepsilon}(x) = 2\right\} < \delta \varepsilon^{-1/3}\right\}.
\end{align*}

\medskip

\noindent\textbf{Lower bound on number of color 2 particles.} We will also need that there was originally a large number of color 2 particles in $\eta^{\mrm{col},M_\varepsilon,\varepsilon}_0$ to the left of $-K\floor{\varepsilon^{-2/3}}$. Note from \eqref{e.eta_0 col} that all color 2 particles in $\eta^{\mrm{col},M_\varepsilon,\varepsilon}_0$ are to the left of $-M_\varepsilon\floor{\varepsilon^{-2/3}}$, so we just need a lower bound on the total number of color 2 particles in $\eta^{\mrm{col},M_\varepsilon,\varepsilon}_0$. From \eqref{e.eta_0 col}, and recalling that $\xi^\varepsilon:\Z\to\Z$ is an \iid $\mrm{Bern}(1-\varepsilon^{1/3})$ process, we see that this is exactly the same as 
\begin{align}\label{e.total color 2 cardinality}
\MoveEqLeft[16]
\#\left\{x\in\intint{-M_\varepsilon\floor{\varepsilon^{-2/3}}, -\tfrac{1}{2}M_\varepsilon\floor{\varepsilon^{-2/3}}}: \eta_0^{M,\varepsilon}(x)= 1, \xi^\varepsilon(x) = 0\right\}.
\end{align}
Now, the cardinality of the same set without the condition that $\xi^\varepsilon(x) = 0$ is $N:=h_0^{M_\varepsilon,\varepsilon}(-M_\varepsilon\floor{\varepsilon^{-2/3}}) - h_0^{M_\varepsilon,\varepsilon}(-\tfrac{1}{2}M_\varepsilon\floor{\varepsilon^{-2/3}})$, and $\xi^\varepsilon(x) = 0$ occurs with probability $\varepsilon^{1/3}$ for each $x$ in that set. So the cardinality of the set in \eqref{e.total color 2 cardinality} has the distribution of a sum of $N$ \iid $\mrm{Bern}(\varepsilon^{1/3})$ random variables. Thus, by a concentration inequality for sums of Bernoulli random variables, \eqref{e.total color 2 cardinality} is at least $\frac{1}{2}\varepsilon^{1/3}N$ with probability at least $1-\exp(-c\varepsilon^{1/3}N)$. 
Since $h_0^{M_\varepsilon,\varepsilon}$ converges to a continuous function on rescaling (as in Definition~\ref{d.rescaled S6V height function}), it follows that, deterministically, for all small enough $\varepsilon$,
\begin{align*}
N \geq \tfrac{1}{4}|\mu'(\alpha)|M_\varepsilon\varepsilon^{-2/3}.
\end{align*}
Combining the above, with probability at least $1-\exp(-cM_\varepsilon\varepsilon^{-1/3})$, at least $\tfrac{1}{8}|\mu'(\alpha)|M_\varepsilon\varepsilon^{-1/3}$ color 2 particles are all to the left of $-K\floor{\varepsilon^{-2/3}}$ in $\eta_0^{\mrm{col},M_\varepsilon,\varepsilon}$. Calling 
$$\msf{TotalCount}^{(2),\varepsilon}:=\left\{\#\{x\leq -K\floor{\varepsilon^{-2/3}}: \eta^{\mrm{col},M_\varepsilon,\varepsilon}_0 = 2\} \geq \tfrac{1}{8}|\mu'(\alpha)|M_\varepsilon\varepsilon^{-1/3}\right\},$$
we have shown that
\begin{align}\label{e.lot of color 2}
\P\left(\msf{TotalCount}^{(2),\varepsilon}\right) \geq 1-\exp(-cM_\varepsilon\varepsilon^{-1/3}).
\end{align}

\medskip

\noindent\textbf{Lower color particles cannot overtake too many color 2 particles.} We have now shown that there are a large number of color 2 particles in the system, and not too many of them to the right of $-K\floor{\varepsilon^{-1/3}}$ at time $\varepsilon^{-1}$. In particular, at least order $\varepsilon^{-1/3}$ many color 2 particles are to the \emph{left} of $-K\floor{\varepsilon^{-1/3}}$ at time $\varepsilon^{-1}$. Next we show using Lemma~\ref{l.overtaking} that not too many particles of lower color can overtake color 2 particles, which will imply that, with high probability, not many lower color particles (in fact, none) can be to the right of $-K\floor{\varepsilon^{-1/3}}$ at time $\varepsilon^{-1}$.

Let $\smash{\msf{Overtake}^{(1)}_{\varepsilon}}$ be the event that the unique color 1 particle in $\smash{\eta_{\floor{\varepsilon^{-1}}}^{\mrm{col},M_\varepsilon,\varepsilon}}$ is to the right of the $\floor{\varepsilon^{-1/6}}$\th color 2 particle (counted from the left).
Then by Lemma~\ref{l.overtaking} with $t=\floor{\varepsilon^{-1}}$, we obtain that
\begin{align*}
 \P\left(\msf{Overtake}^{(1)}_{\varepsilon}\right) \leq q^{\floor{\varepsilon^{-1/6}}}.
\end{align*}
 Observe that, on the event $\smash{\msf{TotalCount}^{(2),\varepsilon}}\cap\smash{\msf{Count}^{(2), \varepsilon}_{\delta}}\cap\smash{\msf{Overtake}^{(1)}_{\varepsilon}}$, the color 1 particle is to the left of $-K\floor{\varepsilon^{-2/3}}$ at time $\floor{\varepsilon^{-1}}$. For $i=-1,0, \ldots, \floor{\varepsilon^{-2}}$, define $\smash{\msf{Overtake}^{(i)}_{\varepsilon}}$ to be the event that the unique color $-i$ particle in $\smash{\eta_{\floor{\varepsilon^{-1}}}^{\mrm{col},M_\varepsilon,\varepsilon}}$ is to the right of the $\floor{\varepsilon^{-1/6}}$\th color 2 particle (counted from the left), and, if no color $-i$ particle exists, the empty set. 
Then it also holds that, on $\smash{\msf{TotalCount}^{(2),\varepsilon}}\cap\msf{Count}^{(2), \varepsilon}_{\delta}\cap\smash{\msf{Overtake}^{(i)}_{\varepsilon}}$, the color $-i$ particle (if it exists) is to the left of $-K\floor{\varepsilon^{-2/3}}$ at time $\floor{\varepsilon^{-1}}$.

As observed earlier, by color merging, the same argument as for the color 1 particle applies to the color $-i$ particle, yielding, for all $i= -1,0, \ldots, \floor{\varepsilon^{-2}}$,
\begin{align}\label{e.overtake prob}
\P\left(\msf{Overtake}^{(i)}_{\varepsilon}\right) \leq q^{\floor{\varepsilon^{-1/6}}}.
\end{align}

Let $\msf{Overtake}^{(\geq \floor{\varepsilon^{-2}})}_{\varepsilon}$ be the event that any particle of color less than or equal to $-\floor{\varepsilon^{-2}}$ is to the right of $-K\floor{\varepsilon^{-2/3}}$ in $\smash{\eta^{\mrm{col},M_\varepsilon,\varepsilon}_{\floor{\varepsilon^{-2}}}}$, i.e., at time $t=\varepsilon^{-1}$. Then the finite speed of propagation estimate (Lemma~\ref{l.finite speed of propagation}) yields that 
\begin{equation}\label{e.overtake remaining prob}
\P\left(\msf{Overtake}^{(\geq \floor{\varepsilon^{-2}})}_{\varepsilon}\right) \leq \exp(-c\varepsilon^{-1}).
\end{equation}
Thus a union bound combined with \eqref{e.count prob}, \eqref{e.lot of color 2}, \eqref{e.overtake prob}, and \eqref{e.overtake remaining prob} shows that, as $\varepsilon\to 0$,  
\begin{align*}
\MoveEqLeft[6]
\P\left(\msf{TotalCount}^{(2),\varepsilon}\cap\msf{Count}^{(2), \varepsilon}_{\delta} \cap (\msf{Overtake}^{(\geq \floor{\varepsilon^{-2}})}_{\varepsilon})^c\cap\bigcap_{i=-1}^{\floor{\varepsilon^{-2}}}(\msf{Overtake}^{(i)}_{\varepsilon})^c\right)\\
&\geq \P\left(\msf{TotalCount}^{(2),\varepsilon}\cap\msf{Count}^{(2), \varepsilon}_{\delta} \right) - \varepsilon^{-2}q^{\floor{\varepsilon^{-1/6}}} - \exp(-c\varepsilon^{-1})
\to 1.
\end{align*}

Since on $\msf{TotalCount}^{(2),\varepsilon}\cap\msf{Count}^{(2), \varepsilon}_{\delta} \cap (\msf{Overtake}^{(\geq \floor{\varepsilon^{-2}})}_{\varepsilon})^c\cap \bigcap_{i=-1}^{\floor{\varepsilon^{-2}}}(\msf{Overtake}^{(i)}_{\varepsilon})^c$ it holds that no particle of color 1 is to the right of $-K\floor{\varepsilon^{-2/3}}$ in $\tilde\eta_{\floor{\varepsilon^{-1}}}^{\mrm{col},M_\varepsilon,\varepsilon}$, it follows from the previous display that
\begin{align}\label{e.no 1 particles to right}
\lim_{\varepsilon\to 0}\P\left(\#\left\{x\geq -K\floor{\varepsilon^{-2/3}} : \tilde\eta^{\mrm{col},M_\varepsilon,\varepsilon}_{\floor{\varepsilon^{-1}}}(x) = 1\right\} = 0\right) = 1.
\end{align}
Recall from \eqref{e.tilde eta_0 col} that $\{x\in\Z:\tilde\eta_{\floor{\varepsilon^{-1}}}^{\mrm{col},M_\varepsilon,\varepsilon}(x) \geq 1\} = \{x\in\Z:\eta_{\floor{\varepsilon^{-1}}}^{\varepsilon}(x) = 1\}$ and $\{x\in\Z:\tilde\eta_{\floor{\varepsilon^{-1}}}^{\mrm{col},M_\varepsilon,\varepsilon}(x) = 2\} = \{x\in\Z:\tilde\eta_{\floor{\varepsilon^{-1}}}^{M_\varepsilon,\varepsilon}(x) = 1\}$. So we obtain that, for $x\geq -K\floor{\varepsilon^{-2/3}}$,
\begin{align*}
h^{\mrm{S6V}}(h^{M_\varepsilon,\varepsilon}_0; x, \floor{\varepsilon^{-1}}) = 
 h^{\mrm{S6V}}(\tilde h^{M_\varepsilon,\varepsilon}_0; x, \floor{\varepsilon^{-1}}) + \#\{y\geq x: \tilde\eta^{\mrm{col},M_\varepsilon,\varepsilon}_{\floor{\varepsilon^{-1}}}(y) = 1\}.
\end{align*}
The final quantity is clearly non-negative and, on the event in \eqref{e.no 1 particles to right}, converges to $0$ as $\varepsilon\to 0$. Thus after rescaling we obtain that, as $\varepsilon\to 0$,
\begin{align*}
\sup_{x\in[-K^\mrm{sc},K^\mrm{sc}]}\bigl|\h^{\mrm{S6V},\varepsilon}(\h^{M_\varepsilon,\varepsilon}_0; x, 1) - \h^{\mrm{S6V}}(\tilde \h^{M_\varepsilon,\varepsilon}_0; x, 1)\bigr| \to 0,
\end{align*}
where the limit is in probability. Since we know from \eqref{e.gen init condition first FP convergence} that $\h^{\mrm{S6V}}(\tilde \h^{M_\varepsilon,\varepsilon}_0; \bm\cdot, 1)$ converges in distribution to $\h(\h_0;\bm\cdot, 1)$ (in the topology of uniform convergence as compact sets) as $\varepsilon\to 0$, it follows that the same is true for $\h^{\mrm{S6V},\varepsilon}(\h^{M_\varepsilon,\varepsilon}_0; \bm\cdot, 1)$. This completes the proof in the case of S6V.

\medskip

\noindent\textbf{Modifications for ASEP.}  First, as noted earlier, the definition of $\eta_0^{M,\varepsilon}$ for ASEP is as given in \eqref{e.dominating particle configuration for V ASEP} (rather than \eqref{e.dominating particle configuration for V}). Next $\tilde\eta_0^{M,\varepsilon}$ is as in \eqref{e.gen init condition tilde eta} with $\xi^\varepsilon_+$ in place of $\xi^\varepsilon_-$ for $x<-M\floor{\varepsilon^{-2/3}}$. We couple $\tilde\eta_0^{M,\varepsilon}$ and $\eta_0^{M,\varepsilon}$ such that
\begin{align*}
\tilde\eta_0^{M,\varepsilon}(x) \geq \eta_0^{M,\varepsilon}(x) \qquad\text{ for all } x\in\Z
\end{align*}
by coupling $\xi_-^\varepsilon$ to be a thinning of $\xi_+^\varepsilon$; thus we have the opposite ordering of $\tilde\eta_0^{M,\varepsilon}$ and $\eta_0^{M,\varepsilon}$ as compared to \eqref{e.eta ordering for V argument}. We still know from Proposition~\ref{p.KPZ FP convergence via approximation bernoullis} and Corollary~\ref{c.agreement of fixed points} that $\h^{\mrm{S6V},\varepsilon}(\tilde\h_0^{M,\varepsilon}; \bm\cdot, 1)$ converges weakly in the topology of uniform convergence on compact sets to $\h(\h_0;\bm\cdot, 1)$ as $\varepsilon\to0$ and then $M\to\infty$, so we need to show
\begin{align*}
\lim_{M\to\infty} \lim_{\varepsilon\to 0}\sup_{[-K^\mrm{sc},K^\mrm{sc}]}\left|\h^{\mrm{S6V},\varepsilon}(\h_0^{M,\varepsilon}; \bm\cdot, 1) -  \h^{\mrm{S6V},\varepsilon}(\tilde \h^{M,\varepsilon}_0; \bm\cdot, 1)\right| = 0,
\end{align*}
where the limit is in probability. As in the S6V proof, this comes down to controlling the movement of color 1 particles in the colored system
\begin{align*}
\tilde\eta^{\mrm{col},M,\varepsilon}_0(x) &= \begin{cases}
2 & \text{if } \eta_0^{M, \varepsilon}(x) = 1,\\
1 & \text{if } \eta_0^{M, \varepsilon}(x) = 0 \text{ and } \tilde \eta_0^{M,\varepsilon}(x) = 1, \text{ and} \\
-\infty & \text{if } \tilde \eta_0^{M,\varepsilon}(x) = 0
\end{cases}\nonumber\\[6pt]
&= \begin{cases}
2 & \text{if } \eta_0^{M,\varepsilon}(x) = 1 \text{ and } x \geq -M\floor{\varepsilon^{-2/3}};\\
2 & \text{if } \xi^\varepsilon_{-}(x) = 1 \text{ and } x < -M\floor{\varepsilon^{-2/3}}; \\
1 & \text{if } \tilde\eta_0^{M,\varepsilon}(x) = 1, \xi^\varepsilon_{-}(x) = 0, \text{ and } x < -M\floor{\varepsilon^{-2/3}}; \text{ and} \\
-\infty & \text{otherwise}.
\end{cases}\label{e.tilde eta_0 col}
\end{align*}
For this we again define a refined colored initial condition $\eta^{\mrm{col},M,\varepsilon}_0$ by
\begin{align*}
\eta^{\mrm{col},M,\varepsilon}_0(x)
&= \begin{cases}
3 & \text{if } x > -\tfrac{1}{2}M\floor{\varepsilon^{-2/3}} \text{ or } x < -M\floor{\varepsilon^{-2/3}}, \text{ and }  \eta_0^{M, \varepsilon}(x) = 1; \\
2 + \zeta^\varepsilon(x) & \text{if } \tilde\eta_0^{M,\varepsilon}(x) = 1 \text{ and } x \in \intint{-M\floor{\varepsilon^{-2/3}}, -\tfrac{1}{2}M\floor{\varepsilon^{-2/3}}}; \\
2-k & \text{if } x=x_k; \text{ and} \\
-\infty & \text{otherwise},
\end{cases}
\end{align*}
with $x_k$ a labeling from right to left of the color 1 particles in $\tilde\eta_0^{\mrm{col},M,\varepsilon}$ and $\zeta^\varepsilon$ is an \iid Bern($1-\varepsilon^{1/3})$ process.
From this point the proof proceeds in exactly the same way as in the S6V case.
\end{proof}

\subsection{Proofs of Theorem~\ref{t.main kpz fixed point convergence}}
\label{s.general proof}

We start by establishing Theorem~\ref{t.main kpz fixed point convergence} in the case of a single initial condition and time and when the limiting initial condition $\h_0$ is continuous, which we state next for the convenience of the reader. We assume continuity only to be able to apply Proposition~\ref{p.KPZ FP convergence for V shape bernoulli}, which in turn assumed continuity in order to make use of Proposition~\ref{p.KPZ FP convergence via approximation bernoullis}.

\begin{proposition}\label{p.main result in continuous case}
Fix $q\in[0,1)$, $z\in(0,1)$, $\alpha\in(z,z^{-1})$, and $t>0$. Let $\lambda > 0$. Suppose that $h_0^{\varepsilon}:\Z\to \Z$ is a sequence (indexed by $\varepsilon>0$) of initial height functions for S6V satisfying \eqref{e.initial condition assumption} for all $\varepsilon>0$ with that fixed $\lambda$ and is such that  $\h_0^{\varepsilon}\to \h_0$ uniformly on compact sets as $\varepsilon\to 0$, where $\h_0:\R\to\R$ is continuous. Then, 
$$\h^{\mrm{S6V},\varepsilon}(\h_0^{\varepsilon};\bm\cdot, t) \stackrel{d}{\to} \h(\h_0;\bm\cdot, t) \text{ as } \varepsilon\to 0 $$
under the topology of uniform convergence on compact sets. The same holds for the ASEP height function for fixed $t>0$, $q\in[0,1)$, and $\alpha\in(-1,1)$.
\end{proposition}

The idea of the proof of Proposition~\ref{p.main result in continuous case} is as follows. From Lemma~\ref{l.prelimiting variation inequality s6v} we have a lower bound on $\h^{\mrm{S6V},\varepsilon}(\h_0^\varepsilon; \bm\cdot, t)$ by a sequence of functions which converges to $\h(\h_0;\bm\cdot, t)$, so we only need to prove a similar upper bound. Recall $\h_0^\varepsilon$ grows at most like $\lambda(1+|x|)$ for some $\lambda>0$, and consider a height function $\smash{\h_0^{\shortuparrow,\lambda',M,\varepsilon}}$ which agrees with $\h_0^{\varepsilon}$ on a large interval $[-M^\mrm{sc},M^\mrm{sc}]$, is made to increase above $\h_0^\varepsilon$ on the neighboring intervals $(-R^\mrm{sc},-M^\mrm{sc})$ and $(M^\mrm{sc},R^\mrm{sc})$ for some $R>M$ (and recalling the ``sc'' notation from \eqref{e.Msc}), and then behaves like a random walk with drift $-\lambda'$ on $(-\infty, -R^\mrm{sc})$ and with drift $\lambda'$ on $(R^\mrm{sc},\infty)$ for some $\lambda'>\lambda$. By increasing $\lambda'$, $\smash{\h_0^{\shortuparrow,\lambda',M,\varepsilon}}$, with high probability, is greater than $\h_0^\varepsilon$ on all of $\R$, and Proposition~\ref{p.KPZ FP convergence for V shape bernoulli} applies to it. So by height monotonicity, $\h^{\mrm{S6V},\varepsilon
}(\h_0^\varepsilon;\bm\cdot, t)$ is upper bounded by $\h^{\mrm{S6V},\varepsilon}(\smash{\h_0^{\shortuparrow,\lambda,M,\varepsilon}};\bm\cdot, t)$. By taking $\varepsilon\to0$ and increasing $M$ to $\infty$, the limit of $\h^{\mrm{S6V},\varepsilon}(\smash{\h_0^{\shortuparrow,\lambda',M,\varepsilon}};\bm\cdot, t)$ is $\h(\h_0;\bm\cdot, t)$, which will give the desired upper bound matching the earlier mentioned lower bound.

As is apparent from this discussion, to prove Proposition~\ref{p.main result in continuous case}, we will need a estimate on random walks staying above lines, which is the following. As its proof is a standard argument based on the exponential martingale, we defer it to Appendix~\ref{app.other proofs}. Recall that a Bernoulli random walk with drift $\lambda \in (-1,0)$ means a random walk whose increment distribution is supported on $\{-1,0\}$ with probability $|\lambda|$ given to $-1$ and probability $1-|\lambda|$ to $0$.

\begin{lemma}\label{l.RW avoiding line}
Let $S:\Z_{\geq 0}\to\Z$ be a Bernoulli random walk with drift $\lambda \in (-1,0)$. Fix $\rho\in(-1, \lambda)$. Then, there exists $c = c(\lambda)$ such that, for all $M>0$,
\begin{align*}
\P\Bigl(S(t) \geq \rho t-M \ \forall t\in\Z_{\geq 0}\Bigr) \geq 1- \exp(-c(\lambda-\rho)M).
\end{align*}

\end{lemma}

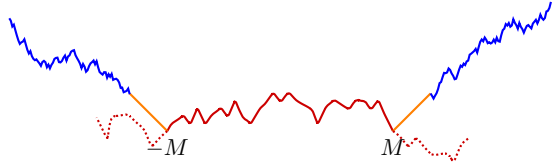
\begin{figure}
\begin{tikzpicture}[scale=1,default/.style={very thick,decorate, decoration={random steps,segment length=3pt,amplitude=1pt}}, use fpu reciprocal]

\draw[red!80!black, default, thick, rounded corners=1pt] (0,0) -- ++(0.2,0.2) -- ++(0.1,-0.1) -- ++(0.1, 0.2) -- ++(0.1, -0.2) -- ++(0.2,0.2) -- ++ (0.1, -0.1) -- ++(0.1,0.2) -- ++(0.2, -0.3) -- ++(0.3, 0.4) -- ++(0.1,-0.1) -- ++(0.1,0.1) -- ++(0.3, -0.2) -- ++(0.1, -0.2) -- ++(0.1, 0.2) -- ++(0.1,0.1) -- ++(0.2, -0.1) -- ++(0.1, 0.1) -- ++(0.1,0) -- ++(0.1,-0.1) -- ++(0.1,0.1) -- ++(0.1, -0.2) -- ++(0.1,-0.2);

\draw[orange, thick] (3,0) -- ++(0.5,0.5);
\draw[orange, thick] (0,0) -- ++(-0.5,0.5);

\draw[red!80!black, default, thick, rounded corners=1.25pt, densely dotted] (0,0) -- ++(-0.2,-0.2) -- ++(-0.1, 0.2) -- ++(-0.2, 0.2) -- ++(-0.2,0) -- ++(-0.1,-0.3) -- ++(-0.1, 0.2) -- ++(-0.1, 0);

\draw[red!80!black, default, thick, rounded corners=1.25pt, densely dotted] (3,0) -- ++(0.2,-0.2) -- ++(0.1, 0.1) -- ++(0.2, -0.2) -- ++(0.2,0) -- ++(0.1,-0.1) -- ++(0.1, 0.2) -- ++(0.1, 0.1);

\bb{80}{0.02}{0.07}{0.5}{1.7}{thick, blue}{46}{3.5} 
\bb{80}{0.02}{0.07}{1.5}{0.5}{thick, blue}{50}{-2.1} 

\node[anchor=north, scale=0.8] at (0,0) {$-M$};
\node[anchor=north, scale=0.8] at (3,0) {$M$};

\end{tikzpicture}
\caption{A depiction of the height functions $\h_0^{\shortuparrow,\lambda,M,\varepsilon}$ and $\h_0^{\varepsilon}$ (red along with dotted red). Note that $\h_0^{\shortuparrow,\lambda,M,\varepsilon}$ dominates $\h_0^\varepsilon$ on the high probability event that the random walks on either side do not dip too low. 
}\label{f.dominating height function for general continuous}
\end{figure}

\begin{proof}[Proof of Proposition~\ref{p.main result in continuous case}]
 We start with the S6V case and address the minor changes for ASEP at the end. We first define the height function which will dominate $\smash{\h_0^\varepsilon}$ and for which we can show convergence to the KPZ fixed point. For $M>0$ and $\lambda'>2\lambda$, let $\zeta^{\lambda', M,\varepsilon}_{+}, \zeta^{\lambda', M,\varepsilon}_{-}:\Z\to\{0,1\}$ be particle configurations such that the height functions $h^{\lambda', M,\varepsilon}_{\zeta,+}, h^{\lambda', M,\varepsilon}_{\zeta,-}:\Z\to\Z$ respectively associated to them with $\smash{h^{\lambda', M,\varepsilon}_{\zeta,\pm}(\pm M)} = 0$ satisfy 
\begin{equation}\label{e.zeta rescaled height function}
\h^{\lambda', M,\varepsilon}_{\zeta, \pm}(x) = \max(\h^\varepsilon_0(x) - \h^\varepsilon_0(\pm M^\mrm{sc}), 0) + \lambda'|x| + \delta_\varepsilon(x)
\end{equation}
for $x\in\R$, where $\delta_\varepsilon(x)\in [0,\varepsilon^{1/3})$ ensures that equality occurs given the constraint that the lefthand side takes values in $\varepsilon^{1/3}\Z$). Let $\xi^{\lambda',\varepsilon}_{\pm}:\Z\to\{0,1\}$ be \iid Bern($\rho_{\varepsilon,\pm\lambda'})$. For the S6V case, let $\eta_0^\varepsilon:\Z\to\{0,1\}$ be the particle configuration associated to $h_0^\varepsilon$ and consider the initial condition $\eta_0^{\shortuparrow, \lambda',M,\varepsilon}:\Z\to\Z$ given by
\begin{align}\label{e.dominating particle configuration}
\eta^{\shortuparrow,\lambda', M,\varepsilon}_0(x) = \begin{cases}
\eta_0^\varepsilon(x) & \text{if } |x| < M\floor{\varepsilon^{-2/3}},\\
\zeta^{\lambda', M,\varepsilon}_{+}(x) & \text{if } x\in \intint{M\floor{\varepsilon^{-2/3}}, R\floor{\varepsilon^{-2/3}}},\\
\zeta^{\lambda', M,\varepsilon}_{-}(x) & \text{if } x\in \intint{-R\floor{\varepsilon^{-2/3}}, -M\floor{\varepsilon^{-2/3}}},\\
\xi^{\lambda',\varepsilon}_{+}(x) & \text{if } x > R\floor{\varepsilon^{-2/3}} \text{, and} \\
\xi^{\lambda',\varepsilon}_{-}(x) & \text{if } x < -R\floor{\varepsilon^{-2/3}},
\end{cases}
\end{align}
where $R = R(M)>M$ will be specified just ahead. Let $h_0^{\shortuparrow,\lambda', M,\varepsilon}:\Z\to\Z$ be the height function associated to $\eta_0^{\shortuparrow,\lambda', M,\varepsilon}$ with $h_0^{\shortuparrow,\lambda', M,\varepsilon}(0) = h_0^\varepsilon(0)$ (see Figure~\ref{f.dominating height function for general continuous}). In particular, $h_0^{\shortuparrow,\lambda', M,\varepsilon}(x) = h_0^\varepsilon(x)$ for all $|x|\leq M\floor{\varepsilon^{-2/3}}$ and $h_0^{\shortuparrow,\lambda', M,\varepsilon}(x) \geq h_0^\varepsilon(x)$ for $x\in \intint{-R\floor{\varepsilon^{-2/3}}, -M\floor{\varepsilon^{-2/3}}}$. We choose $R>M$ to be the smallest real number such that, for all $\varepsilon>0$, 
$$\h_0^{\shortuparrow,\lambda', M,\varepsilon}(-R^\mrm{sc}) > \tfrac{1}{2}\lambda'(1+R^\mrm{sc}),$$
where recall the righthand side is an upper bound on $\h_0(-R^\mrm{sc})$ since $\lambda'>2\lambda$. Such an $R$ exists and is finite since $\h_0^{\shortuparrow,\lambda',M}(-M^\mrm{sc}) = \h_0(-M^\mrm{sc}) > -\infty$ by the continuity of $\h_0$. We adopt this condition to ensure that $\h^{\shortuparrow,\lambda',M,\varepsilon}_0$ continuing with slope $\lambda'$ is sufficient to stay above $\h_0^\varepsilon$ on $\R$ (note that if $\h_0^{\varepsilon}(x)$ is far below $\lambda(1+|x|)$ it could have an arbitrarily large slope for a short duration).

By Lemma~\ref{l.prelimiting variation inequality s6v}, we know that for any $M>0$ and all $\varepsilon>0$ small enough, with probability at least $1-C\exp(-c\varepsilon^{-1/6})$, for all $x\in [-M^\mrm{sc}, M^\mrm{sc}]$,
\begin{align}\label{e.general initial condition lower bound}
\h^{\mrm{S6V},\varepsilon}(\h_0^\varepsilon; x,t) \geq \sup_{y\in[-M,M]}\left(\h_0^\varepsilon(y) + \cL^{\mrm{S6V},\varepsilon}(y,0;x,t)\right) - \sigma(\alpha)^{-1}\varepsilon^{1/6},
\end{align}
where $\h^{\mrm{S6V},\varepsilon}(\h_0; \bm\cdot,t)$ and $\cL^{\mrm{S6V},\varepsilon}(\bm\cdot,0; \bm\cdot,t)$ are coupled via the basic coupling. By Lemma~\ref{l.RW avoiding line} it also holds that, with probability at least $1-\exp(-c\lambda')$, $h_0^\varepsilon(x) \leq h_0^{\shortuparrow,\lambda', M,\varepsilon}(x)$ for all $x\in\Z$, $M>0$, and $\varepsilon>0$ (the probability is to ensure that the random walks on either side stay above $h_0^\varepsilon$). So by approximate height monotonicity (Lemma~\ref{l.s6v approximate height monotonicity}), with probability at least $1-\exp(-c\varepsilon^{-1/6})-\exp(-c\lambda')$, for all $x\in[-M^\mrm{sc},M^\mrm{sc}]$ it holds that
\begin{align}\label{e.general initial condition upper bound}
\h^{\mrm{S6V},\varepsilon}(\h_0^\varepsilon; x,t) \leq \h^{\mrm{S6V},\varepsilon}(\h_0^{\shortuparrow,\lambda', M,\varepsilon}; x,t) + \varepsilon^{1/6},
\end{align}
where the two sides are coupled via the basic coupling. By Proposition~\ref{p.KPZ FP convergence for V shape bernoulli}, for each $\lambda'$ there exists a sequence $R_\varepsilon\to\infty$ as $\varepsilon\to0$ (and thus also a sequence $M_\varepsilon\to\infty$ as $\varepsilon\to 0$) such that the righthand side of \eqref{e.general initial condition upper bound} (with $M$ and the implicit $R$ replaced by $M_\varepsilon$ and $R_\varepsilon$) converges weakly to $\h(\h_0;\bm\cdot, t)$ in the topology of uniform convergence on compact sets as $\varepsilon\to 0$. By Theorem~\ref{t.s6v airy sheet} and Lemma~\ref{l.argmax control}, the righthand side of \eqref{e.general initial condition lower bound} converges weakly in the same topology to $\h(\h_0;\bm\cdot,t)$ as well. Thus, by the ordering provided by \eqref{e.general initial condition lower bound} and \eqref{e.general initial condition upper bound} under the basic coupling and  Lemma~\ref{l.joint kpz fixed point convergence},
\begin{align}\label{e.general joint convergence}
  \left(\h^{\mrm{S6V},\varepsilon}(\h_0^{\shortuparrow,\lambda', M_\varepsilon,\varepsilon}; \bm\cdot,t), \sup_{y\in[-M^\mrm{sc}_\varepsilon,M^\mrm{sc}_\varepsilon]}\left(\h_0^\varepsilon(y) + \cL^{\mrm{S6V},\varepsilon}(y,0;\bm\cdot,t)\right)\right) \stackrel{d}{\to} (\h(\h_0;\bm\cdot, t), \h(\h_0;\bm\cdot, t))
\end{align}    
(i.e., the two coordinates are identically equal) as $\varepsilon\to0$ for each $\lambda'$; thus the same holds as $\varepsilon\to 0$ and then $\lambda'\to\infty$.

By relabeling $M_\varepsilon$ if necessary, there exists a sequence $\lambda'_\varepsilon$ such that, if we replace $\lambda'$ by $\lambda'_\varepsilon$ in \eqref{e.general joint convergence}, then taking $\varepsilon\to 0$ in \eqref{e.general joint convergence} yields joint convergence to $(\h(\h_0;\bm\cdot,t), \h(\h_0;\bm\cdot,t))$. We may assume this convergence happens almost surely on a common probability space which also supports $\h^{\mrm{S6V},\varepsilon}(\h_0^\varepsilon; \bm\cdot,t)$ and on which \eqref{e.general initial condition lower bound} and \eqref{e.general initial condition upper bound} both hold by the Skorohod representation theorem; in particular, $\h^{\mrm{S6V},\varepsilon}(\h_0^\varepsilon;\bm\cdot, t)$ is sandwiched between $\h^{\mrm{S6V},\varepsilon}(\h_0^{\shortuparrow,\lambda'_\varepsilon, M_\varepsilon,\varepsilon}; \bm\cdot,t)$ and  $\sup_{y\in[-M^\mrm{sc}_\varepsilon,M^\mrm{sc}_\varepsilon]}(\h_0^\varepsilon(y) + \cL^{\mrm{S6V},\varepsilon}(y,0;\bm\cdot,t))$. Then taking $\varepsilon\to0$ yields that $\h^{\mrm{S6V},\varepsilon}(\h_0^\varepsilon; \bm\cdot,t)$ also converges almost surely to $\h(\h_0;\bm\cdot, t)$, which completes the proof for S6V.

For the ASEP case, the proof proceeds in exactly the same way with a slightly different definition of $\eta^{\shortuparrow,\lambda', M,\varepsilon}_0$ as compared to \eqref{e.dominating particle configuration}, namely with $\zeta^{\lambda', M,\varepsilon}_{+}(x)$ and $\zeta^{\lambda', M,\varepsilon}_{-}(x)$ switched, and $\xi^{\lambda', M,\varepsilon}_{+}(x)$ and $\xi^{\lambda', M,\varepsilon}_{-}(x)$ switched, i.e., 
\begin{equation*}
\eta^{\shortuparrow,\lambda', M,\varepsilon}_0(x) = \begin{cases}
\eta_0^\varepsilon(x) & \text{if } |x| < M\floor{\varepsilon^{-2/3}},\\
\zeta^{\lambda', M,\varepsilon}_{+}(x) & \text{if } x\in \intint{-R\floor{\varepsilon^{-2/3}}, -M\floor{\varepsilon^{-2/3}}},\\
\zeta^{\lambda', M,\varepsilon}_{-}(x) & \text{if } x\in \intint{M\floor{\varepsilon^{-2/3}}, R\floor{\varepsilon^{-2/3}}},\\
\xi^{\lambda',\varepsilon}_{+}(x) & \text{if } x < -R\floor{\varepsilon^{-2/3}} \text{, and} \\
\xi^{\lambda',\varepsilon}_{-}(x) & \text{if } x > R\floor{\varepsilon^{-2/3}}.
\end{cases}
\end{equation*}
The change is to ensure that $\h_0^{\shortuparrow,\lambda', M,\varepsilon}$ dominates $\h_0^\varepsilon$, since the rescaling as given in Definition~\ref{d.rescaled ASEP height function} has an extra negative sign as compared to its S6V analog Definition~\ref{d.rescaled S6V height function}.
\end{proof}

\noindent With these preliminaries we give the proof of Theorem~\ref{t.main kpz fixed point convergence}. In other words, we upgrade Proposition~\ref{p.main result in continuous case} to the case that $\h_0:\R\to\R\cup\{-\infty\}$ is upper semi-continuous.

\begin{proof}[Proof of Theorem~\ref{t.main kpz fixed point convergence}]
We give the proof for S6V, as the proof for ASEP is identical. By Lemma~\ref{l.joint kpz fixed point convergence}, it suffices to prove convergence to the KPZ fixed point for a single initial condition and single time. We assume without loss of generality that $s_1=0$, i.e., the initial condition is started at time $0$. By Proposition~\ref{p.main result in continuous case}, we have the latter result in the case that $\h_0:\R\to\R$ is continuous.

So we need only consider the case that $\h_0:\R\to\R\cup\{-\infty\}$ is in $\mrm{UC}$ and $\h_0^\varepsilon\to \h_0$ locally in $\mrm{UC}$ as $\varepsilon\to 0$. We still have \eqref{e.prelimiting inequality} that 
\begin{equation}\label{e.UC proof lower bound}
\h^{\mrm{S6V},\varepsilon}(\h_0^\varepsilon; x,t) \geq \sup_{y\in[-M,M]} \left(\h_0^{\varepsilon}(y) + \S^{\mrm{S6V},\varepsilon}(y;x)\right) - o_\varepsilon(1)
\end{equation}
with probability at least $1-o_\varepsilon(1)$ for all $\varepsilon>0$, $x\in[-\varepsilon^{-1}, \varepsilon^{-1}]$, and $M>0$. Since $\h_0^\varepsilon\to\h_0$ locally in UC, it follows that the righthand side of \eqref{e.UC proof lower bound} converges (as $\varepsilon\to 0$) to $\sup_{y\in[-M,M]} (\h_0(y) + \S(y;x))$, uniformly over $x\in[-J,J]$ for any fixed $J$. Further, as $M\to\infty$, $\sup_{y\in[-M,M]} (\h_0(y) + \S(y;x)) \to \sup_{y\in\R} (\h_0(y) + \S(y;x))$ in the topology of uniform convergence on compact sets (using Lemma~\ref{l.argmax control}). Thus, there exists a sequence $M_\varepsilon\to\infty$ such that 
$$\sup_{y\in[-M_\varepsilon, M_\varepsilon]} (\h_0^{\varepsilon}(y) + \S^{\mrm{S6V},\varepsilon}(y;x))  \to \sup_{y\in\R} (\h_0(y) + \S(y;x))$$
in distribution in the topology of uniform convergence on compact sets and for all $\varepsilon>0$, with probability at least $1-o_\varepsilon(1)$  and for all $x\in\R$,
\begin{align}\label{e.variational lower bound for uc}
\h^{\mrm{S6V},\varepsilon}(\h_0^\varepsilon; x,t) \geq \sup_{y\in[-M_\varepsilon, M_\varepsilon]} (\h_0^{\varepsilon}(y) + \S^{\mrm{S6V},\varepsilon}(y;x)) -o_\varepsilon(1).
\end{align}

Further, the following is a consequence of the fact that upper semi-continuous functions can be approximated from above by continuous functions, the continuity of the KPZ fixed point in $\mrm{UC}$ \cite[Theorem 4.1]{matetski2016kpz} (this also follows from the variational definition \eqref{e.fixed point via landscape} given the localization of the maximizer from Lemma~\ref{l.argmax control}), the fact that the topology on UC induces the topology of uniform convergence on compact sets on the space of continuous functions, and the continuity of $\h(\g; \bm\cdot, t)$ for any $\g\in\mrm{UC}$: for any $\delta>0$ there exists $\tilde\h_{0,\delta}: \R\to\R$ continuous such that $\tilde\h_{0,\delta}(y)\geq \h_0(y)$ for all $y\in\R$ and 
\begin{align*}
\sum_{\ell=1}^{\infty}2^{-\ell}\min\left(1, \sup_{y\in[-\ell,\ell]}\bigl|\h(\tilde\h_{0,\delta}; y, t) - \h(\h_{0}; y,t)\bigr|\right) \leq \delta.
\end{align*}

Next, for each $\delta>0$ there exists a sequence $\tilde h_{0,\delta}^\varepsilon:\Z\to\Z$ indexed by $\varepsilon>0$ such that $\tilde h_{0,\delta}^\varepsilon(x) \geq h_{0}^\varepsilon(x)$ for all $\varepsilon>0$ and $x\in\Z$ and $\smash{\tilde \h^\varepsilon_{0,\delta}} \to \smash{\tilde\h_{0,\delta}}$ as $\varepsilon\to0$ uniformly on compact sets. Thus it follows from the above and Proposition~\ref{p.main result in continuous case} that 
$$\lim_{\delta\to 0}\lim_{\varepsilon\to 0} \h^{\mrm{S6V},\varepsilon}(\tilde \h^{\varepsilon}_{0,\delta}; \bm\cdot ,t) = \h(\h_0; \bm\cdot ,t)$$
where the limits are in distribution in the topology of uniform convergence on compact sets. In particular, there exists a sequence $\{\delta_\varepsilon\}_{\varepsilon>0}$ converging to $0$ as $\varepsilon\to 0$ such that $\lim_{\varepsilon\to 0} \h^{\mrm{S6V},\varepsilon}(\tilde \h^{\varepsilon}_{0,\delta_\varepsilon}; \bm\cdot ,t) = \h(\h_0; \bm\cdot ,t)$ weakly in the same topology. Thus by height monotonicity (Lemma~\ref{l.asep height monotonicity}) and \eqref{e.variational lower bound for uc}, it follows that for all $\varepsilon>0$, with probability at least $1-o_\varepsilon(1)$,
\begin{align*}
\h^{\mrm{S6V},\varepsilon}(\tilde \h_{0,\delta_\varepsilon}^\varepsilon; \bm\cdot, t) \geq \h^{\mrm{S6V},\varepsilon}(\h_{0}^\varepsilon; \bm\cdot, t) \geq \sup_{y\in[-M_{\varepsilon},M_{\varepsilon}]} (\h_0^{\varepsilon}(y) + \S^{\mrm{S6V},\varepsilon}(y;x)) - o_\varepsilon(1)
\end{align*}
where the three processes are coupled via the basic coupling for each $\varepsilon>0$. By the Skorohod representation theorem, we can find a coupling of these processes across $\varepsilon>0$ such that both the leftmost and rightmost sides of the previous display converge (as $\varepsilon\to 0$) uniformly on compact sets to $\h(\h_0; \bm\cdot, t)$ with probability 1 (using Lemma~\ref{l.argmax control} to ensure the uniformity of the convergence of the righthand side). Thus the same is true for $\h^{\mrm{S6V},\varepsilon}(\h_{0}^\varepsilon; \bm\cdot, t)$, completing the proof for S6V. 

The proof for ASEP is identical except for the simplification that $\h^{\mrm{ASEP},\varepsilon}(\h_0^\varepsilon; x,t)$ is almost surely lower bounded by $\sup_{y\in[-M,M]} (\h_0^{\varepsilon}(y) + \S^{\mrm{ASEP},\varepsilon}(y;x))$ and that the sequence $\tilde h_{0,\delta}^\varepsilon$ should be taken to be \emph{upper} bounded by $h_0^\varepsilon$ to take into account the sign change on going to $\tilde\h_{0,\delta}^\varepsilon$ according to Definition~\ref{d.rescaled ASEP height function}.
\end{proof}

\appendix

\section{}\label{app.other proofs}

Here we prove various statements from the main body whose proofs were deferred. Theorem~\ref{t.s6v airy sheet} is proved in Section~\ref{s.s6v landscape proof}; Lemmas~\ref{l.prelimiting variation inequality asep}, \ref{l.prelimiting variation inequality s6v}, and \ref{l.joint kpz fixed point convergence} are proved in Section~\ref{s.joint from marginal proofs}; Lemma~\ref{l.s6v approximate height monotonicity} on approximate height monotonicity in Section~\ref{s.approximate monotonicity proof}; Lemma~\ref{l.argmax control} in Section~\ref{s.fixed point maximizer proof}; Lemma~\ref{l.finite speed of propagation} in Section~\ref{s.finite speed of prop}; and Lemma~\ref{l.RW avoiding line} in Section~\ref{s.random walk proof}.

\subsection{S6V landscape convergence}\label{s.s6v landscape proof}

\begin{proof}[Proof of Theorem~\ref{t.s6v airy sheet}]
 Let $N\in\N$ be a constant. \cite{aggarwal2024scaling} works with initial conditions with finitely many arrows, all incident inside $\intint{-N,N}$. In our notation, the initial conditions considered in that paper are $\tilde h^{\mrm{step}}_{0,y}$, given by $z\mapsto (N-y+1)\one_{z\leq y-1} + (N-z)\one_{z\in\intint{y, N}}$ for each $y\in\intint{-N,N}$, i.e., $\tilde h^{\mrm{step}}_{0,y}(z)$ is the number of arrows strictly above $(1,z)$ when there are arrows horizontally incident at each vertex in $\{(1,k):k\in\intint{y,N}\}$. Then, if we denote the object defined in \cite[eq. (2.17)]{aggarwal2024scaling} by $\tilde h_N^{\mrm{S6V}}$, it can be written in our notation as $\tilde h^{\mrm{S6V}}_N(y,0;x,t) = h^{\mrm{S6V}}(\tilde h^{\mrm{step}}_{0,y}; x,t)$ (recall \eqref{e.s6v colored height function}). 

Note that for $z\in\intint{-N,N}$, $h^{\mrm{step}}_{0,y}(z) = \tilde h^{\mrm{step}}_{0,y}(z) - (N-y+1)$.
Constant height changes are preserved by the dynamics, so, by a finite speed of discrepancy estimate (see, e.g., Lemma~\ref{l.finite speed of discrepancy for s6v basic coupling} ahead), $h^{\mrm{S6V}}$ and $\tilde h_N^{\mrm{S6V}}$ are related by
\begin{align}\label{e.relation between height function definitions}
h^{\mrm{S6V}}(y,0;x,t) = \tilde h_N^{\mrm{S6V}}(y,0;x,t) - N +y -1
\end{align}
on an event with probability tending to $1$ as $N\to\infty$; the equality is not almost sure due to the discrepancy in the initial conditions associated to the systems outside of $\intint{-N,N}$. 

Taking this relation into account, if we call the object defined in \cite[eq. (2.20)]{aggarwal2024scaling} as $\tilde \S^{\mrm{S6V},\varepsilon}$ (with $N$ set to be at least $2\alpha\varepsilon^{-1}$), the definition \eqref{e.rescaled s6v definition} of $\S^{\mrm{S6V}, \varepsilon}$ and $\tilde \S^{\mrm{S6V},\varepsilon}$ agree up to a term of $\sigma(\alpha)^{-1}\varepsilon^{1/3}$ on an event whose probability tends to $1$ as $\varepsilon\to 0$ (the discrepancy arising due to the ignoring of the $-1$ term in \eqref{e.relation between height function definitions} in \cite{aggarwal2024scaling}, as it disappears in the $\varepsilon\to 0$ limit). 
This in turn shows that Theorem~\ref{t.s6v airy sheet} follows from \cite[Theorem 2.20]{aggarwal2024scaling} combined with the standard fact that if random elements $\smash{X_n \to X}$ in distribution as $n\to\infty$ in a Polish space $(\mc X, d)$ and $d(X_n,Y_n) \to 0$ in probability, then it follows that $Y_n \to X$ in distribution as $n\to\infty$ as well.
\end{proof}

\subsection{Joint convergence from marginal convergence} \label{s.joint from marginal proofs}

We will first prove Lemmas~\ref{l.prelimiting variation inequality asep} and \ref{l.prelimiting variation inequality s6v} on the variational inequality for ASEP and S6V before turning to the proof of Lemmas~\ref{l.joint kpz fixed point convergence} on upgrading marginal convergence to KPZ fixed points to joint convergence, as the latter will use the former. The proofs are essentially the same as that of \cite[Corollary 2.12 (3)]{aggarwal2024scaling}.

\begin{proof}[Proofs of Lemmas~\ref{l.prelimiting variation inequality asep} and \ref{l.prelimiting variation inequality s6v}]
We give the proof for S6V and indicate the minor simplifications for ASEP at the end. We first claim that, almost surely, for every $x,y\in\Z$,
\begin{align}\label{e.s6v easy inequality fixed point}
 h^{\varepsilon}_0(x) \geq h^{\varepsilon}_0(y) + h^{\mrm{S6V}}(y, 0; x, 0).
\end{align}
First, this holds for $x=y$ since $h^{\mrm{S6V}}(y, 0; y, 0) = 0$ (recall the definition from \eqref{e.s6v colored height function}). Observe that both sides are non-increasing functions of $x$ and decrease by at most $1$ on going from $x$ to $x+1$. Since the righthand side decreases by exactly 1 on going from $x$ to $x+1$ when $x\geq y$, \eqref{e.s6v easy inequality fixed point} is proven for $x\geq y$. Since the righthand side stays constant on going from $x$ to $x-1$ when $x\leq y$ while the lefthand side increases by $0$ or $1$, \eqref{e.s6v easy inequality fixed point} is established.

We now view \eqref{e.s6v easy inequality fixed point} as an ordering of initial conditions: for each $y$ fixed, we regard $h^{\varepsilon}_0(y)$ as a constant and both $h^{\mrm{S6V}}(y, 0; \bm\cdot, 0)$ and $\smash{h^{\varepsilon}_0(\bm\cdot)}$ as initial conditions. Evolving both for time $\floor{t\varepsilon^{-1}}$ under the basic coupling, Lemma~\ref{l.s6v approximate height monotonicity} (with $N=\varepsilon^{-2}$ and $M=\varepsilon^{-1/6}$) and a union bound over $y\in \intint{-\varepsilon^{-2}, \varepsilon^{-2}}$ guarantees that, with probability at least $1-C\exp(-c\varepsilon^{-1/6})$, for all $x,y\in\intint{-\varepsilon^{-2}, \varepsilon^{-2}}$,
\begin{align}\label{e.inequality after monotone evolution}
h^{\mrm{S6V}}(h_0^{\varepsilon}; x, \floor{t\varepsilon^{-1}}) \geq h^{\varepsilon}_0(y) + h^{\mrm{S6V}}(y, 0; x, \floor{t\varepsilon^{-1}}) - \varepsilon^{-1/6}.
\end{align}
Taking the maximum over $y\in\intint{-\varepsilon^{-2}, \varepsilon^{-2}}$ and writing in terms of the rescaled quantities (as defined in Definition~\ref{d.rescaled S6V height function}) yields that, with probability at least $1-C\exp(-\varepsilon^{-1/6})$, for all $x\in\intint{-\varepsilon^{-1},\varepsilon^{-1}}$,
\begin{align}\label{e.rescaled inequality}
\h^{\mrm{S6V},\varepsilon}(\h_0^\varepsilon; x,t) \geq \sup_{y\in [-\varepsilon^{-1}, \varepsilon^{-1}]}\left(\h_0^\varepsilon(y) + \cL^{\mrm{S6V},\varepsilon}(y, 0; x,t)\right) - \sigma(\alpha)^{-1}\varepsilon^{1/6}.
\end{align}
This completes the proof for S6V, i.e, Lemma~\ref{l.prelimiting variation inequality s6v}. For ASEP, i.e., Lemma~\ref{l.prelimiting variation inequality asep}, a very similar argument holds modulo a few small changes: (i) the analog of \eqref{e.s6v easy inequality fixed point} holds with the inequality reversed since, from \eqref{e.asep unscaled sheet}, $h^{\mrm{ASEP}}(y,0; x,0) = (y-x)\one_{x\leq y}$ instead of $h^{\mrm{S6V}}(y,0; x,0) = (y-x)\one_{x\geq y}$; (ii) the analog of the reversed version of \eqref{e.inequality after monotone evolution} holds for all $x,y\in\Z$ with probability 1 without the $-\varepsilon^{-1/6}$ term since Lemma~\ref{l.asep height monotonicity} guarantees exact height monotonicity; and (iii) \eqref{e.rescaled inequality} holds with the supremum over $y$ taken over $\R$ and for all $x\in\R$ on a probability $1$ event (this uses that there is an extra sign change in going from \eqref{e.inequality after monotone evolution} to \eqref{e.rescaled inequality} by Definition~\ref{d.rescaled ASEP height function}).
\end{proof}

\begin{proof}[Proof of Lemma~\ref{l.joint kpz fixed point convergence}]
We give the proof in the case of S6V as it is the same for ASEP. It suffices to prove that, for every finite $\mc U\subseteq\mc T$,
\begin{equation}\label{e.kpz fixed point fdd to show}
\begin{split}
\MoveEqLeft[16]
\left(\mf h^{\mrm{S6V}, \varepsilon}(\h^{(i), \varepsilon}_0, s_i; \bm\cdot, t), \mc L^{\mrm{S6V},\varepsilon}(\bm\cdot, s_i; \bm\cdot, t) : i\in\intint{1,k}, t\in\mc U\right)\\
&\stackrel{d}{\to}
 \left(\mf h(\mf h^{(i)}_0, s_i; \bm\cdot, t), \mc L(\bm\cdot, s_i; \bm\cdot, t) : i\in\intint{1,k}, t\in\mc U\right)
\end{split}
\end{equation}
in the topology of uniform convergence on compact sets.

First, $(\h^{\mrm{S6V},\varepsilon}(\h^{(i),\varepsilon}_0, s_i; \bm\cdot, t):i\in\intint{1,k}, t\in\mc U)$ is a tight sequence in $\varepsilon$, and for each $i\in\intint{1,k}$ and $t\in\mc U$, we have the marginal convergence of $\h^{\mrm{S6V}, \varepsilon}(\smash{\h^{(i), \varepsilon}_0}, s_i; \bm\cdot, t)$ to $\h(\mf h_0^{(i)}, s_i; \bm\cdot, t)$; the former follows from the latter, which we have assumed in the case $t=1$ and follows in the general case by rescaling. We also know from Theorem~\ref{t.s6v airy sheet} that, in the topology of uniform convergence on compact sets, for any $t\in\mc U$,
\begin{align*}
\mc L^{\mrm{S6V},\varepsilon}(\bm\cdot, s_i; \bm\cdot, t)
&\stackrel{d}{\to}
\mc L(\bm\cdot, s_i; \bm\cdot, t).
\end{align*}

Let $(\tilde \h(\h^{(i)}_0, s_i; \bm\cdot, t), \mc L(\bm\cdot, s_i; \bm\cdot, t) : i\in\intint{1,k}, t\in\mc U)$ be any subsequential limit point of the lefthand side of \eqref{e.kpz fixed point fdd to show}, so that the $\mc L$ terms are coupled as the directed landscape and each $\smash{\tilde \h(\h^{(i)}_0, s_i; \bm\cdot, t)}$ is marginally distributed as the KPZ fixed point $\smash{\h(\h^{(i)}_0, s_i; \bm\cdot, t)}$ for each fixed $i\in\intint{1,k}$ and $t\in\mc U$.
Now, \eqref{e.kpz fixed point fdd to show} is reduced to showing that, for any $i\in\intint{1,k}$ fixed, almost surely, for all $x\in\R^2$,
\begin{align}\label{e.variational to show fixed point}
\tilde \h(\h^{(i)}_0, s_i; x,t) = \sup_{y\in\R}\left(\h^{(i)}_0(y) + \mc L(y,s_i; x,t)\right).
\end{align}
By Lemma~\ref{l.prelimiting variation inequality s6v} and a union bound, with probability at least $1-Ck|\mc U|\exp(-c\varepsilon^{-1/6})$ under the basic coupling, for all $x\in [-\varepsilon^{-1}, \varepsilon^{-1}]$ and $t\in\mc U\cap \{s:s>s_i\}$,
\begin{align*}
\h^{\mrm{S6V}, \varepsilon}(\h^{(i), \varepsilon}_0, s_i; x, t) \geq \sup_{y\in[-\varepsilon^{-1}, \varepsilon^{-1}]}\left(\h_0^{\varepsilon,(i)}(y) + \mc L^{\mrm{S6V},\varepsilon}(y, s_i; x, t)\right).
\end{align*}

In particular, we may also take the supremum over $y\in [-M,M]$. Taking $\varepsilon\to 0$ and then $M\to\infty$, the previous display along with Lemma~\ref{l.argmax control} yields that, almost surely, for all $y\in\R$, $t\in \mc U\cap \{s:s>s_i\}$, and $i\in\intint{1,k}$,
\begin{align}\label{e.subsequential inequality fixed point}
\tilde \h(\h^{(i)}_0, s_i; x,t) \geq \sup_{y\in\R}\left(\h^{(i)}_0(y) + \cL(y,s_i; x,t)\right).
\end{align}
However, we know that both sides of this inequality are marginally distributed as $\h(\h^{(i)}_0, s_i; x, t)$ (for the righthand side, using \eqref{e.fixed point via landscape}) for every fixed $x$.
Thus, Lemma~\ref{l.X=Y} implies that the two sides are equal on a probability 1 event for all $x\in\Q$ and $i\in\intint{1,k}$; since both sides are continuous in $x$, this yields equality for all $x\in\R$, $t\in\mc U\cap \{s:s>s_i\}$, and $i\in\intint{1,k}$, completing the proof for S6V. For ASEP one simply replaces the invocation of Lemma~\ref{l.prelimiting variation inequality s6v} with Lemma~\ref{l.prelimiting variation inequality asep}.
\end{proof}

\subsection{Approximate height monotonicity}\label{s.approximate monotonicity proof}

Here we prove Lemma~\ref{l.s6v approximate height monotonicity}. As mentioned, it is a quick consequence of a form of approximate height monotonicity  for finite systems and a finite speed of discrepancy estimate. We recall precise statements from \cite{aggarwal2024scaling} next, before turning to the proof of Lemma~\ref{l.s6v approximate height monotonicity}.

\begin{lemma}[Approximate height monotonicity of S6V, {\cite[Proposition D.3]{aggarwal2024scaling}}]\label{l.s6v approximate height monotonicity finite system}
Let $q\in[0,1)$ and $b^{\shortrightarrow} \in (0,1)$. Fix $N\in\N$. Let $H\in\Z$ and let $\smash{h_0^{(1)}, h_0^{(2)}}:\intray{-N-1}\to\Z_{\geq 0}$ be two height functions such that $\smash{h_0^{(1)}(x)} +H \geq \smash{h_0^{(2)}(x)}$ for all $x\in\intray{-N-1}$, and, for $i=1,2$, $\smash{h_0^{(i)}}(-N-1)\leq N$ and $\smash{h_0^{(i)}}(x) = 0$ for all large $x$. There exist positive constants $C, c>0$, depending on only $q$ and $b^{\shortrightarrow}$, such that, for $t\in\N$ and under the basic coupling, the following holds. For any $M\geq (\log N)^2$, with probability at least $1-C\exp(-cM)$, $\hssv(h_0^{(1)}; x,t) +H \geq \hssv(h_0^{(2)}; x,t) - M$ for all $x\in\intray{-N-1}$.
\end{lemma}

\begin{lemma}[Finite speed of discrepancy, {\cite[Lemma D.4]{aggarwal2024scaling}}]\label{l.finite speed of discrepancy for s6v basic coupling}
For any real numbers $q\in[0,1)$ and $b^{\shortrightarrow} \in (0,1)$, there exist constants $c>0$ and $C > 1$ such that the following holds. Let $m, N_1\in\N$ be integers, and for $k\in\intint{1,m}$, let $h^{(1), k}_0 : \intray{-N_1} \to \Z$  and $h_0^{(2), k}:\Z\to\Z$ be height functions which equal $0$ for all large $x$. Assume that for $k\in\intint{1,m}$, $h^{(1), k}_0(x) = h^{(2), k}_0(x)$ for all $x\in\intint{-N_1, N_1}$. For $k\in\intint{1,m}$, let $(j^{(1), k}_{(x,y)} : (x,y)\in\intray{-N_1}\times\intray{0})$ and $(j^{(2), k}_{(x,y)} : (x,y)\in\Z\times\intray{0})$ denote samples from S6V models with initial height functions given by $\smash{h^{(i), k}_0}$, all coupled via the basic coupling. Then, for any $T\in\N$ such that $T\leq \frac{1}{4}(1-b^{\shortrightarrow})N_1$, with probability at least $1-Cme^{-cN_1}$, we have that $j^{(1), k}_{(x,y)} = j^{(2), k}_{(x,y)}$ for all $k\in\intint{1,m}$ and $(x,y)\in\intint{1,T}\times\intint{-\floor{\frac{1}{2}N_1}, \floor{\frac{1}{2}N_1}}$.
\end{lemma}

\cite[Lemma D.4]{aggarwal2024scaling} actually assumes that there are only finitely many arrows in the initial condition for each system, but an inspection of the proof shows that the argument applies equally in the case that there are infinitely many arrows.

\begin{proof}[Proof of Lemma~\ref{l.s6v approximate height monotonicity}]
Consider two S6V models with the following pairs of initial conditions. In the first system, let $h^{(1), i} = h^{(i)}$ for $i\in\{1,2\}$, where $h^{(1)}$ and $h^{(2)}$ are as in the statement of Lemma~\ref{l.s6v approximate height monotonicity}. In the second system, let $h^{(2),1}$ and $h^{(2),2}$ be defined by setting, for $i\in\{1,2\}$, $h^{(2),i}(x) = h^{(i)}(x)$ for $x\in\intint{-2N,2N}$, $h^{(2),i}(x) = h^{(2), i}(-2N)$ for $x\leq -2N$, and $h^{(2),i}(x) = h^{(2), i}(2N)$ for $x\geq -2N$. Thus the initial configuration of the second system has finitely many arrows. Coupling all the evolutions across both systems via the basic coupling, we see from Lemma~\ref{l.s6v approximate height monotonicity finite system} that, with probability at least $1-C\exp(-cM)$,
\begin{align*}
h^{\mrm{S6V}}(h_0^{(2), 1};x,t) + H\geq h^{\mrm{S6V}}(h_0^{(2),2}; x,t) - M
\end{align*}
for all $x\in\intint{-2N,2N}$. By Lemma~\ref{l.finite speed of discrepancy for s6v basic coupling}, with probability at least $1-C\exp(-cN)$, $h^{\mrm{S6V}}(h_0^{(1),i}; x,t) = h^{\mrm{S6V}}(h_0^{(2),i}; x,t)$ for $i\in\{1,2\}$ and all $(t,x)\in \intint{1,T}\times \intint{-N,N}$ as long as $T \leq \frac{1}{2}(1-b^{\shortrightarrow})N$. Since $C\exp(-cN) < C\exp(-cM)$, this completes the proof.
\end{proof}

\subsection{Control of maximizer in KPZ fixed point definition}\label{s.fixed point maximizer proof}

\begin{proof}[Proof of Lemma~\ref{l.argmax control}]
By \cite[Lemma 3.8]{IGCI}, the argmax in the probability we are lower bounding is almost surely a non-decreasing function of $y$. Thus it is equal to
\begin{align*}
\P\left(\argmax_{z\in\R}\bigl(\h_0(z) + \cL(z,0;y,t)\bigr) \in [-M,M]  \ \ \forall y\in\{-J,J\}\right).
\end{align*}
We will prove the required lower bound for each fixed $y \in \{-J,J\}$, which suffices. We may assume $M>M_0$ for an $M_0 = M_0(J,\lambda,t)$ that we may set as then the claimed bound is obtained by modifying $C$ appropriately. Recall we have assumed $\h_0$ satisfies \eqref{e.initial condition assumption}. So let $z_0\in[-\lambda,\lambda]$ be such that $\h_0(z_0) \geq -\lambda$, and let $M_0$ be larger than $\lambda$. Now taking the complement of the event in the previous display, a union bound and the fact that $z_0\in[-M,M]$ implies that
\begin{align}
\P\left(\argmax_{z\in\R}\bigl(\h_0(z) + \cL(z,0;y,t)\bigr) \not\in [-M,M]\right) 
&\leq \P\Bigl(\h_0(z_0) + \cL(z_0,0; y, t) \leq -\lambda- M\Bigr) \label{e.argmax split up}\\
&\qquad + \P\left(\sup_{|z|>M} \left(\h_0(z) + \cL(z,0; y,t)\right) \geq -\lambda - M\right).\nonumber
\end{align}

Using that $\h_0(z_0)\geq -\lambda$, the first term of the righthand side of \eqref{e.argmax split up} can be bounded above by 
$$\P\left(\cL(z_0, 0; y, t) \leq -M\right) = \P\left(\cL(0, 0; 0, t) \leq -M + (z_0-y)^2/t\right).$$
By raising $M_0$ depending on $J, \lambda, t$, we may upper bound $-M+(z_0-y)^2/t$ by $-M/2$. Next, $\smash{\cL(0,0;0,t) \stackrel{d}{=} t^{1/3}\cL(0,0;0,1)}$ and $\cL(0,0;0,1)$ has the GUE Tracy-Widom distribution. The latter is well-known to have tails of the form $\exp(-cx^3)$. Together, this yields that the probability in the previous display is upper bounded by $C\exp(-cM^3)$ for some positive constants $C,c$ depending on $t$.

For the second term of \eqref{e.argmax split up}, we use that $\h_0(z) \leq \lambda(1+|z|)$ for all $z$ to upper bound it by
\begin{align*}
\P\left(\sup_{|z|>M} \bigl(\cL(z,0; y,t) + \lambda|z|\bigr) \geq -2\lambda-M\right)
\end{align*}
We recall that $\cL(z,0;y,t) \stackrel{d}{=} \cL(y,0;z,t)$ as processes in $y,z$ (by combining \cite[Proposition 1.23]{dauvergne2021scaling} with temporal stationarity, \cite[Lemma 10.2 (1)]{dauvergne2018directed}), and that $\smash{\cL(y,0;z,t) \stackrel{d}{=} t^{1/3}\cL(0,0; z-y,1)}$ as a process in $z$. Then invoking standard estimates in the literature like \cite[Proposition~8.5]{ganguly2022sharp} bounds the previous display by $C\exp(-cM^{3/2})$ for positive constants $C,c$ depending on $t, \lambda, J$. This completes the proof.
\end{proof}

\subsection{Finite speed of propagation}\label{s.finite speed of prop}

\begin{proof}[Proof of Lemma~\ref{l.finite speed of propagation}]
We give the proof for the case of S6V first. We merge colors to arrive at a system with particles of two colors, $1$ and $2$ (as well as holes). The evolution of the system can be described by allowing the color 2 particles to evolve first and then allowing the color 1 particles to allow; in particular, this means that if a color 1 particle is incident on a vertex that a color 2 particle passes through, its output from that vertex is determined by the trajectory of the color 2 particles.

For the rightmost color 1 particle to cross a distance of at least $\varepsilon^{-2}$ in time $\floor{\varepsilon^{-1}}$, there must be at least one vertical column where it moved vertically through at least $\varepsilon^{-1}$ many vertices. Observe that for any given vertical column this has probability at most $b^{\varepsilon^{-1}}$. Indeed, letting $n$ be the number of locations (if any) that the color 2 particles passed horizontally through that column, we see that the probability we are bounding is at most $(qb^{\shortrightarrow})^{\varepsilon^{-1}-n}\cdot(b^{\shortrightarrow})^{n} \leq (b^{\shortrightarrow})^{\varepsilon^{-1}}$, for each selection of locations where the color 2 particles enter. This completes the proof.

For ASEP, we do a similar merging procedure to obtain a system with particles of colors 1 and 2, along with holes. Consider the trajectory of the rightmost color 1 particle. By suppressing left jump attempts, its position is stochastically upper bounded by that of a particle which was at the same location at time 0 and jumps to the right at rate $1+q$ without seeing any particles, i.e., is never blocked from jumping (obtained by combining the possibility that the color 1 particle jumps to the right, which has rate 1, and the possibility that the site to the right had a color 2 particle which jumps to the left, which has rate $q$). The probability that this particle travels at least $\varepsilon^{-2}$ in time $\varepsilon^{-1}$ is upper bounded by the probability that a sum of $\varepsilon^{-2}$ \iid exponential random variables of rate $1+q$ is smaller than $\varepsilon^{-1}$. Since the mean of the sum is $(1+q)^{-1}\varepsilon^{-2}$, standard concentration inequalities (e.g., \cite[theorem 5.1]{janson2018tail}) yield that this event has probability at most $\exp(-c\varepsilon^{-2})$.
\end{proof}

\subsection{Random walk staying above line}\label{s.random walk proof}

\begin{proof}[Proof of Lemma~\ref{l.RW avoiding line}]
Let $\xi$ be a random variable assigning probability $|\lambda|$ to $-1$ and $1-|\lambda|$ to $0$ and let $g(z) = \log \E[e^{z\xi}] = \log(|\lambda| e^{-z} + (1-|\lambda|))$.  We pick $z_0 = z_0(\rho)>0$ such that $g(-z_0) = -\rho z_0$, whose existence is justified as follows. It is an easy calculation that $f(z):= g(z)-\rho z$ satisfies $f(0)=0$, $f'(0) = \lambda-\rho>0$, $f''(z) >0$ for all $z\in\R$, and $\lim_{z\to-\infty} f'(z) = -1-\rho <0$. Together these facts imply there exists $z_0>0$ such that $f(-z_0) = 0$, as desired. Further, it holds that $z_0 \in (c(\lambda-\rho), C(\lambda-\rho))$ for some $c,C>0$ which may depend on $\lambda$. This follows by Taylor expanding $g$ to second order and noting that the error term, determined by the third derivative of $g$, can be uniformly bounded (with the bound depending on $\lambda$) on a disk around the origin using the analyticity of $g$ on that set.

It is standard that $X(t):=\exp(-z_0S(t)-tg(-z_0)) = \exp(-z_0(S(t)-\rho t))$ is a martingale. Consider the stopping time $T = \min\{t\in\Z_{\geq 0}: S(t) < \rho t -M\}$, with $\min\emptyset = \infty$. Since $S$ takes steps of $0$ or $-1$ only, it follows that $S(T)\in (\rho (T-1)-M-1,\rho T-M)$ when $T<\infty$.  In particular, we note that $X(t\wedge T)$ is a bounded martingale (since $z_0>0$). Also, since $\lambda>\rho$, it is easy to see that $\lim_{t\to\infty} X(t) = 0$ almost surely. Combining these facts with the optional stopping and dominated convergence theorems, we obtain
\begin{align*}
\E[X(T)\one_{T<\infty}] = \lim_{t\to\infty}\E[X(t\wedge T)] = \E[X(0)] = 1.
\end{align*}
By our earlier observation on the value of $S(T)$, on $\{T<\infty\}$, $X(T) \geq \exp(z_0M)$. Combining this with the fact that $z_0\geq c(\lambda-\rho)$, we see that $\P(T<\infty) \leq \exp(-c(\lambda-\rho)M)$, as desired.
\end{proof}

\bibliographystyle{alpha}
\bibliography{gen_initial_data}

\end{document}